\documentclass[12pt]{amsart}

\usepackage{amsfonts,amsthm,amssymb}
\usepackage{latexsym,amsmath,setspace}
\usepackage{verbatim,amscd,graphics}
\usepackage{mathrsfs}
\usepackage[pdftex]{graphicx}
\usepackage[matrix,arrow]{xy}
\usepackage[active]{srcltx}

\numberwithin{equation}{section}

\usepackage[utf8]{inputenc}
\usepackage{amssymb}
\usepackage{amsmath}
\usepackage{graphicx}
\usepackage{fullpage}
\usepackage{subfig}
\usepackage{url}
\usepackage{enumerate}
\usepackage{mathrsfs}  
\usepackage{mathtools} 
\usepackage{pdflscape} 
\usepackage{multirow} 
\usepackage{listings}
\usepackage{bbding}
\usepackage[font=small,labelfont=bf]{caption}
\usepackage{hyperref}
\usepackage{array}
\usepackage{mathtools}
\usepackage{amssymb}
\usepackage{tikz-cd}
\usepackage{slashed}
\usepackage{physics}
\usepackage{xcolor}
\usepackage[english]{babel}
\usepackage{graphicx}
\graphicspath{ {./images/}}
\usepackage{amsmath}
\usepackage{amsfonts}
\makeatletter
\def\amsbb{\use@mathgroup \M@U \symAMSb}
\makeatother

\DeclarePairedDelimiter{\ceil}{\lceil}{\rceil}
\DeclarePairedDelimiter\floor{\lfloor}{\rfloor}
\newcommand\Z{\ensuremath{\mathbb{Z}}} 
\newcommand\N{\ensuremath{\mathbb{N}}} 
\newcommand\Q{\ensuremath{\mathbb{Q}}} 
\newcommand\R{\ensuremath{\mathbb{R}}} 

\newcommand\C{\ensuremath{\mathbb{C}}}

\newcommand{\D}{\mathbb{D}}
\newcommand{\ubar}[1]{\underline{#1}}
\newcommand{\supp}{{\rm supp}}

\newcommand{\M}{\mathcal M}

\newcommand{\CZ}{{\rm CZ}}
\newcommand{\wind}{{\mathrm{wind}}}
\newcommand{\ind}{{\mathrm{ind}}}

\newcommand{\jtil}{{\widetilde{J}}}
\newcommand{\util}{{\widetilde{u}}}
\newcommand{\vtil}{\widetilde{v}}

\newcommand{\loc}{{\rm loc}}
\newcommand{\Ker}{\operatorname{Ker}}

\newtheorem{theorem}{Theorem}[section]
\newtheorem{proposition}[theorem]{Proposition}
\newtheorem{lemma}[theorem]{Lemma}
\newtheorem{corollary}[theorem]{Corollary}

\newtheorem{prop}[theorem]{Proposition}
\newtheorem{lem}[theorem]{Lemma}

\theoremstyle{definition}
\newtheorem{definition}[theorem]{Definition}
\newtheorem{remark}[theorem]{Remark}
\newtheorem{example}[theorem]{Example}

\begin{document}

\title[Proof of Hofer-Wysocki-Zehnder's two or infinity conjecture]{Proof of Hofer-Wysocki-Zehnder's two or infinity conjecture}

\begin{abstract}

We prove that if a Reeb flow on a closed connected three-manifold has more than two simple periodic orbits, then it has infinitely many, as long as the associated contact structure has torsion first Chern class.  
As a special case, we prove a conjecture of Hofer-Wysocki-Zehnder published in 2003 asserting that a smooth and autonomous Hamiltonian flow on $\R^4$ has either two or infinitely many simple periodic orbits on any regular compact connected energy level that is transverse to the radial vector field.  Other corollaries settle some old problems about Finsler metrics: we show that every Finsler metric on $S^2$ has either two or infinitely many prime closed geodesics; and we show that a Finsler metric on $S^2$ with at least one closed geodesic that is not irrationally elliptic must have infinitely many prime closed geodesics.   
The novelty of our work is that we do not make any nondegeneracy hypotheses.   
\end{abstract}

\author{Dan Cristofaro-Gardiner}
\author{Umberto Hryniewicz}
\author{Michael Hutchings}
\author{Hui Liu}

\maketitle

\tableofcontents

\section{Introduction}

\subsection{The main result and corollaries}

Let $Y$ be a closed oriented three-manifold. A {\bf contact form\/} on $Y$ is a $1$-form $\lambda$ such that $\lambda\wedge d\lambda>0$ everywhere. Given a contact form~$\lambda$, the two-plane field $\xi=\Ker(\lambda)\subset TY$ is oriented by $d\lambda$, and this oriented two-plane field is called the {\bf contact structure\/} associated to $\lambda$. In addition, $\lambda$ has an associated {\bf Reeb vector field\/} $R$ defined by the equations
\[
d\lambda(R, \cdot) = 0, \quad \quad \lambda(R) = 1.
\]
The flow of $R$ is called the {\bf Reeb flow\/}.
A periodic orbit of the Reeb flow, which we call a {\bf Reeb orbit\/} for short, is a map $\gamma:\R/T\Z\to Y$ for some $T>0$ such that $\gamma'(t)=R(\gamma(t))$. Here $T$ is the period of $\gamma$. We declare two Reeb orbits to be equivalent if they differ by reparametrizing the domain $\R/T\Z$ by a translation. Any Reeb orbit $\gamma$ is a degree $d$ cover of its image for some $d\ge 1$. We refer to the integer $d$ as the {\bf covering multiplicity\/} of $\gamma$, and we say that $\gamma$ is {\bf simple\/} when $d=1$, i.e.\ the map $\gamma$ is an embedding.

Our main result is the following:

\begin{theorem}
\label{thm:main_intro}
Let $Y$ be a closed connected three-manifold, let $\lambda$ be a contact form on $Y$, and let $\xi=\Ker(\lambda)$ be the associated contact structure. Assume that the first Chern class $c_1(\xi) \in H^2(Y;\mathbb{Z})$ is torsion. Then $\lambda$ has either two or infinitely many simple Reeb orbits.
\end{theorem}

\begin{remark}
\label{rem:nondegenerate}
The conclusion of Theorem~\ref{thm:main_intro} was previously shown in \cite{CGHP} under the additional hypothesis that the contact form $\lambda$ is nondegenerate (see \S\ref{sec:orbitsets} for the definition of ``nondegenerate''). Subsequently, the hypothesis on the first Chern class was removed by Colin-Dehornoy-Rechtman \cite{CDR}, but still assuming nondegeneracy.   The novelty of Theorem~\ref{thm:main_intro} is that we do not require any 
 genericity hypothesis on the Reeb flow.  See \S\ref{sec:history} for more discussion of history and related results.
\end{remark}

The following is an important example of Theorem~\ref{thm:main_intro}. Let $Y\subset\R^4$ be a compact connected smooth hypersurface which is transverse to the radial vector field. Following the symplectic geometry convention, we call such a hypersurface {\bf star-shaped\/}. 
Indeed, $Y$ bounds a compact domain which is star-shaped in the usual sense.
The standard Liouville form on $\R^4$ defined by
\begin{equation}
\label{eqn:slf}
\lambda = \frac{1}{2} \sum_{i=1}^2\left(x_i\,dy_i - y_i\,dx_i\right)
\end{equation}
restricts to a contact form on $Y$. If $H:\R^4\to\R$ is a smooth function having $Y$ as a regular level set, then the Reeb vector field of $\lambda|_Y$ is parallel to the Hamiltonian vector field $X_H$ on~$Y$, so that Reeb orbits are the same as periodic orbits of $X_H$ on~$Y$ up to reparametrization.

Since $Y$ is diffeomorphic to $S^3$, Theorem~\ref{thm:main_intro} is applicable to this example to prove the following result, which was a conjecture of Hofer-Wysocki-Zehnder from \cite{fols}; Hofer-Wysocki-Zehnder proved this in \cite{fols}, under the additional hypothesis that the contact form is nondegenerate and the stable and unstable manifolds of all hyperbolic Reeb orbits intersect transversely.

\begin{corollary}
\label{cor:hwzconj}
\cite[Conjecture 1.13]{fols}
Every star-shaped hypersurface in $\R^4$, with the restriction of the standard Liouville form \eqref{eqn:slf}, has either two or infinitely many simple Reeb orbits.
\end{corollary}

Theorem~\ref{thm:main_intro} also has implications for Finsler geometry:

\begin{corollary}
\label{cor:finsler1}
\cite[Problem 15]{paiva}
\cite[Conjecture 2.2.1]{BurnsMatveev}
\cite[Conjecture 1]{long}
Every Finsler metric on $S^2$ has either two or infinitely many prime closed geodesics.
\end{corollary}

\begin{proof}
Let $F$ be a Finsler metric on $S^2$, and let $Y$ be the unit tangent bundle of $S^2$ associated to $F$.  As reviewed in \cite[\S2.3.3]{HS13}, there exists a contact form $\lambda_F$ on $Y$ such that prime closed geodesics on $S^2$ are in bijective correspondence with simple Reeb orbits on $Y$ under the projection map $Y\to S^2$. Thus the corollary follows from Theorem~\ref{thm:main_intro}, since $H^2(Y;\mathbb{Z})$ is torsion.
\end{proof}

\begin{remark}
It has been a longstanding conjecture
 (\cite[Problem 15]{paiva}, \cite[Conjecture 2.2.1]{BurnsMatveev}, \cite[Conjecture 1]{long}) 
  that every irreversible Finsler metric on $S^2$ has either two or infinitely many prime closed geodesics.   
 To put this result in context, it has long been known that there are infinitely many prime closed geodesics on every closed surface with a Riemannian metric; for positive genus surfaces this is due to Hadamard~\cite{Hadamard}, and the genus zero case was proved by Bangert and Franks \cite{Bangert2,Franks}. The results in the Riemannian case remain true in the Finsler case for positive genus surfaces, and can be proved by similar methods. However, the genus zero case presents a strong contrast, as Katok~\cite{Katok} has constructed Finsler metrics on $S^2$ with exactly two prime closed geodesics. 
  \end{remark}

\begin{remark}
A theorem outlined by Lyusternik-Schnirelmann \cite{ls} and proved by Grayson \cite[Corollary 0.2]{Grayson} asserts that every Riemannian metric on $S^2$ has at least three distinct embedded closed geodesics. Hence Corollary~\ref{cor:finsler1} recovers the result of Bangert and Franks \cite{Bangert2,Franks} asserting that every Riemannian metric on $S^2$ has infinitely many prime closed geodesics.
\end{remark}

Our previous work \cite{CGHHL} established strong restrictions in the case of two simple Reeb orbits\footnote{For examples with exactly two simple Reeb orbits, see \cite[Ex. 1.1]{CGHHL}.}, and so we also obtain as a corollary various criteria guaranteeing the existence of infinitely many simple Reeb orbits: 

\begin{corollary}
\label{cor_one_hyperbolic_implies_infinitely_many}
Let $Y$ be a closed three-manifold and let $\lambda$ be a contact form on $Y$ such that the associated contact structure $\xi=\Ker(\lambda)$ has $c_1(\xi)\in H^2(Y;\Z)$ torsion.  Assume that at least one of the following is true:
\begin{enumerate}[(a)]
\item $Y$ is not a lens space.
\item There is at least one Reeb orbit which is not irrationally elliptic (see \S\ref{sec:orbitsets} for the definition of ``irrationally elliptic'').
\item $\xi$ is overtwisted.
\end{enumerate}
Then there are infinitely many simple Reeb orbits.
\end{corollary}
\begin{proof}
Without loss of generality, $Y$ is connected. By \cite[Thm.\ 1.2, Cor.\ 1.3, Thm.\ 1.5]{CGHHL}, if there are exactly two simple Reeb orbits, then both simple Reeb orbits (and all their iterates) are irrationally elliptic, $Y$ is a lens space, 
and the contact structure is (universally) tight,
contrary to our hypothesis. Hence by the dichotomy in Theorem~\ref{thm:main_intro}, there are infinitely many simple Reeb orbits.
\end{proof}

The above corollary in turn implies the following result, which is slightly stronger than a conjecture of Long~\cite[Conjecture 2.2.2]{BurnsMatveev} about the kinds of geodesics that must appear on a Finsler $S^2$.

\begin{corollary}
\label{cor:long}
A Finsler metric on $S^2$ with a closed geodesic that is not irrationally elliptic has infinitely many prime closed geodesics.
\end{corollary}

\begin{proof}
Apply Corollary~\ref{cor_one_hyperbolic_implies_infinitely_many}(b) to the contact form on the unit tangent bundle of $S^2$ determined by the Finsler metric, as in Corollary~\ref{cor:finsler1}.
\end{proof}

One could alternatively deduce Corollary~\ref{cor:long} from Corollary~\ref{cor:finsler1} as a consequence of \cite[Thm. 1.2]{lw}.  

\subsection{Some history and related results}
\label{sec:history}

The three-dimensional case of the Weinstein conjecture \cite{weinstein} asserts that every contact form on a closed three-manifold has at least one Reeb orbit. This was proved for star-shaped hypersurfaces in $\R^4$ using variational methods by Rabinowitz \cite{Rabinowitz}, who also proved a higher dimensional analogue in $\R^{2n}$.  The Weinstein Conjecture was proved for all contact forms on~$S^3$, and for all overtwisted contact forms on any closed three-manifold, in work of Hofer \cite{Hofer93} where the foundations of holomorphic curve theory in symplectizations were first introduced and developed. The Weinstein conjecture was established for all closed three-manifolds by Taubes~\cite{Taubes_Weinstein_Conjecture} using Seiberg-Witten theory; see \cite{tw} for a survey. Taubes subsequently generalized this work to prove an isomorphism between embedded contact homology (ECH) and Seiberg-Witten Floer cohomology reviewed in \S\ref{sec:defECH} below. This isomorphism, together with a ``Weyl law'' for the ECH spectrum reviewed in \S\ref{sec:volume} below, was used in \cite{CGH} to show that in fact, every contact form on a closed three-manifold has at least two simple Reeb orbits.  It was proved in \cite{GHHM} that the symplectic field theory of Eliashberg-Givental-Hofer \cite{EGH} can be combined with arguments of Hingston \cite{Hingston} and Ginzburg \cite{Ginzburg} to prove that the Reeb flow has at least two simple closed orbits in the special case of star-shaped hypersurfaces in $\R^4$.  Prior to these works, it was shown by Bangert and Long  \cite{BangertLong} that there are at least two prime closed geodesics for any Finsler metric on $S^2$.  A detailed description of the dynamics when there are exactly two simple Reeb orbits was given in \cite{CGHHL}, using a result from \cite{HT} on the nondegenerate case. Such examples are quite special; among other things, the three-manifold must be diffeomorphic to $S^3$ or a lens space.  In fact it was shown by Irie \cite{irie} that for a $C^\infty$ generic contact form on a closed three-manifold $Y$, not only are there infinitely many simple Reeb orbits, but moreover the union of the images of the Reeb orbits is dense in $Y$.

The two or infinity dichotomy in Corollary~\ref{cor:hwzconj} was proved for the special case of convex hypersurfaces by Hofer-Wysocki-Zehnder \cite{convex}. The main ingredient in the proof is to show that in this case there is always a disk-like global surface of section. This perspective is in fact central to our work, as the notion of a global surface of section, or more precisely the slightly more general notion of a Birkhoff section (see \S\ref{sec:gsscriterion} for the definition), plays a key role in our arguments. As mentioned above, in \cite{fols}, Hofer-Wysocki-Zehnder proved the conclusion of Corollary~\ref{cor:hwzconj} for star-shaped hypersurfaces under the additional hypothesis that the contact form is nondegenerate and the stable and unstable manifolds of all hyperbolic Reeb orbits intersect transversely. This work introduced the theory of finite energy foliations, which generalizes Birkhoff sections.

As mentioned in Remark~\ref{rem:nondegenerate}, 
Colin-Dehornoy-Rechtman \cite{CDR} proved that there are two or infinitely many simple Reeb orbits for any nondegenerate contact form on a closed connected three-manifold. This is a consequence of a general existence result for ``broken book decompositions''. The latter are analogues of finite-energy foliations, and were used in \cite{CDR} to establish deep structural results for Reeb flows of nondegenerate contact forms on closed three-manifolds.

\subsection{Idea of the proof and organization of the paper}

The strategy of the proof of Theorem~\ref{thm:main_intro} is as follows. Suppose that there are only finitely many simple Reeb orbits. Following the strategy in \cite{CGHP}, we would like to use holomorphic curves to find a genus zero Birkhoff section for the Reeb flow. Roughly speaking, this is an open surface $\Sigma\subset Y$ such that the simple Reeb orbits correspond to 
ends
of $\Sigma$, together with simple periodic orbits of an area-preserving return map $\varphi:\Sigma\to\Sigma$. If such a genus zero Birkhoff section can be found, then one can apply a theorem of Franks \cite{Franks} to deduce that there are either two or infinitely many simple Reeb orbits.

The difficulty with this approach is to deal with the case when the contact form $\lambda$ is degenerate. Much of the theory of holomorphic curves that we would want to use, including embedded contact homology (ECH), is only defined for nondegenerate contact forms. To get started, we can find a sequence of nondegenerate contact forms $\lambda_k$ converging to $\lambda$.  The idea is then to find
Birkhoff sections for $\lambda_k$, via projections of holomorphic curves $C_k$, that can be taken to converge to a
Birkhoff section for $\lambda$.  A priori, there are two significant challenges with this: i) in general Birkhoff sections need not converge; ii) the arguments in \cite{CGHP} used the assumption of finitely many orbits, and one can not assume this for the $\lambda_k$. 

To attempt to get around the first point, we develop some criteria in terms of the $C_k$ in order to guarantee convergence;  a crucial role is played by the speed of convergence of the ends of the $C_k$.   The difficulty, then, is finding such $C_k$.  Fortuitously, it turns out if the $C_k$ are cylinders of low enough action, detected by the ECH ``$U$-map", then the desired conditions hold as a consequence of the ECH ``partition conditions".  Thus, we return to point ii), with an additional constraint that the Birkhoff section should be an annulus;  we also need a uniform upper bound on the periods of the boundary Reeb orbits.
All of this requires lengthy
arguments involving details of the holomorphic curves arising from ECH, and a new invariant called ``the score" of a curve is introduced.  These arguments appear in Section 3, and the arguments for taking the $C_k$ to the limit appear in Sections 4, 5 and 6; the arguments Section 4 - 6 also require considerable
technical
work to deal with holomorphic curves asymptotic to possibly degenerate Reeb orbits.

In fact, in the end, we do not quite prove that the annular Birkhoff sections for $\lambda_k$ can be taken to converge to an annular Birkhoff section for $\lambda$. This may be true, but we prove a slightly weaker statement which is sufficient for our purposes. Namely, the limits of the boundary Reeb orbits of the Birkhoff sections give a two-component link $\mathcal{L}$ in $Y$ whose components are simple Reeb orbits of $\lambda$. In addition, there is a moduli space $\mathcal{M}$ consisting of holomorphic cylinders with one marked point, and a covering map $\mathcal{M}\to Y\setminus\mathcal{L}$. Finally, the Reeb flow on $Y\setminus\mathcal{L}$ lifts to a flow on $\mathcal{M}$ to which Franks's theorem can be applied to complete the proof of Theorem~\ref{thm:main_intro}.

\subsection{Some open questions and a small step}

It remains an open question whether the 
torsion hypothesis can be removed from Theorem~\ref{thm:main_intro}.  By \cite[Cor. 1.3]{CGHHL}, this is equivalent to the question of whether every contact form on a closed three-manifold that is not a lens space must have infinitely many simple Reeb orbits. 

In the cases where there are countably infinitely many simple Reeb orbits, one can ask for lower bounds on the growth rate of the number of simple Reeb orbits with period $\le T$ as $T\to\infty$. Our methods allow a small step in this direction. Enumerate the simple Reeb orbits as $\gamma_1,\gamma_2,\ldots$ such that the period of $\gamma_i$ is a nondecreasing function of $i$. Then it follows from Proposition~\ref{prop:general} below that for a torsion contact form on a three-manifold other than $S^3$ or a lens space, we obtain an upper bound (possibly very large) on the period of $\gamma_{n+1}$ in terms of data associated to $\gamma_1,\ldots,\gamma_n$.

\subsection*{Acknowledgements}

Key discussions for this project occurred during the conference ``Low-dimensional topology and homeomorphism groups", hosted by the Brin Mathematics Research Center at the University of Maryland.  We thank the Brin Center for their support. 

DCG thanks Dan Pomerleano for numerous very helpful discussions related to the present work; our earlier joint paper \cite{CGHP} also contained many insights which were crucial for this one.  DCG also thanks the NSF for their support under agreements DMS-2227372 and DMS-2238091.

UH thanks Pedro Salom\~ao for various discussions about obtaining global surfaces of section for degenerate Reeb flows, which were crucial to the current paper. UH acknowledges the support by DFG SFB/TRR 191 ‘Symplectic Structures in Geometry, Algebra and Dynamics’, Projektnummer 281071066-TRR~191.

MH was partially supported by NSF grant DMS-2005437.

HL was partially supported by NSFC (Nos. 12371195, 12022111) and the Fundamental Research Funds for the Central Universities (No. 2042023kf0207)

\section{Preliminaries on embedded contact homology}
\label{prelim_sec}

We now review the background on embedded contact homology (ECH) that we will need. More details about the basic definitions may be found for example in \cite{ECHlecture}.

\subsection{Orbit sets and the ECH index}
\label{sec:orbitsets}

Let $Y$ be a closed $3$-manifold, and $\lambda$ be a contact form on $Y$. Let $\xi=\Ker \lambda$ denote the associated contact structure, and let $R$ denote the Reeb vector field on $Y$ determined by $\lambda$. If $\gamma:\R/T\Z\to Y$ is a Reeb orbit, then the derivative of the Reeb flow at time $T$ defines a linearized return map
\[
P_\gamma: \xi_{\gamma(0)} \longrightarrow \xi_{\gamma(0)},
\]
which is symplectic with respect to $d\lambda$. We say that the Reeb orbit $\gamma$ is {\bf nondegenerate\/} if~$1$ is not an eigenvalue of $P_\gamma$. 
We say that the contact form $\lambda$ is {\bf nondegenerate\/} if all Reeb orbits are nondegenerate.
Let $\gamma$ be a Reeb orbit.
If both eigenvalues of $P_\gamma$ lie on the unit circle we say that $\gamma$ is {\bf elliptic\/}, otherwise we say that $\gamma$ is {\bf hyperbolic\/}.
In the hyperbolic case both eigenvalues of $P_\gamma$ are real and lie outside the unit circle, and they are both positive or both negative; in the former case $\gamma$ is said to be {\bf positive hyperbolic\/}, in the latter case~$\gamma$ is said to be {\bf negative hyperbolic\/}.
In the elliptic case, we call $\gamma$ {\bf irrationally elliptic\/} if the eigenvalues of $P_\gamma$ are of the form $e^{\pm i2\pi\alpha}$ for some irrational number $\alpha \in \R \setminus\Q$.
When $\gamma$ and all its iterates are nondegenerate then there are three mutually exclusive cases: positive hyperbolic, negative hyperbolic or irrationally elliptic.

 If $\tau$ is a $d\lambda$-symplectic trivialization of $\gamma^*\xi \to \R/T\Z$, then the linearized Reeb flow along $\gamma$ in the trivialization $\tau$ has a well-defined rotation number $\theta\in\R$. If $\gamma$ is nondegenerate, then the rotation number $\theta$ is an integer when $\gamma$ is positive hyperbolic, $1/2$ plus an integer when $\theta$ is negative hyperbolic, and irrational when $\gamma$ is elliptic and all iterates of~$\gamma$ are nondegenerate. If $\gamma$ is nondegenerate\footnote{If $\gamma$ is degenerate, there are different conventions for the Conley-Zehnder index. It would be possible to use \eqref{eqn:CZtau} to define the Conley-Zehnder index in the degenerate case and we used this convention in \cite{CGHHL}. However in \S\ref{sec_prelim_2} we will use a different convention which in the degenerate case is sometimes one less than \eqref{eqn:CZtau}.}, the {\bf Conley-Zehnder index} of $\gamma$ with respect to $\tau$ is given by
\begin{equation}
\label{eqn:CZtau}
\CZ_\tau(\gamma) = \floor{\theta} + \ceil{\theta}.
\end{equation} 
For the rest of the discussion of ECH we assume unless otherwise stated that $\lambda$ is nondegenerate; this holds for $C^\infty$-generic contact forms.

\begin{definition}
An {\bf orbit set\/} is a finite set of pairs $\alpha=\{(\alpha_i,m_i)\}$ where the $\alpha_i$ are distinct simple Reeb orbits and the $m_i$ are positive integers. The homology class of the orbit set $\alpha$ is defined to be
\[
[\alpha] = \sum_im_i[\alpha_i] \in H_1(Y)
\]
and its {\bf symplectic action\/} of $\alpha$ is defined by
\[
\mathcal{A}(\alpha) = \sum_i m_i\int_{\alpha_i}\lambda.
\]
Note here that $\int_{\alpha_i}\lambda$ is simply the period of $\alpha_i$.
An {\bf ECH generator\/} is an orbit set $\alpha=\{(\alpha_i,m_i)\}$ such that $m_i=1$ whenever $\alpha_i$ is hyperbolic.
\end{definition}

If $\alpha=\{(\alpha_i,m_i)\}$ and $\beta=\{(\beta_j,n_j)\}$ are orbit sets with $[\alpha]=[\beta]\in H_1(Y)$, define $H_2(Y,\alpha,\beta)$ to be the set of relative homology classes of $2$-chains $Z$ in $Y$ with
\[
\partial Z = \sum_i m_i\alpha_i - \sum_j n_j\beta_j,
\]
modulo boundaries of $3$-chains. The set $H_2(Y,\alpha,\beta)$ is an affine space over $H_2(Y)$.

\begin{definition}
Given $Z\in H_2(Y,\alpha,\beta)$ as above, the {\bf ECH index\/} is defined by
\[
I(\alpha,\beta,Z) = c_\tau(Z) + Q_\tau(Z) + \sum_i\sum_{k=1}^{m_i}\CZ_\tau(\alpha_i^k) - \sum_j\sum_{k=1}^{n_j}\CZ_\tau(\beta_j^k) \in \Z.
\]
\end{definition}

Here $\tau$ is a trivialization of the contact structure $\xi$ over the Reeb orbits $\alpha_i$ and $\beta_i$; the ECH index does not depend on the choice of $\tau$, although the individual terms in it do. The term $c_\tau(Z)$ is the relative first Chern class $c_1(f^*\xi,\tau)$, where $f:\Sigma\to Y$ is a smooth map from a compact oriented surface $\Sigma$ with boundary mapping to $\alpha-\beta$, representing the relative homology class $Z$. This relative first Chern class is defined by choosing a generic section of $f^*\xi$ over $\Sigma$ which on each boundary component is nonvanishing and has winding number zero with respect to $\tau$, and then counting its zeroes with signs. The term $Q_\tau(Z)$ is the relative intersection pairing; the detailed definition is not needed here and may be found in \cite[\S3.3]{ECHlecture} or \cite[\S2.7]{ir}.

For orbit sets $\alpha$ and $\beta$ with $[\alpha]=[\beta]=\Gamma\in H_1(Y)$, and for a trivialization $\tau$ over the Reeb orbits in $\alpha$ and $\beta$, the above notions depend on the relative homology class $Z$ as follows: Given another relative homology class $W\in H_2(Y,\alpha,\beta)$, one can form the difference $Z-W\in H_2(Y)$, and we have
\begin{align}
\label{eqn:ctauambiguity}
c_\tau(Z) - c_\tau(W) &= \langle Z-W,c_1(\xi)\rangle,\\
\nonumber
Q_\tau(Z)-Q_\tau(W) &= \langle Z - W, 2\operatorname{PD}(\Gamma)\rangle.
\end{align}
Here $c_1(\xi)\in H^2(Y;\Z)$ denotes the usual first Chern class of $\xi$, and $\operatorname{PD}(\Gamma)\in H^2(Y;\Z)$ is the Poincar\'e dual of $\Gamma$. It follows that we have the ``index ambiguity formula''
\begin{equation}
\label{eqn:iaf}
I(\alpha,\beta,Z) - I(\alpha,\beta,W) = \langle Z-W, c_1(\xi) + 2\operatorname{PD}(\Gamma)\rangle.
\end{equation}

\subsection{Holomorphic currents and the index inequality}
\label{sec:hcie}

\begin{definition}
An almost complex structure $J$ on $\R\times Y$ is {\bf $\lambda$-compatible\/} if:
\begin{itemize}
\item
$J(\partial_a)=R$, where $a$ denotes the $\R$ coordinate.
\item
$J$ is invariant under translation of the $\R$ factor.
\item
$J$ sends the contact structure $\xi$ to itself, rotating positively with respect to $d\lambda$.
 \end{itemize}
\end{definition}

Given a $\lambda$-compatible almost complex structure $J$, we can consider a $J$-holomorphic map
\[
u:(\Sigma,j) \longrightarrow (\R\times Y,J)
\]
where the domain $\Sigma$ is a compact connected Riemann surface without boundary, with finitely many punctures, and for each puncture there is a neighborhood which maps asymptotically to a Reeb orbit with $a\to\pm\infty$; for more precise definitions, see \S\ref{sec_prelim_2}. We declare two such holomorphic maps to be equivalent if they agree up to biholomorphic reparametrization of the domain. Equivalence classes are called $J$-holomorphic curves. If the map $u$ is somewhere injective, then we can identify the associated curve with its image in $\R\times Y$, which we typically denote using the letter $C$.

\begin{example}
If $\gamma$ is any simple Reeb orbit, then the ``trivial cylinder'' $\R \times \gamma \subset \R \times Y$ is $J$-holomorphic, by the definition of ``$\lambda$-compatible''.
\end{example}

We define a {\bf $J$-holomorphic current\/} to be a finite linear combination
\[
\mathcal{C} = \sum_i d_i C_i
\]
where the $C_i$ are distinct somewhere injective $J$-holomorphic curves as above, and the $d_i$ are positive integers. Note that we use the calligraphic letter $\mathcal{C}$ to denote a current, and the ordinary letter $C$ to denote a somewhere injective curve.

For a $J$-holomorphic current $\mathcal{C}$ as above, there are orbit sets $\alpha$ and $\beta$ such that as currents, we have $\lim_{a\to\infty}\mathcal{C}\cap(\{a\}\times Y) = \alpha$ and $\lim_{a\to-\infty}\mathcal{C}\cap(\{a\}\times Y)=\beta$. We call $\alpha$ the {\bf positive orbit set\/} of $\mathcal{C}$, and we call $\beta$ the {\bf negative orbit set\/} of $\mathcal{C}$. The projection of the current $\mathcal{C}$ to $Y$ defines a relative homology class
\[
[\mathcal{C}]\in H_2(Y,\alpha,\beta).
\]
We define the ECH index of $\mathcal{C}$ by
\begin{equation}
\label{eqn:ECHC}
I(\mathcal{C}) = I(\alpha,\beta,[\mathcal{C}])\in\Z.
\end{equation}
We let $\mathcal{M}^J(\alpha,\beta)$ denote the moduli space of $J$-holomorphic currents with positive orbit set $\alpha$ and negative orbit set $\beta$. Given $Z\in H_2(Y,\alpha,\beta)$, we define
\[
\mathcal{M}^J(\alpha,\beta,Z) = \left\{\mathcal{C}\in\mathcal{M}^J(\alpha,\beta) \;\big|\; [\mathcal{C}]=Z \right\}.
\]

The starting point for the foundations of ECH is the following ``index inequality'', originally from \cite{HutJEMS}, see also \cite[\S3.4]{ECHlecture}:

\begin{proposition}
\label{prop:indexinequality}
If $C\in\mathcal{M}^J(\alpha,\beta)$ is somewhere injective, then
\begin{equation}
\label{eqn:indexinequality}
\operatorname{ind}(C) \le I(C),
\end{equation}
with equality only if $C$ is embedded in $\R\times Y$.
\end{proposition}

Here $\operatorname{ind}(C)$ denotes the {\bf Fredholm index\/} of $C$, which equals the dimension of the moduli space of $J$-holomorphic curves near $C$ when $J$ is generic; see e.g.\ \cite[\S3.2]{ECHlecture} for details.

Using the index inequality, one can prove the following key result:

\begin{proposition}
\label{prop:lowindexcurves}
\cite[Prop.\ 3.7]{ECHlecture}
Suppose $J$ is a generic $\lambda$-compatible almost complex structure. Let $\alpha$ and $\beta$ be orbit sets and let $\mathcal{C}\in\mathcal{M}^J(\alpha,\beta)$. Then:
\begin{itemize}
\item[(0)] $I(\mathcal{C})\ge 0$, with equality if and only if $\mathcal{C}$ is a sum of trivial cylinders with multiplicities.
\item[(1)]
If $I(\mathcal{C})=1$, then $\mathcal{C} = \mathcal{C}_0+C_1$, where $I(\mathcal{C}_0)=0$, and $C_1$ is embedded and disjoint from $\mathcal{C}_0$ and satisfies $\operatorname{ind}(C_1)=I(C_1)=1$.
\item[(2)]
If $I(\mathcal{C})=2$, and if $\alpha$ and $\beta$ are ECH generators, then $\mathcal{C}=\mathcal{C}_0+C_1$, where $I(\mathcal{C}_0)=0$, and $C_1$ is embedded and disjoint from $\mathcal{C}_0$ and satisfies $\operatorname{ind}(C_1)=I(C_1)=2$.
\end{itemize}
\end{proposition}

\subsection{The definition of ECH}
\label{sec:defECH}

Let $\Gamma\in H_1(Y)$, and let $\lambda$ be a nondegenerate contact form on~$Y$. The {\bf embedded contact homology\/} $ECH(Y,\lambda,\Gamma)$ is a vector space over $\Z/2$ defined as follows. 
It is also possible to define ECH with integer coefficients, as explained in \cite[\S9]{HTJSG2}, but we will not need this.

\begin{definition}
We define the chain complex $ECC(Y,\lambda,\Gamma)$ to be the $\Z/2$ vector space generated by ECH generators $\alpha$ with $[\alpha]=\Gamma$. Next, choose a generic $\lambda$-compatible almost complex structure $J$. We define a differential
\[
\partial_J:ECC(Y,\lambda,\Gamma) \longrightarrow ECC(Y,\lambda,\Gamma)
\]
as follows. If $\alpha$ is an ECH generator with $[\alpha]=\Gamma$, then
\[
\partial_J\alpha = \sum_\beta\langle\partial_J\alpha,\beta\rangle\beta
\]
where
\[
\langle\partial_J\alpha,\beta\rangle = \#\{\mathcal{C}\in\mathcal{M}^J(\alpha,\beta)/\R \mid I(\mathcal{C})=1\}.
\]
Here the sum is over ECH generators $\beta$ with $[\beta]=\Gamma$, and $\#$ denotes the mod 2 count.
\end{definition}

A compactness argument in \cite[\S5.3]{ECHlecture} shows that $\partial_J$ is well defined. It is shown in \cite{HTJSG1,HTJSG2} that $\partial_J^2=0$. We then define $ECH(Y,\lambda,\Gamma)$ to be the homology of the chain complex $(ECC(Y,\lambda,\Gamma),\partial_J)$.

It follows from the index ambiguity formula \eqref{eqn:iaf} that the ECH index defines a relative $\Z/d$ grading on $ECH(Y,\lambda,\Gamma)$, where $d$ is the divisibility of $c_1(\xi)+2\operatorname{PD}$ in $H^2(Y;\Z)$ mod torsion, which is always an even integer. Here the index difference between two ECH generators $\alpha$ and $\beta$ with $[\alpha]=[\beta]=\Gamma$ is defined to be $I(\alpha,\beta,Z)\mod d$ for any $Z\in H_2(Y,\alpha,\beta)$. We say that a class in $ECH(Y,\lambda,\Gamma)$ is {\bf homogeneous\/} if it can be represented by a sum of ECH generators each of which have the same relative $\Z/d$-grading. In the special case $\Gamma=0$, there is a canonical $\Z/d$ grading, where the grading of an ECH generator $\alpha$ with $[\alpha]=0$ is defined to be
\begin{equation}
\label{eqn:canonicalgrading}
I(\alpha)=I(\alpha,\emptyset,Z)
\end{equation}
for any $Z\in H_2(Y,\alpha,\emptyset)$. By \cite[Prop.\ 1.6(c)]{HutJEMS}, there is also a canonical $\Z/2$ grading on $ECH(Y,\lambda,\Gamma)$ for all $\Gamma$; an ECH generator $\alpha=\{(\alpha_i,m_i)\}$ has even grading if and only if an even number of the Reeb orbits $\alpha_i$ are positive hyperbolic.

Taubes \cite{Taubes-I} proved that if $Y$ is connected, then there is a canonical isomorphism of relatively $\Z/d$-graded $\Z/2$-vector spaces
\begin{equation}
\label{eqn:Taubesisomorphism}
ECH_*(Y,\lambda,\Gamma) \simeq \widehat{HM}^{-*}(Y,\mathfrak{s}_\xi + \operatorname{PD}(\Gamma);\Z/2).
\end{equation}
Here the right hand side is a version of Seiberg-Witten Floer cohomology, as defined by Kronheimer-Mrowka \cite{KMbook}, with $\Z/2$-coefficients. Also $\mathfrak{s}_\xi$ denotes the spin-c structure on $Y$ determined by the oriented $2$-plane field $\xi$; the spin-c structure $\mathfrak{s}_\xi+\operatorname{PD}(\Gamma)$ has first Chern class $c_1(\xi)+2\operatorname{PD}(\Gamma)$. In particular, it follows from the canonical isomorphism \eqref{eqn:Taubesisomorphism} that $ECH(Y,\lambda,\Gamma)$ does not depend on the choice of generic $\lambda$-compatible almost complex structure $J$, and so we can denote it by $ECH(Y,\xi,\Gamma)$.

\subsection{The $U$ map}
\label{sec:Umap}

Suppose now that the three-manifold $Y$ is connected. There is then a ``$U$-map''
\begin{equation}
\label{eqn:Umap}
U: ECH_*(Y,\xi,\Gamma) \longrightarrow ECH_{*-2}(Y,\xi,\Gamma)
\end{equation}
which will play a key role in our arguments.

\begin{definition}
Let $y$ be a point on $Y$ which is not on the image of any Reeb orbit, and let $J$ be a generic $\lambda$-compatible almost complex structure. We then define a map
\[
U_{J,y}: ECC_*(Y,\lambda,\Gamma) \longrightarrow ECC_{*-2}(Y,\lambda,\Gamma)
\]
as follows. If $\alpha$ is an ECH generator with $[\alpha]=\Gamma$, then
\[
U_{J,y}\alpha = \sum_\beta\langle U_{J,y}\alpha,\beta\rangle\beta
\]
where
\begin{equation}
\label{eqn:Uchainmap}
\langle U_{J,y}\alpha,\beta\rangle = \#\left\{\mathcal{C}\in\mathcal{M}^J(\alpha,\beta) \mid I(\mathcal{C})=2, \; (0,y)\in\mathcal{C}\right\}.
\end{equation}
\end{definition}

As explained in \cite[\S2.5]{HT}, $U_{J,y}$ is a chain map:
\[
\partial_J\circ U_{J,y} = U_{J,y}\circ\partial_J.
\]
Consequently $U_{J,y}$ induces a map on $ECH(Y,\xi,\Gamma)$. As shown in \cite[\S2.5]{HT}, this map does not depend on the choice of point $y$ (here one uses the assumption that $Y$ is connected), so we can denote it by $U_J$. Taubes \cite{Taubes-V} further showed that under the isomorphism \eqref{eqn:Taubesisomorphism}, the map $U_J$ agrees with a ``U-map'' on Seiberg-Witten Floer cohomology, so in particular we have a well-defined map as in \eqref{eqn:Umap}.

We will need the following fact about the nontriviality of the $U$-map, which is a consequence of \cite[Lem.\ A.1(a)]{CGHP}.

\begin{definition}
A {\bf $U$-sequence\/} in $ECH(Y,\xi,\Gamma)$ is a sequence of nonzero classes $\{\sigma_k\}_{k\ge k_0}$ in $ECH(Y,\xi,\Gamma)$ such that $U\sigma_{k+1}=\sigma_k$ for each $k\ge k_0$.
\end{definition}

\begin{proposition}
\label{prop:usequence}
If $Y$ is connected and $c_1(\xi)+2\operatorname{PD}(\Gamma)\in H^2(Y;\Z)$ is torsion, then there exists a $U$-sequence $\{\sigma_k\}_{k\ge k_0}$ in $ECH(Y,\xi,\Gamma)$ such that each class $\sigma_k$ is homogeneous and has even grading.
\end{proposition}

\subsection{Spectral invariants and the volume property}
\label{sec:volume}

It follows from the definition of ``$\lambda$-compatible'' almost complex structure that the differential $\partial_J$, and the U-map $U_{J,y}$, decrease the symplectic action: If $\alpha$ and $\beta$ are ECH generators with $\langle\partial_J\alpha,\beta\rangle$ or $\langle U_{J,y}\alpha,\beta\rangle$ nonzero, then $\mathcal{A}(\alpha)>\mathcal{A}(\beta)$. In particular, for each real number $L$, we can define the {\bf filtered embedded contact homology\/}, denoted by $ECH^L(Y,\lambda,\Gamma)$, to be the homology of the subcomplex of $(ECC(Y,\lambda,\Gamma),\partial_J)$ generated by ECH generators with symplectic action less than $L$. There is also an inclusion-induced map
\begin{equation}
\label{eqn:iim}
ECH^L(Y,\lambda,\Gamma) \longrightarrow ECH(Y,\xi,\Gamma).
\end{equation}
It is shown in \cite[Thm.\ 1.3]{cc2} that the filtered ECH and the inclusion-induced map \eqref{eqn:iim} do not depend on the choice of generic $\lambda$-compatible almost complex structure $J$. However the filtered ECH, unlike the regular ECH, does depend on the contact form $\lambda$ and not just on the contact structure $\xi$, as detected for example by the following spectral invariants.

\begin{definition}
Given a nonzero class $\sigma\in ECH(Y,\xi,\Gamma)$, we define the {\bf spectral invariant\/}
\[
c_\sigma(Y,\lambda)\in\R
\]
to be the infimum over $L$ such that $\sigma$ is in the image of the inclusion-induced map \eqref{eqn:iim}.
\end{definition}

Equivalently, if one makes a choice of generic $\lambda$-compatible $J$, then $c_\sigma(Y,\lambda)$ is the minimum $L$ such that the class $\sigma$ can be represented by a cycle in the chain complex $(ECC(Y,\lambda,\Gamma),\partial_J)$ which is a sum of ECH generators each having symplectic action less than or equal to $L$. As shown in \cite{qech}, each spectral invariant $c_\sigma$ is a $C^0$-continuous function of the contact form $\lambda$, and as such extends by continuity to degenerate contact forms.

The following ``Weyl law'' for the ECH spectrum will play a key role in the proof of the main theorem. This is a special case of \cite[Thm.\ 1.3]{ECHasymptotics}.

\begin{proposition}
\label{prop:vol}
Suppose $Y$ is connected and $c_1(\xi)+2\operatorname{PD}(\Gamma)\in H^2(Y;\Z)$ is torsion. Let $\{\sigma_k\}_{k\ge k_0}$ be a $U$-sequence in $ECH(Y,\xi,\Gamma)$. Then
\[
\lim_{k \to \infty} \frac{ c_{\sigma_k}(Y,\lambda)^2}{k} = 2 \int_Y \lambda \wedge d \lambda.\]
\end{proposition}

\subsection{Partition conditions}
\label{sec:part}

Let $\alpha=\{(\alpha_i,m_i)\}$ and $\beta=\{(\beta_j,n_j)\}$ be orbit sets, and suppose that $C\in\mathcal{M}^J(\alpha,\beta)$ is a somewhere injective curve. Then $C$ has positive ends at covers of~$\alpha_i$ with total multiplicity~$m_i$; the multiplicities of these covers determine a partition of~$m_i$, which we denote by $p_{\alpha_i}^+(C)$. Likewise, the multiplicities of the covers of $\beta_j$ at which $C$ has negative ends determine a partition of $n_j$, which we denote by $p_{\beta_j}^-(C)$. It turns out that if equality holds in the index inequality \eqref{eqn:indexinequality}, then these partitions are uniquely determined as follows.

\begin{definition}
If $\theta$ is a real number and $m$ is a positive integer, we define two partitions of $m$, the {\bf positive partition\/} $p_\theta^+(m)$ and the {\bf negative partition\/} $p_\theta^-(m)$, as follows. We define the positive partition $p^+_\theta(m)$ to be the set of horizontal displacements of the maximal polygonal path that starts at $(0,0)$, ends at $(m, \lfloor m \theta \rfloor)$, is the graph of a piecewise linear concave function with vertices at integer lattice points, and does not go above the line $y = m \theta$. 
Here ``horizontal displacement'' means the horizontal displacement between two consecutive integer lattice points on the path. Similarly, we define the negative partition $p^-_\theta(m)$ to be the set of horizontal displacements of the edges of the minimal polygonal path that starts at $(0,0)$, ends at $(m,\ceil{m\theta})$, is the graph of a piecewise linear convex function with vertices at integer lattice points, and does not go below the line $y=m\theta$.
\end{definition}

Note that $p_\theta^-(m) = p_{-\theta}^+(m)$. Also note that for a given $m$, the partitions $p_\theta^+(m)$ and $p_\theta^-(m)$ depend only on the equivalence class of $\theta$ in $\R/\Z$. The following is a basic example:

\begin{example}
\label{example:smallangle}
If $0 \le \theta < 1/m$, then $p_\theta^+(m)=(1,\ldots,1)$. This is because the maximal polygonal path described above is the horizontal line from $(0,0)$ to $(m,0)$.
\end{example}

We will also need the following generalization:

\begin{example}
\label{example:nearlyrationalangle}
If $\theta$ is a rational number $a/b$ in lowest terms, and if $m$ is a multiple of $b$, then for $\epsilon\ge 0$ sufficiently small with respect to $m$ and $\theta$, we have $p_{a/b+\varepsilon}^+(m)=(m/b,\ldots,m/b)$. This is because the maximal polygonal path described above is the line from $(0,0)$ to $(m,ma/b)$.
\end{example} 

If $\gamma$ is a simple Reeb orbit, then the linearized Reeb flow has a well-defined rotation number $\theta\in\R/\Z$ independent of a choice of trivialization. Given a positive integer $m$, we define
\[
p_\gamma^+(m) = p_\theta^+(m), \quad\quad\quad p_\gamma^-(m)=p_\theta^-(m).
\]

Now let $J$ be a generic $\lambda$-compatible almost complex structure, suppose that $\mathcal{C}\in \mathcal{M}^J(\alpha,\beta)$ has ECH index $1$ or $2$, and in the latter case suppose that $\alpha$ and $\beta$ are ECH generators. Write $\mathcal{C}=\mathcal{C}_0+C_1$ as in Proposition~\ref{prop:lowindexcurves}. For each $i$, let $m_i'$ denote the total covering multiplicity of the covers of $\alpha_i$ at which $C_1$ has positive ends. Thus $m_i'\le m_i$, and $m_i-m_i'$ is the multiplicity of the trivial cylinder $\R\times\alpha_i$ in $\mathcal{C}_0$. Likewise, let $n_j'$ denote the total covering multiplicity of the covers of $\beta_j$ at which $C_1$ has negative ends.

\begin{proposition}
\label{prop:partitionconditions}
Let $J$ be generic, let $\mathcal{C}\in\mathcal{M}^J(\alpha,\beta)$ be a $J$-holomorphic current with ECH index $1$ or $2$, assume that $\alpha$ and $\beta$ are ECH generators if $I(\mathcal{C})=2$, and write $\mathcal{C}=\mathcal{C}_0+C_1$ as in Proposition~\ref{prop:lowindexcurves}. Then:
\begin{itemize}
\item
For each $i$ with $m_i'>0$, and for each $n$ with $0\le n\le m_i-m_i'$, we have
\begin{equation}
\label{eqn:positivepartitions}
p_{\alpha_i}^+(m_i'+n) = p_{\alpha_i}^+(C_1) \cup p_{\alpha_i}^+(n).
\end{equation}
\item
Similarly, for each $j$ with $n_j'>0$, and for each $n$ with $0\le n \le n_j-n_j'$, we have
\begin{equation}
\label{eqn:negativepartitions}
p_{\beta_j}^-(n_j'+n) = p_{\beta_j}^-(C_1) \cup p_{\beta_j}^-(n).
\end{equation}
\end{itemize}
Here $\cup$ denotes the union of multisets, and we interpret $p_\gamma^\pm(0)=\emptyset$.
\end{proposition}

In particular, \eqref{eqn:positivepartitions} implies that if $C_1$ has positive ends at covers of $\alpha_i$, then the multiplicities of these covers are determined by $p_{\alpha_i}^+(C_1)=p_{\alpha_i}^+(m_i')$.

\begin{remark}
\label{rem:7.28}
In the case when $\mathcal{C}_0$ includes the trivial cylinder $\R\times\alpha_i$ and $C_1$ has a positive end asymptotic to a cover of $\alpha_i$, so that $m_i>m_i'>0$, it follows from \eqref{eqn:positivepartitions} with $n=m_i-m_i'$ that if $1\in p_{\alpha_i}^+(m_i-m_i')$, then $1\in p_{\alpha_i}^+(m_i)$. The converse is also true; this follows from \cite[Lem.\ 7.28]{HTJSG1}. Likewise, if $n_j>n_j'>0$, then $1\in p_{\beta_j}^-(n_j-n_j')$ if and only if $1\in p_{\beta_j}^-(n_j)$.
\end{remark}

Here are some additional facts about the partitions that we will need.

\begin{lemma}
\label{lem:partitionfacts}
Let $\theta\in\R$ be irrational, and let $m$ be a positive integer. Then:
\begin{itemize}
\item[(a)]
$1\in p_\theta^\pm(m)\Longleftrightarrow 1\notin p_\theta^\mp(m)$.
\item[(b)]
If $m>1$, then the partitions $p_\theta^+(m)$ and $p_\theta^-(m)$ are disjoint.
\item[(c)]
If $|p_\theta^+(m)|+|p_\theta^-(m)|\le 3$, then $m\{\theta\}<2$ or $m(1-\{\theta\})<2$, where $\{\theta\}\in[0,1)$ denotes the fractional part of $\theta$.
\end{itemize}
\end{lemma}

\begin{proof}
Statements (a) and (b) are proved in \cite[Rem.\ 4.4]{HutJEMS}; see also \cite[Ex.\ 3.13(c)]{ECHlecture} for a stronger version of (b). Statement (c) is shown in the proof of \cite[Lem.\ 4.9]{CGHP}.
\end{proof}

\subsection{The $J_0$ index and topological complexity bounds}
\label{sec:J0index}

Recall from the index inequality \eqref{eqn:indexinequality} that the ECH index bounds the Fredholm index of a somewhere injective curve. There is a variant of the ECH index, called the {\bf $J_0$-index}, which instead bounds the topological complexity of the curve.

\begin{definition}
Let  $\alpha=\{(\alpha_i,m_i)\}$ and $\beta=\{(\beta_j,n_j)\}$ be orbit sets with $[\alpha]=[\beta]$, and let $Z\in H_2(Y,\alpha,\beta)$. We define
\begin{equation}
\label{eqn:defJ0}
J_0(\alpha,\beta,Z) = I(\alpha,\beta,Z) - \left( 2c_\tau(Z) + \CZ_\tau^{\operatorname{top}}(\alpha) - \CZ_\tau^{\operatorname{top}}(\beta)\right).
\end{equation}
Here $\CZ_\tau^{\operatorname{top}}(\alpha)$ adds up the Conley-Zehnder indices of the ``top'' iterates of the orbits $\alpha_i$, namely
\begin{equation}
\label{eqn:CZtop}
\CZ_\tau^{\operatorname{top}}(\alpha) = \sum_i\CZ_\tau(\alpha_i^{m_i}).
\end{equation}
\end{definition}

If $\mathcal{C}\in\mathcal{M}^J(\alpha,\beta)$, then similarly to \eqref{eqn:ECHC}, we write $J_0(\mathcal{C})=J_0(\alpha,\beta,[\mathcal{C}])$.

In the situation relevant to us, $J_0$ bounds topological complexity as follows. Suppose that the almost complex structure $J$ is generic, let $\alpha$ and $\beta$ be ECH generators, suppose $\mathcal{C}\in\mathcal{M}^J(\alpha,\beta)$ has ECH index $2$, and write $\mathcal{C}=\mathcal{C}_0+C_1$ as in Proposition~\ref{prop:lowindexcurves}. For each simple Reeb orbit $\gamma$, define $e_\gamma^+(\mathcal{C})$ to be twice the number of positive ends of $C_1$ at covers of $\gamma$, minus $1$ if $C_1$ has positive ends at covers of $\gamma$ but $\mathcal{C}_0$ does not include the trivial cylinder $\R\times\gamma$. Likewise, define $e_\gamma^-(\mathcal{C})$ to be the twice the number of negative ends of $C_1$ at covers of $\gamma$, minus $1$ if $C_1$ has negative ends at covers of $\gamma$ but $\mathcal{C}_0$ does not include the trivial cylinder $\R\times\gamma$. Write $e_\gamma(\mathcal{C}) = e_\gamma^+(\mathcal{C})+e_\gamma^-(\mathcal{C})$. We then have:

\begin{proposition}
\label{prop:J0}
In the above situation we have
\begin{equation}
\label{eqn:j0c}
J_0(\mathcal{C}) = 2g(C_1) - 2 + \sum_\gamma e_\gamma(\mathcal{C}).
\end{equation}
Here $g(C_1)$ denotes the genus of $C_1$, and the sum is over simple Reeb orbits $\gamma$.
\end{proposition}

\begin{proof}
The fact that the left hand side is greater than or equal to the right hand side is a restatement of \cite[Lem.\ 3.5]{HT}. Although this inequality is all that we actually need, one can further obtain equality as explained in the footnote on \cite[page 919]{HT}.
\end{proof}

\subsection{Criterion for a Birkhoff section}
\label{sec:gsscriterion}

We now recall a criterion for when the nontrivial component of a holomorphic curve counted by the $U$ map gives rise to a Birkhoff section for the Reeb flow.

\begin{definition}
\label{def:gss}
Let $(Y,\lambda)$ be a contact three-manifold. A {\bf Birkhoff section\/} for the Reeb flow is an embedded open surface $\Sigma\subset Y$ such that:
\begin{itemize}
\item
The Reeb vector field $R$ is transverse to $\Sigma$.
\item
There is a compact surface with boundary $\overline{\Sigma}$, such that $\operatorname{int}(\overline{\Sigma})=\Sigma$, and the inclusion $\Sigma\to Y$ extends to an immersion $g:\overline{\Sigma}\to Y$ such that each boundary circle of $\overline{\Sigma}$ is mapped to the image of a Reeb orbit, and $\Sigma\cap g\left(\partial\overline{\Sigma}\right)=\emptyset$.
\item
For every $y\in Y\setminus g\left(\partial\overline{\Sigma}\right)$, the Reeb trajectory starting at $y$ intersects $\Sigma$ in both forward and backward time.
\end{itemize}
\end{definition}

\begin{proposition}
\label{prop:gsscriterion}
Let $J$ be a generic $\lambda$-compatible almost complex structure, let $\alpha$ and $\beta$ be ECH generators, let $\mathcal{C}\in\mathcal{M}^J(\alpha,\beta)$, and suppose that $I(\mathcal{C})=2$. Write $\mathcal{C}=\mathcal{C}_0 + C_1$ as in Proposition~\ref{prop:lowindexcurves}. Suppose that:
\begin{itemize}
\item[(a)] $C_1$ has at least two ends.
\item[(b)] $C_1$ has genus zero.
\item[(c)] $C_1$ does not have two positive ends, or two negative ends, at covers of the same simple Reeb orbit.
\item[(d)] $C_1$ has no ends at positive hyperbolic orbits.
\item[(e)] The component of the moduli space $\mathcal{M}^J(\alpha,\beta)/\R$ containing $C_1$ is compact.
\end{itemize}
Then the restriction of the projection $\R\times Y\to Y$ to $C_1$ is an embedding, and its image is a Birkhoff section for the Reeb flow.
\end{proposition}

\begin{proof}
This follows from \cite[Prop.\ 3.2]{CGHP}, as explained in \cite[\S5.2]{CGHP}. 
\end{proof}

\section{Birkhoff sections for nondegenerate perturbations}
\label{sec:Birkhoff}

\subsection{Introduction}

The aim of this section is to prove the following proposition.

\begin{prop}
\label{prop:main}
Let $Y$ be a closed connected three-manifold, let $\lambda$ be a contact form on $Y$ with finitely many simple Reeb orbits, and assume that $c_1(\xi) \in H^2(Y;\Z)$ is torsion, where $\xi=\ker \lambda$. Let $\eta>0$ be given. Then there exists a sequence $(\lambda_k,J_k,\alpha_k,\beta_k,C_k)_{k=1,2,\dots}$ where:
\begin{itemize}
\item
$\lambda_k = f_k  \lambda$ is a nondegenerate contact form, $f_k \to 1$ in $C^\infty$, and for every simple Reeb orbit $\gamma$ of $\lambda$, we have $f_k|_\gamma = 1$ and $df_k|_\gamma = 0$.
\item
$J_k$ is a $\lambda_k$-compatible almost complex structure, and $J_k\to J$ in $C^\infty$ as $k\to\infty$, where $J$ is a $\lambda$-compatible almost complex structure,
\item
$\alpha_k$ and $\beta_k$ are (not necessarily simple) Reeb orbits for $\lambda_k$ such that $\alpha_k\to\alpha_\infty$ and $\beta_k\to\beta_\infty$ in $C^\infty$, where $\alpha_\infty$ and $\beta_\infty$ are Reeb orbits for $\lambda$,
\item
$C_k$ is a $J_k$-holomorphic cylinder in $\R\times Y$ with positive end asymptotic to $\alpha_k$ and negative end asymptotic to $\beta_k$ for each $k$,
\end{itemize}
such that for each $k$:
\begin{enumerate}
\item[(a)] $\operatorname{ind}(C_k)=I(C_k)=2$, see \S\ref{sec:hcie}.
\item[(b)]
$\mathcal{A}(\alpha_k)-\mathcal{A}(\beta_k)<\eta$.
\item[(c)] If $\alpha_\infty$ is degenerate, then the rotation number of $\alpha_k$ is in the interval $(-1/2,0)$. In particular this rotation number converges to $0$ from below as $k\to\infty$. Similarly, if $\beta_\infty$ is degenerate, then the rotation number of $\beta_k$ is in the interval $(0,1/2)$.
\item[(d)] $C_k$ projects to a Birkhoff section for the Reeb flow for $\lambda_k$, see \S\ref{sec:gsscriterion}.
\end{enumerate}
\end{prop}

\begin{remark}
In  \cite[Prop.\ 4.2]{CGHP}, it was shown that for a single nondegenerate contact form with only finitely many simple Reeb orbits, there is always a genus zero Birkhoff section with at most three punctures. For our purposes it is important to obtain a sequence of cylinders with uniform bounds on the action as $k\to\infty$, for nondegenerate contact forms that do not necessarily have finitely many simple Reeb orbits; new ideas are required for this.  We discuss this further  in the outline of the proof below.
\end{remark}

\begin{remark}
As explained in \cite[\S4.6]{HT}, the existence of a Birkhoff section given by holomorphic cylinders implies that the three-manifold $Y$ is a lens space or $S^3$. Hence Proposition~\ref{prop:main} already implies Theorem~\ref{thm:main_intro} when $Y$ is not $S^3$ or a lens space; in this case $\lambda$ must have infinitely many simple Reeb orbits.
\end{remark}

\subsubsection*{Outline of the proof of Proposition~\ref{prop:main}} For any given $\lambda_k$, one can use the non-triviality of the $U$-map in ECH to produce a ``chain'' of a large number of $J_k$-holomorphic curves with ECH index $2$. The first challenge is to show that some of these are cylinders.  Building on ideas from \cite{HT, CGHP}, we accomplish this by a kind of probabilistic argument, showing that in fact ``most" of the curves must be cylinders.  As in \cite{HT,CGHP}, the starting point behind this is to show that the curves on average have $J_0$ approximately $2$; this is the one place where we use the assumption that $c_1(\xi)$ is torsion.  
A $J_0 = 2$ curve need not be a cylinder, however (see \S\ref{sec:J0index}), and so a deeper analysis is needed; moreover, the arguments in \cite{CGHP} are geared towards finding curves with at most three punctures, under the assumption of finitely many simple Reeb orbits, and so are far from sufficing for this.
The main new ingredient we introduce for this purpose is to define a ``score'' for each curve; roughly speaking, the idea is that cylinders should be characterized by having score $0$, and on the other hand the score is nonnegative, so, one can deduce the desired existence result by arguing that the total score of the chain of curves is not too large. 

Since we must consider a sequence of contact forms $\lambda_k$, we must also keep careful track of the various parameters in these probabilistic arguments. In particular, we will define a $k$-independent constant $T_q$ such that all holomorphic curves under consideration go between orbit sets for $\lambda_k$ with action $\le T_q$. 

As a result, the proof of Proposition~\ref{prop:main} shows the following more general statement, which is also applicable to contact forms with infinitely many simple Reeb orbits; as mentioned previously, it gives, for example, information about the growth rate of the orbits when the manifold is not a lens space and the contact structure is torsion.  (For simplicity we will not write a separate proof of this as it is not needed for the main theorem.)

\begin{proposition}
\label{prop:general}
Let $Y$ be a closed connected three-manifold, let $\lambda$ be a contact form on $Y$, and assume that $c_1(\xi)\in H^2(Y;\Z)$ is torsion, where $\xi=\Ker(\lambda)$. Suppose that $\lambda$ has at least $n$ simple Reeb orbits. Let $\gamma_1,\ldots,\gamma_n$ denote the first $n$ simple Reeb orbits in terms of action\footnote{That is, $\mathcal{A}(\gamma_1)\le \cdots \le \mathcal{A}(\gamma_n)$, and any other simple Reeb orbit has action at least that of $\gamma_n$.}. Let $\eta>0$ be given. Then there is a constant $T$ computable in terms of $\eta$, the ECH spectral invariants of $\lambda$, and the actions, rotation numbers, and homology clases of $\gamma_1,\ldots,\gamma_n$, such that the following holds: If $\lambda$ does not have any simple Reeb orbit of action $\le T$ other than $\gamma_1,\ldots,\gamma_n$, then there exists a sequence $(\lambda_k,J_k,\alpha_k,\beta_k,C_k)_{k=1,2,\dots}$ as in Proposition~\ref{prop:main} satisfying all of the conclusions of Proposition~\ref{prop:main}.
\end{proposition}


\subsection{Setting the stage}
\label{sec:stage}

We first collect various preliminaries results, estimates, and constants that will be used in our arguments.  A main theme is to extract constants that are independent of sufficiently large $k$.

\subsubsection{The perturbation lemma}

The sequence of contact forms in Proposition~\ref{prop:main} will come from the following lemma, the idea of which is due to Bangert \cite{Bangert}:

\begin{lem}
\label{lem:pert}
Let $(Y,\lambda)$ be a closed contact three-manifold with finitely many simple Reeb orbits. Let $N$ be the union of pairwise disjoint open tubular neighborhoods of these orbits.  
Then there is a sequence of nondegenerate contact forms $(\lambda_k = f_k \lambda)_{k=1,2,\ldots}$ converging in $C^{\infty}$ to $\lambda$, such that:
\begin{enumerate}
\item[(a)] For every $k$ and every simple Reeb orbit $\gamma$ of $\lambda$, we have $f_k|_\gamma=1$ and $df_k|_\gamma=0$. In particular, $\gamma$ is also a simple Reeb orbit of $\lambda_k$.
\item[(b)] There exists a compact neighborhood $K\subset N$ of the Reeb orbits of $\lambda$ such that $f_k|_{Y \setminus K}=1$ for every $k$.
\item[(c)] For every~$T>0$ there exists $k_T$ such that if $k \geq k_T$ then each simple Reeb orbit~$\gamma$ for $\lambda_k$ with action $\le T$ is contained in $N$, and for each such $\gamma$, there is a simple Reeb orbit $\gamma'$ for $\lambda$ in the same component of $N$ such that:
\begin{enumerate}
\item[(1)] If $\gamma'$ and all its iterates are nondegenerate, then $\gamma = \gamma'$.
\item[(2)] If $\gamma'$ has rational rotation number, written in lowest terms with respect to some trivialization as $a/b$ with $b > 0$, then either $\gamma = \gamma'$, or $\gamma$ is $C^\infty$-close to the $b$-fold cover of $\gamma'$. 
\end{enumerate}
In particular, for any $\varepsilon>0$, if $k$ is sufficiently large, then the action and the rotation number of $\gamma$ are within $\varepsilon$ of those for $\gamma'$ in case (1), and $(\gamma')^b$ in case (2).
\end{enumerate}
\end{lem}

\begin{proof}
Consider a sequence of nondegenerate contact forms $\lambda_k = f_k \lambda$ such that $f_k \to 1$  in $C^\infty$, and such that on each of the finitely many simple Reeb orbits $\gamma$ of $\lambda$ we have $f_k|_{\gamma}=1$ and $df_k|_{\gamma}=0$. In particular, each $\gamma$ is a simple Reeb orbit for $\lambda_k$ with the same period. We can further assume that $f_k=1$ outside of a small compact neighborhood $K \subset N$ of the Reeb orbits of~$\lambda$. All the desired conclusions now follow from Bangert's lemma as stated in~\cite[Lemma~3.8]{CGHHL} or in~\cite[Proposition~1]{Bangert}.
\end{proof}

\begin{remark}
In part (c), when $\gamma'$ is degenerate, there may be many simple Reeb orbits $\gamma$ of $\lambda_k$ associated to it.
\end{remark}


\subsubsection{Setting up the proof of Proposition~\ref{prop:main}}
\label{sec:setup}

Let $(Y,\lambda)$ and $\eta>0$ as in Proposition~\ref{prop:main} be given.  Throughout the rest of this section, fix a sequence of contact forms $\lambda_k$ as provided by Lemma~\ref{lem:pert}. Let $J$ be any $\lambda$-compatible almost complex structure on $\R\times Y$. We can find a sequence $J_k$ of $\lambda_k$-compatible almost complex structures on $\R\times Y$ that converges in $C^\infty$ to $J$, and we can arrange that $J_k$ is generic as needed to define the ECH of $\lambda_k$.

Let $\mathcal{P}_k$ denote the set of simple Reeb orbits of $\lambda_k$. An orbit set $\alpha$ for $\lambda_k$ is then equivalent to a finitely supported ``multiplicity'' function $\mathcal{P}_k\to\Z_{\ge 0}$. The orbit set $\alpha=\{(\alpha_i,m_i)\}$ corresponds to the function sending $\alpha_i\mapsto m_i$ and sending all other simple Reeb orbits of $\lambda_k$ to zero. We denote the multiplicity function evaluated on a simple Reeb orbit $\gamma\in\mathcal{P}_k$ by $m(\gamma,\alpha)$. We refer to the numbers $m(\gamma,\alpha)$ as ``components'' of $\alpha$. Define the {\bf complexity} of $\alpha$ to be the maximum of the components $m(\gamma,\alpha)$, and denote the number of simple Reeb orbits that appear in $\alpha$ by $|\alpha| = \left|\{ \gamma\in\mathcal{P}_k \mid m(\gamma,\alpha)>0 \}\right|$. Likewise let $\mathcal{P}$ denote the set of simple Reeb orbits of $\lambda$; then orbit sets for $\lambda$ correspond to finitely supported functions $\mathcal{P}\to\Z_{\ge 0}$, and we define the complexity analogously.

Fix $N$ as in Lemma~\ref{lem:pert}, and fix a symplectic trivialization $\tau$ of $\xi$ over all simple Reeb orbits of $\lambda$. This induces trivializations over all (not necessarily simple) Reeb orbits of $\lambda$. Moreover, once an action bound $T>0$ is fixed, if we asssume that $k\ge k_T$, where $k_T$ is the constant given by Lemma~\ref{lem:pert}, then $\tau$ induces trivializations, again denoted by $\tau$, of all Reeb orbits of $\lambda_k$ of action $\leq T$. This is because the trivialization $\tau$ induces a symplectic trivialization, uniquely up to homotopy, of the contact structure $\xi$ over $N$. In addition, in this situation, if $\alpha$ is an orbit set for $\lambda_k$, then there is a unique orbit set $\underline{\alpha}$ for which is homotopic to $\alpha$ in $N$.

Fix a point $y\in Y\setminus N$ to define the chain map $U_{J_k,y}$, see \S\ref{sec:Umap}. We define a {\bf $U$-curve\/} to be a current counted by this chain map, namely a current $\mathcal{C}\in\mathcal{M}^{J_k}(\alpha,\beta)$ with $(0,y)\in\mathcal{C}$, where $\alpha$ and $\beta$ are ECH generators for $\lambda_k$, and the ECH index $I(\mathcal{C})=2$.


\subsubsection{Basic estimates}

The following estimate gives a useful lower bound on the difference of actions between certain orbit sets.

\begin{lem}
\label{lem:thresh}
Let $M > 1$.   Then there is a constant $\epsilon_M>0$ such that for any two orbit sets $\alpha$, $\beta$ for $\lambda$ with distinct actions, if either orbit set has complexity $\le M$, or if both orbit sets have exactly one simple Reeb orbit with multiplicity $> M$ and these simple orbits are the same, then
\begin{equation}
\label{eqn:actionbound}
\left|\mathcal{A}(\alpha) - \mathcal{A}(\beta)\right| \ge \epsilon_M \, .
\end{equation}
\end{lem}

\begin{proof}
Let $S_M$ denote the set of actions of orbit sets of complexity $\le M$. Since
the set $\mathcal{P}$ of simple Reeb orbits of $\lambda$ is finite, it follows that the set $S_M$ is finite, and the set $S_\infty$ of all actions of orbit sets is discrete.  Thus there is a positive lower bound $\epsilon'_M$ on the set of distances between two distinct actions of orbit sets with at least one in $S_M$. 

Let $\alpha$ and $\beta$ be as in the statement of the lemma.  
If either orbit set has complexity $\le M$, then their actions will differ by at least $\epsilon'_M>0$.  

If on the other hand both $\alpha$ and $\beta$ have exactly one simple Reeb orbit with multiplicity $>M$, which is the same for both orbit sets, then denote this simple Reeb orbit by $\gamma$. Consider the orbit sets $\alpha_0$ and $\beta_0$ obtained from $\alpha$ and $\beta$ by changing the multiplicity of $\gamma$ to zero and keeping the multiplicities of all other simple Reeb orbits the same. Then $\alpha_0$ and $\beta_0$ are orbit sets with complexity $\le M$. Since the set $S_M$ is finite, it follows that the set $S_M'$ of differences between elements of $S_M$ is also finite. Now we have
\[
\left|\mathcal{A}(\alpha) - \mathcal{A}(\beta)\right| = \left|\mathcal{A}(\alpha_0) - \mathcal{A}(\beta_0) + z\mathcal{A}(\gamma)\right|
\]
for some integer $z$. By finiteness of the set $S_M'$, the right hand side of the above equality, when nonzero, is bounded from below by some $\epsilon_M''>0$.
We then take $\epsilon_M = \min(\epsilon_M',\epsilon_M'')$.
\end{proof}

There is also the following useful bound on the relative Chern class. If $\alpha$ is a null-homologous orbit set (for $\lambda$ or $\lambda_k$), define
\[
c_\tau(\alpha) = c_\tau(\alpha,\emptyset,Z) \in \Z,
\]
where $c_\tau$ denotes the relative first Chern class defined in \S\ref{sec:orbitsets}, $\tau$ is the trivialization we fixed in \S\ref{sec:setup}, and $Z\in H_2(Y,\alpha,\emptyset)$ is arbitrary. It follows from equation \eqref{eqn:ctauambiguity} and the assumption that $c_1(\xi)$ is torsion that $c_\tau(\alpha)$ does not depend on the choice of $Z$.

\begin{lem}
\label{lem:chern}
For each $\ell > 0$ there exists $\delta_1 > 0$ (depending only on $\lambda$ and $\ell$) such that if $\alpha$ is a nullhomologous orbit set for $\lambda$, then
\begin{equation}
\label{eqn:chernbound}
\left|2 c_{\tau}(\alpha) + \CZ^{\rm top}_{\tau}(\alpha))\right| + 2 \frac{\mathcal{A}({\alpha})}{\ell} \le \delta_1 \mathcal{A}(\alpha).
\end{equation}
\end{lem}

Here $\CZ^{\operatorname{top}}$ is defined as in equation \eqref{eqn:CZtop}. In that equation, we use the definition of the Conley-Zehnder index in equation \eqref{eqn:CZtau} also for degenerate Reeb orbits, as in \cite{CGHHL}.

\begin{proof}[Proof of Lemma~\ref{lem:chern}.]
Without the $2 \mathcal{A}({\alpha})/\ell$ term, this follows from the same argument as in \cite[Lemma 4.3(b)]{CGHP}. The $2 \mathcal{A}({\alpha})/\ell$ term can be accommodated by increasing $\delta_1$.
\end{proof}

\subsubsection{Exceptional pairs}

Let $\gamma$ be a simple Reeb orbit for $\lambda_k$ and let $m$ be a positive integer. As in \cite{CGHP}, we call the pair $(\gamma,m)$ {\bf exceptional} if $|p^+_\gamma(m)| + |p^-_\gamma(m)| \leq 3$, where the partitions $p_\gamma^\pm(m)$ are defined in \S\ref{sec:part}. We will need the following variant of \cite[Lem.\ 4.9]{CGHP}, restricting the exceptional pairs in a $k$-independent manner:

\begin{lem}
\label{lem:exceptional}
There exists a constant $M \ge 3$, depending only on $\lambda$, such that the following holds. For every $T>0$ there is a constant $k_T'$ such that if $k \geq k_T'$, if $\gamma$ is a simple Reeb orbit for $\lambda_k$, if $m \ge M$, and if $m\mathcal{A}(\gamma) \le T$, then the pair $(\gamma,m)$ is not exceptional.
\end{lem}

\begin{proof}
Given a number $\theta\in\R/\Z$, write
\[
\overline{\theta} = \min\{\{\theta\},1-\{\theta\}\}\in [0,1/2].
\]
By Lemma~\ref{lem:partitionfacts}(c), if $\gamma\in\mathcal{P}_k$ has rotation number $\theta$, and if $m\overline{\theta}>2$, then the pair $(\gamma,m)$ is not exceptional. Also, by Example~\ref{example:smallangle}, if $m\overline{\theta}<1$ and $m\ge 3$ then the pair $(\gamma,m)$ is not exceptional. So to prove the lemma, it is enough to show that there exists a constant $M\ge 3$, depending only on $\lambda$, such that:
\begin{itemize}
\item[(*)] For every $T>0$, there exists a constant $k_T'$ such that if $k\ge k_T'$, if $\gamma\in\mathcal{P}_k$ has rotation number $\theta$, if $m\ge M$, and if $m\mathcal{A}(\gamma)\le T$, then $m\overline{\theta}\notin(1,2)$.
\end{itemize}

Let $\Theta\subset\R/\Z$ denote the finite set consisting of the rotation numbers of the simple Reeb orbits of $\gamma$. We claim that the requirement (*) will be fulfilled as long as $M\ge 3$ and 
\begin{equation}
\label{eqn:Mchoice}
M > \max\left\{2/\overline{\theta} \;\big|\; \theta\in\Theta\setminus\{0\}\right\}.
\end{equation}

To see this, let $T>0$ be given. By Lemma~\ref{lem:pert}, if $k$ is sufficiently large, then the rotation number $\theta$ of a simple Reeb orbit $\gamma$ of $\lambda_k$ with action less than $T$ will be close to $\Theta\cup\{0\}$. If $\theta$ is close to $\Theta\setminus\{0\}$ and $m\ge M$, then we will have $m\overline{\theta}>2$ by \eqref{eqn:Mchoice}. If $\theta$ is close to $0$ and $m\ge M$ and $m\mathcal{A}(\gamma)\le T$, then we will have $m\overline{\theta}<1$, because by Lemma~\ref{lem:pert}, for $k$ sufficiently large we have $m\le T/\ell+1$ where $\ell$ is the minimum action of a Reeb orbit of $\lambda$.
\end{proof}


\subsubsection{Some constants}
\label{sec:constants}

We now define some constants associated to $\lambda$ that will be used below.  We emphasize that these constants only depend on $\lambda$ and are independent of~$k$.

\begin{itemize}
\item
Let $B_0$ be the maximum of the denominators $b \geq 1$ of the rational numbers $a/b$ (in lowest terms) that are rotation numbers of simple orbits for $\lambda$, and let $B=4B_0$; if there are no degenerate orbits set $B_0 = 1$.
\item
Let $M$ be the constant provided by Lemma~\ref{lem:exceptional}, governing the exceptional pairs $(\gamma,m)$. By increasing $M$ if necessary, we also assume that $M>B_0$.
\item
Fix $\ell>0$ which is less than the smallest action of a Reeb orbit for $\lambda$, and also less than the action of all Reeb orbits for $\lambda_k$ for every $k$.
\item
Let $\epsilon_{BM}$ be the constant provided by Lemma~\ref{lem:thresh} using $M$ as above, giving a lower bound on the action difference between orbit sets with low complexity.  By shrinking $\epsilon_{BM}$ if necessary, we can further assume that $\epsilon_{BM}<\ell$.
\item
Let $\epsilon' = \min\{\eta,\epsilon_{BM}/2\} > 0$, where $\eta>0$ was given as an input to Proposition~\ref{prop:main}, see \S\ref{sec:setup}.
\item
Let $\delta_1 > 0$ be the constant provided by Lemma~\ref{lem:chern} using $\ell$ as above, controlling the first Chern class of nullhomologous orbit sets.
\end{itemize}

To define the remaining constants, note that by Proposition~\ref{prop:usequence} and our hypothesis that $c_1(\xi)$ is torsion, we can choose a $U$-sequence $\{\sigma_p\}_{p\ge p_0}$ in $ECH(Y,\xi,0)$ with $p_0>0$ such that each $\sigma_p$ has grading $2p$, as defined in equation \eqref{eqn:canonicalgrading}. Denote the corresponding spectral invariant by $c_p=c_{\sigma_p}(Y,\lambda)\in\R$.

\begin{itemize}
\item
By the ``Weyl law'' of Proposition~\ref{prop:vol}, we can fix a constant $\delta_2>0$ such that 
\begin{equation}
\label{eqn:cqbound}
c_p + 1 \le \delta_2 p^{1/2}
\end{equation}
holds for all $p> p_0$.
\item
Fix $q > p_0$ to be any integer large enough so that
\begin{align}
\label{eqn:defnq1}
q^{2/3} &> 4 \delta_2 q^{1/2}/\ell + 6 \delta_1 \delta_2 q^{1/2} + 9 \delta_2 q^{1/2} / \epsilon', \\
\label{eqn:defnq2}
q^{4/5} &> 4 \delta_2 q^{1/2}/\ell + 4 \delta_1 \delta_2 q^{1/2} + 7 \left(q^{2/3} + \delta_2 q^{1/2}/\epsilon'\right), \\
\label{eqn:defnq3}
q^{5/6} & >
2\delta_2 q^{1/2}/\ell + 4 \delta_1 \delta_2 q^{1/2} +8\left(\delta_2q^{1/2}/\epsilon' + q^{4/5} + q^{2/3}\right),\\
\label{eqn:defnq4}
q - p_0 &>  q^{5/6} + q^{4/5} + q^{2/3} + \delta_2 q^{1/2}/\epsilon' . 
\end{align}
We will use these inequalities at the end of the proof in \S\ref{sec:puttogether}. For now we note that they can all be satisfied for sufficiently large $q$, where the demand of sufficiency depends only on $\ell$, $\delta_1$, $\delta_2$, $\epsilon'$, and $p_0$, all of which are $k$-independent.
\item
Finally, define 
\begin{equation}
\label{eqn:Tq}
T_q = \delta_2 q^{1/2}.
\end{equation}
\end{itemize}


\subsubsection{More basic estimates}

We will also need two more estimates; these involve the constants from the previous section.  

We have the following lower bound on the energy of holomorphic currents counted by the $U$-map:

\begin{lemma}
\label{lem:nonloc}
There is a constant $\bar{h} > 0$ such that the following holds when $k$ is sufficiently large: If $\mathcal{C}\in\mathcal{M}^{J_k}(\alpha,\beta)$ is a $U$-curve and $\mathcal{A}(\alpha)<T_q$, then
\begin{equation}
\label{eqn:nonloc}
\mathcal{A}(\alpha) - \mathcal{A}(\beta) > \bar{h}.
\end{equation}     
\end{lemma}

\begin{proof}
This follows the proof of \cite[Lemma~3.1(b)]{CGH}. The idea is to suppose that there exists a sequence of $U$-curves $\mathcal{C}_i\in\mathcal{M}^{J_i}(\alpha_i,\beta_i)$ with $\mathcal{A}(\alpha_i)<T_q$ and $\lim_{i\to\infty}(\mathcal{A}(\alpha_i)-\mathcal{A}(\beta_i))=0$. Once can then use Gromov compactness for currents as proved by Taubes \cite[Prop.\ 3.3]{taubes_currents} (see \cite[Prop.\ 1.9]{dw} for a more general version) to obtain a contradiction.
\end{proof}

\begin{lem}
\label{lem:threshold}
If $k$ is sufficiently large, then for any two orbit sets $\alpha$ and $\beta$ for $\lambda_k$ with action $\le T_q$ and with $|\alpha| + |\beta| \le 4$, if at least one of them has complexity $\le M$, or if both have exactly one simple Reeb orbit with multiplicity $> M$ and these simple orbits are the same, we have either
\begin{equation}
\label{eqn:toavoid}
|\mathcal{A}(\alpha) - \mathcal{A}(\beta)| \le \bar{h}/2
\end{equation}
or
\begin{equation}
\label{eqn:keyeq}
|\mathcal{A}(\alpha) - \mathcal{A}(\beta)| \ge \epsilon'.
\end{equation}
\end{lem}

Here the constant $\bar{h}$ is provided by Lemma~\ref{lem:nonloc}. Also recall that $\epsilon'\le \epsilon_{BM}/2$ where the constant $\epsilon_{BM}$ is provided by Lemma~\ref{lem:thresh}.

\begin{proof}
Assume to start that $k\ge k_{T_q}$ where the constant $k_{T_q}$ is given by Lemma~\ref{lem:pert}(c). Recall from \S\ref{sec:setup} that $\alpha$ and $\beta$ determine orbit sets $\underline{\alpha}$ and $\underline{\beta}$ for $\lambda$. By Lemma~\ref{lem:pert}(c), if $k$ is sufficiently large, then
\[
|\mathcal{A}(\alpha)-\mathcal{A}(\underline{\alpha})| \le \bar{h}/4,
\]
and likewise for $\beta$. Thus if $\mathcal{A}(\underline{\alpha})=\mathcal{A}(\underline{\beta})$, and if $k$ is sufficiently large, then the inequality \eqref{eqn:toavoid} holds. So we can assume without loss of generality that $\mathcal{A}(\underline{\alpha}) \neq \mathcal{A}(\underline{\beta})$, and we want to show that the inequality \eqref{eqn:keyeq} holds if $k$ is sufficiently large.

If $\gamma'$ is a simple Reeb orbit of $\lambda$, let $\mathcal{P}_k(\gamma')$ denote the set of simple Reeb orbits of $\lambda$ that are contained in the same component of $N$ as $\gamma'$. Then by Lemma~\ref{lem:pert}(c), we have
\begin{equation}
\label{estimate_on_multiplicity}
m(\gamma',\underline{\alpha}) \leq B_0 \sum_{\gamma\in P_{k}(\gamma')} m(\gamma,\alpha).
\end{equation}
Since $|\alpha|\le 4$ by hypothesis, it follows that if $\alpha$ has complexity $\le M$, then $\underline{\alpha}$ has complexity $\le BM$. The same holds for $\beta$. Thus if $\alpha$ or $\beta$ has complexity $\le M$, then Lemma~\ref{lem:thresh} applies to give
\begin{equation}
\label{eqn:applythresh}
|\mathcal{A}(\underline{\alpha})-\mathcal{A}(\underline{\beta})|\ge \epsilon_{BM},
\end{equation}
so that the inequality \eqref{eqn:keyeq} holds if $k$ is sufficiently large.

The remaining possibility is that $\alpha$ and $\beta$ have exactly one simple Reeb orbit with multiplicity $>M$, and this simple orbit is the same for each, call it $\gamma$. There is then a simple Reeb orbit $\gamma'$ for $\lambda$ such that $\gamma\in\mathcal{P}_k(\gamma')$. If $\gamma''$ is a simple Reeb orbit for $\lambda$ with $\gamma''\neq\gamma'$, then it follows from \eqref{estimate_on_multiplicity} that $m(\gamma'',\underline{\alpha})\le BM$, and likewise for $\underline{\beta}$. Thus Lemma~\ref{lem:thresh} applies again to give \eqref{eqn:applythresh}, so that the inequality \eqref{eqn:keyeq} holds if $k$ is sufficiently large.
\end{proof}


\subsubsection{A chain of $U$-curves }
\label{sec:uchain}

We now use the $U$-map on ECH to obtain a ``chain'' of $U$-curves. The following is loosely related to~\cite[Lemma~4.4]{CGHP}:

\begin{lem}
\label{lem:ucurves}
If $k$ is sufficiently large, then there exist:
\begin{itemize}
\item ECH generators $\alpha(p_0), \ldots, \alpha(q)$ for $\lambda_k$, with $[\alpha(i)] = 0 \in H_1(Y)$ and $\mathcal{A}(\alpha(q)) \le T_q$,
\item
$U$-curves $\mathcal{C}(i)\in\mathcal{M}^{J_k}(\alpha(i),\alpha(i-1))$ for $p_0 < i \le q$,
\end{itemize}
such that:
\begin{itemize}
\item[(a)] There are at most $T_q/\epsilon'$ indices $i\in\{p_0+1,\ldots,q\}$ such that
\[
\mathcal{A}(\alpha(i))-\mathcal{A}(\alpha(i-1)) \ge \epsilon'.
\]
\item[(b)] We have
\begin{equation}
\label{eqn:j0equation}
\sum^q_{i={p_0+1}} J_0(\mathcal{C}(i)) \le 2(q-p_0) + 2 \delta_1T_q.
\end{equation}
\end{itemize}
\end{lem}

\begin{proof}
Since the ECH spectral invariants are $C^0$-continuous functions of the contact form, see \S\ref{sec:volume}, it follows from the bound \eqref{eqn:cqbound} and equation \eqref{eqn:Tq} that if $k$ is sufficiently large then
\begin{equation}
\label{eqn:cqkbound}
c_{\sigma_q}(Y,\lambda_k) \le T_q.
\end{equation}
By the definition of the spectral invariant $c_{\sigma_q}$, we can find a cycle $x$ in the ECH chain complex $(ECC(Y,\lambda,0),\partial_{J_k})$ representing the class $\sigma_q$, such that each ECH generator that is a summand in $x$ has action $\le T_q$. By the definition of $U$-sequence, we have $U_{J_k,y}^{q-p_0}x\neq 0$. There is then some summand of $x$, which we can call $\alpha(q)$, such that $U_{J_k,y}^{q-p_0}x\neq 0$. This means that there are ECH generators $\alpha(p_0),\ldots,\alpha(q-1)$ such that
\[
\langle U_{J_k,y}\alpha(i),\alpha(i-1)\rangle \neq 0
\]
for $p_0<i\le q$. By the definition \eqref{eqn:Uchainmap} of the chain map $U_{J_k,y}$, it follows that for each $i$ with $p_0<i\le q$ there exists a $U$-curve $\mathcal{C}(i)\in\mathcal{M}^{J_k}(\alpha(i),\alpha(i-1))$.

Assertion (a) follows immediately from the fact that $\mathcal{A}(\alpha(q))\le T_q$ and the nonnegativity of $\mathcal{A}(\alpha(i))$ and $\mathcal{A}(\alpha(i))-\mathcal{A}(\alpha(i-1))$ for each $i$. It remains to prove assertion (b).

By equation \eqref{eqn:defJ0} and the fact that $I(\mathcal{C}(i))=2$, we have a telescoping sum
\[
\begin{split}
\sum_{i=p_0+1}^q J_0(\mathcal{C}(i)) - 2(q - p_0) &= \sum_{i=p_0+1}^q\left(J_0(\mathcal{C}(i)) - I(\mathcal{C}(i))\right)\\
&= \sum_{i=p_0+1}^q\left(\left[2c_\tau(\alpha(i-1))+\CZ_\tau^{\operatorname{top}}(\alpha(i-1))\right] - \left[2c_\tau(\alpha(i)) + \CZ_\tau^{\operatorname{top}}(\alpha(i))\right]\right)\\
&= \left[2c_\tau(\alpha(p_0)) + \CZ_\tau^{\operatorname{top}}(\alpha(p_0))\right] - \left[2c_\tau(\alpha(q)) + \CZ_\tau^{\operatorname{top}}(\alpha(q)\right].
\end{split}
\]
So to prove part (b), it is enough to show that if $k$ is sufficiently large, then
\begin{equation}
\label{eqn:J0Ibound}
\left|2c_\tau(\alpha) + \CZ_\tau^{\operatorname{top}}(\alpha)\right| \le \delta_1T_q
\end{equation}
for every ECH generator $\alpha$ for $\lambda_k$ with $\mathcal{A}(\alpha) \le T_q$.

Suppose to start that $k \geq k_{T_q}$, where the constant $k_{T_q}$ is provided by Lemma~\ref{lem:pert}(c). Recall from \S\ref{sec:setup} that for each orbit set $\alpha$ for $\lambda_k$ with $\mathcal{A}(\alpha)\le T_q$, there is a unique orbit set $\underline{\alpha}$ for $\lambda$ such that $\alpha$ is homologous to $\underline{\alpha}$ in $N$. Since the trivialization $\tau$ of $\xi$ over the Reeb orbits in $\alpha$ and $\underline{\alpha}$ extends to a trivialization of $\xi$ over all of $N$, it follows from the definition of the relative first Chern class in \S\ref{sec:orbitsets} that
\[
c_\tau(\alpha) = c_\tau(\underline{\alpha}).
\]
By this equation and Lemma~\ref{lem:chern}, to prove \eqref{eqn:J0Ibound}, it is enough to show that
\[
\left|\CZ_\tau^{\operatorname{top}}(\alpha) - \CZ_\tau^{\operatorname{top}}\left(\underline{\alpha}\right)\right| \le \frac{\mathcal{A}(\alpha)}{\ell}.
\]

By our choice of $\ell$ in \S\ref{sec:constants}, there are at most $\mathcal{A}(\alpha)/\ell$ Reeb orbits in $\alpha$. So to complete the proof, it is enough to show that if $k\ge k_{T_q}$ is sufficiently large, then for every (not necessarily simple) Reeb orbit $\gamma$ of $\lambda_k$ with $\mathcal{A}(\gamma)\le T_q$, if $\underline{\gamma}$ is the corresponding Reeb orbit of $\lambda$ (in the notation of Lemma~\ref{lem:pert}(c), $\underline{\gamma}$ is either $\gamma'$ or $(\gamma')^b$), then
\begin{equation}
\label{eqn:CZchange}
\left|\CZ_\tau(\gamma) - \CZ_\tau\left(\underline{\gamma}\right)\right| \le 1.
\end{equation}
By Lemma~\ref{lem:pert}(c) and equation \eqref{eqn:CZtau}, the bound \eqref{eqn:CZchange} holds when $k$ is sufficiently large.
\end{proof}


\subsection{The score}

To analyze the chain of $U$-curves produced in \S\ref{sec:uchain}, we now define a ``score'' associated to each $U$-curve and analyze it in detail.

\subsubsection{Definition of the score}

\begin{definition}
Let $\alpha$ be an orbit set for $\lambda_k$, and let $\gamma$ be a simple Reeb orbit of $\lambda_k$.
\begin{itemize}
\item
We say that $\gamma$ is a {\bf special component\/} of $\alpha$ if $m(\gamma,\alpha)>1$ and $1\notin p_\gamma^+(m(\gamma,\alpha))$. Define $s(\gamma,\alpha)$ to be $1$ if $\gamma$ is a special component of $\alpha$, and $0$ otherwise. Denote the number of special components of $\alpha$ by
\[
s(\alpha) = \sum_{\gamma\in\mathcal{P}_k}s(\gamma,\alpha).
\]
\item
We say that $\gamma$ is a {\bf $p^+$-component\/} of $\alpha$ if $p_\gamma^+(m(\gamma,\alpha))=(m(\gamma,\alpha))$ and $m(\gamma,\alpha)\ge M$, where $M$ is the constant provided by Lemma~\ref{lem:exceptional}. Define $p^+(\gamma,\alpha)$ to be $1$ if $\gamma$ is a $p^+$-component of $\alpha$, and $0$ otherwise. Denote the number of $p^+$-components of $\alpha$ by
\[
p^+(\alpha) = \sum_{\gamma\in\mathcal{P}_k}p^+(\gamma,\alpha).
\]
\item
Similarly, we say that $\gamma$ is a {\bf $p^-$-component\/} of $\alpha$ if $p_\gamma^-(m(\gamma,\alpha))=(m(\gamma,\alpha))$ and $m(\gamma,\alpha)\ge M$. Define $p^-(\gamma,\alpha)$ to be $1$ if $\gamma$ is a $p^-$-component of $\alpha$, and $0$ otherwise. Denote the number of $p^-$-components of $\alpha$ by
\[
p^-(\alpha) = \sum_{\gamma\in\mathcal{P}_k}p^-(\gamma,\alpha).
\]
\end{itemize}
\end{definition}

\begin{definition}
If $\alpha$ is an orbit set for $\lambda_k$, its {\bf orbit score} is the quantity
\[
S(\alpha) := p^+(\alpha) + s(\alpha) - p^-(\alpha).
\]
\end{definition}

\begin{definition}
Let $\mathcal{C}\in\mathcal{M}^{J_k}(\alpha,\beta)$ be a $U$-curve. The {\bf orbit score\/} of $\mathcal{C}$ is
\[
S(\mathcal{C}) = S(\alpha) - S(\beta).
\]
The {\bf total score\/} of $\mathcal{C}$ (also called the ``score'' when there is no danger of confusion) is
\[
T(\mathcal{C}) = 3y(\mathcal{C}) + S(\mathcal{C})
\]
where we define
\[
y(\mathcal{C}) = J_0(\mathcal{C}) - 2.
\]
\end{definition}

\begin{definition}
If $\mathcal{C}$ is a $U$-curve as above, and if $\gamma$ is a simple Reeb orbit of $\lambda_k$, we define the contribution to the orbit score of $\mathcal{C}$ associated to $\gamma$ by
\begin{equation}
\label{eqn:Sgamma}
S_\gamma(\mathcal{C}) = \left(p^+(\gamma,\alpha) + s(\gamma,\alpha) + p^-(\gamma,\beta)\right) - \left(p^+(\gamma,\beta) + s(\gamma,\beta) + p^-(\gamma,\alpha) \right).
\end{equation}
We define the contribution to the total score of $\mathcal{C}$ associated to $\gamma$ by
\begin{equation}
\label{eqn:Tgamma}
T_\gamma(\mathcal{C}) = S_\gamma(\mathcal{C}) + 3e_\gamma(\mathcal{C}),
\end{equation}
where $e_\gamma(\mathcal{C})$ was defined in \S\ref{sec:J0index}.
\end{definition}

\begin{remark}
\label{rem:Tsum}
It follows from equation \eqref{eqn:j0c} and the above definitions that if we write $\mathcal{C}=\mathcal{C}_0+C_1$ as in Proposition~\ref{prop:lowindexcurves}, then
\begin{equation}
\label{eqn:Tsum}
T(\mathcal{C}) = 6g(C_1) - 12 + \sum_{\gamma\in\mathcal{P}_k} T_\gamma(\mathcal{C}) .
\end{equation}
Moreover, in the sum on the right, the contribution $T_\gamma(\mathcal{C})$ is nonzero only if $C_1$ has an end asymptotic to a cover of $\gamma$.
\end{remark}


\subsubsection{How orbits contribute to the score}
\label{sec:orbitcontributions}

We now collect some lemmas about how orbits contribute to the score of a $U$-curve.

\begin{lem}
\label{lem:emp}
If $k$ is sufficiently large then the following holds. Let $\mathcal{C} = \mathcal{C}_0 + C_1 \in \mathcal{M}^{J_k}(\alpha,\beta)$ be a $U$-curve such that $\mathcal{A}(\alpha)\le T_q$. Let $\gamma\in\mathcal{P}_k$ and suppose that $C_1$ has at least one end asymptotic to a cover of $\gamma$. Let $m_0$ denote the multiplicity of the trivial cylinder $\R\times\gamma$ in $\mathcal{C}_0$. Then:
\begin{itemize}
\item[(a)] $T_\gamma(\mathcal{C}) \ge 3$.
\item[(b)]
If $T_\gamma(\mathcal{C}) = 3$, then $m_0=0$, the curve $C_1$ has a unique end at a cover of $\gamma$, and if this end is positive then the covering multiplicity must be one.
\end{itemize}
\end{lem}

\begin{proof}
By equation \eqref{eqn:Sgamma} we have $S_\gamma(\mathcal{C})\ge -3$. If $e_\gamma(\mathcal{C})\ge 3$ then it follows from equation \eqref{eqn:Tgamma} that $T_\gamma(\mathcal{C}) \ge 6$. It therefore remains to consider the cases where $e_\gamma(\mathcal{C}) \in \{1,2\}$.

{\em Case 1:} Suppose that $e_\gamma(\mathcal{C})=1$. In this case, it follows from the definition of $e_\gamma$ in \S\ref{sec:J0index} that $m_0=0$, and the curve $C_1$ has just one end at a cover of $\gamma$.

{\em Subcase 1a:} Suppose that the end of $C_1$ at a cover of $\gamma$ is a positive end.
Then $p^+(\gamma,\beta) = s(\gamma,\beta) = 0$.

If $m(\gamma,\alpha)=1$, then $p^-(\gamma,\alpha)=0$ (since $M>1$), so $S_\gamma(\mathcal{C})\ge 0$, so $T_\gamma(\mathcal{C})\ge 3$, and the conditions for equality in (b) hold.

If $m(\gamma,\alpha)>1$, then it follows from Proposition~\ref{prop:partitionconditions} that $p_{\gamma}^+(m(\gamma,\alpha)) = (m(\gamma,\alpha))$, so $s(\gamma,\alpha)=1$. Since $\alpha$ is an ECH generator, $\gamma$ has irrational rotational number, so it follows from Lemma~\ref{lem:partitionfacts}(b) that $p^-_{\gamma}(m(\gamma,\alpha))\neq(m(\gamma,\alpha))$. This implies that $p^-(\gamma,\alpha)=0$. We conclude that $S_\gamma(\mathcal{C}) > 0$, so $T_\gamma(\mathcal{C})>3$.

{\em Subcase 1b:} Suppose that the end of $C_1$ at a cover of $\gamma$ is a negative end.

Then $p^-(\gamma,\alpha)=0$. Similarly to the above paragraph, it follows from Proposition~\ref{prop:partitionconditions} that $p_\gamma^-(m(\gamma,\beta))=(m(\gamma,\beta))$ and then from Lemma~\ref{lem:partitionfacts}(b) that $p^+(\gamma,\beta)=0$. By Lemma~\ref{lem:partitionfacts}(a), we have $1\in p_\gamma^+(m(\gamma,\beta))$, so $s(\gamma,\beta)=0$. Thus $S_\gamma(\mathcal{C})\ge 0$, so $T_\gamma(\mathcal{C})\ge 3$, and the conditions for equality in (b) hold.

{\em Case 2:} Suppose that $e_\gamma(\mathcal{C})=2$. It will suffice to show that $S_\gamma(\mathcal{C})\ge -2$, so that $T_\gamma(C)>3$.

{\em Subcase 2a:} Suppose that $m_0=0$. Then by the definition of $e_\gamma$, the curve $C_1$ has a positive end and a negative end asymptotic to covers of $\gamma$. Then $s(\gamma,\beta)=0$ by the same argument as in Subcase 1b, so $S_\gamma(\mathcal{C})\ge -2$.

{\em Subcase 2b:} Suppose that $m_0>0$. Then $C_1$ has a unique end asymptotic to a cover of $\gamma$. To prove that $S_\gamma(\mathcal{C})\ge -2$, suppose to get a contradiction that $S_\gamma(\mathcal{C})=-3$. Then $\gamma$ is both a $p^-$-component of $\alpha$ and a $p^+$-component of $\beta$.

Consider the case where the end of $C_1$ asymptotic to a cover of $\gamma$ is a positive end, and let $d$ denote the covering multiplicity. Then $m(\gamma,\beta)=m_0$, and $m(\gamma,\alpha) = m(\gamma,\beta) + d$. By Proposition~\ref{prop:partitionconditions}, 
\[
\begin{split}
p_{\gamma}^+(m(\gamma,\alpha)) &= p_{\gamma}^+(m(\gamma,\beta)) \cup (d)\\
& = (m(\gamma,\beta),d)
\end{split}
\]
since $\gamma$ is a $p^+$-component of $\beta$. In particular $|p_{\gamma}^+(m(\gamma,\alpha))|=2$.

Since $\gamma$ is a $p^-$-component of $\alpha$, we also have $|p_{\gamma}^-(m(\gamma,\alpha))|=1$. We conclude that the pair $(\gamma,m(\gamma,\alpha))$ is exceptional. Since $m(\gamma,\alpha)\ge M$ by the definition of ``$p^+$-component'', this contradicts 
 Lemma~\ref{lem:exceptional} if $k\ge k_{T_q}'$.

The case when the end of $C_1$ asymptotic to a cover of $\gamma$ is a negative end is handled by a symmetric argument.
\end{proof}

The next lemma gives a refinement of the previous one when we assume in addition that $C_1$ has a ``high multiplicity'' positive end.

\begin{definition}
A positive or negative end of a $J_k$-holomorphic curve asymptotic to the $d$-fold cover of an orbit $\gamma\in\mathcal{P}_k$ is {\bf high multiplicity\/} if $d\ge M$.
\end{definition}

\begin{lem}
\label{lem:high}
If $k$ is sufficiently large then the following holds. Let $\mathcal{C} = \mathcal{C}_0 + C_1 \in \mathcal{M}^{J_k}(\alpha,\beta)$ be a $U$-curve such that $\mathcal{A}(\alpha)\le T_q$. Let $\gamma\in\mathcal{P}_k$ and suppose that $C_1$ has at least one high multiplicity positive end asymptotic to a cover of $\gamma$. Let $m_0$ denote the multiplicity of the trivial cylinder $\R\times\gamma$ in $\mathcal{C}_0$. Then:
\begin{itemize}
\item[(a)]
If $C_1$ has a unique end asymptotic to a cover of $\gamma$, then
$T_\gamma(\mathcal{C}) \ge 5$, with equality only if $m_0\neq 1$.
\item[(b)]
If $C_1$ has at least two ends asymptotic to covers of $\gamma$, then $T_\gamma(\mathcal{C}) \ge 8$, with equality only if $C_1$ has exactly one positive and one negative end asymptotic to covers of $\gamma$ and $m_0=0$.
\end{itemize}
\end{lem}

\begin{proof}
(a) Suppose that $C_1$ has a unique end asymptotic to a cover of $\gamma$.

{\em Case 0:} Suppose that $m_0=0$. We need to show that $T_\gamma(\mathcal{C})\ge 5$.

In this case $e_\gamma(\mathcal{C})=1$, so we need to show that $S_\gamma(\mathcal{C})\ge 2$. From Proposition~\ref{prop:partitionconditions} and Lemma~\ref{lem:partitionfacts}(b), we learn that $\gamma$ is a $p^+$-component and a special component but not a $p^-$-component of $\alpha$. Hence $S_\gamma(\mathcal{C})=2$. 

{\em Case 1:}
Suppose that $m_0=1$. We need to show that $T_\gamma(\mathcal{C})\ge 6$.

In this case $e_\gamma(\mathcal{C})=2$, so we need to show that $S_\gamma(\mathcal{C})\ge 0$. Since $C_1$ has no negative ends asymptotic to covers of $\gamma$, it follows that $p^+(\gamma,\beta)=s(\gamma,\beta)=0$. In addition, if $k\ge k_{T_q}''$, then $p^-(\gamma,\alpha)=0$, because otherwise the pair $(\gamma,m(\gamma,\alpha))$ would be exceptional by Proposition~\ref{prop:partitionconditions}, and this would contradict Lemma~\ref{lem:exceptional} because $m(\gamma,\alpha)\ge M+1$. It then follows from equation \eqref{eqn:Sgamma} that $S_\gamma(\mathcal{C})\ge 0$.

{\em Case 2:}
Suppose that $m_0 \ge 2$. We need to show that $T_\gamma(\mathcal{C})\ge 5$.

In this case $e_\gamma(\mathcal{C})=2$ as above, so we need to show that $S_\gamma(\mathcal{C})\ge -1$. By Proposition~\ref{prop:partitionconditions} and Remark~\ref{rem:7.28}, we have $1\in p_\gamma^+(m(\gamma,\alpha))$ if and only if $1\in p_\gamma^+(m_0)$. Since $m_0=m(\gamma,\beta)$, it follows that $s(\gamma,\alpha) - s(\gamma,\beta)=0$. 

To finish showing that $S_\gamma(\mathcal{C})\ge -1$, we observe that if $k\ge k_{T_q}'$ then $\gamma$ cannot be both a $p^-$-component of $\alpha$ and a $p^+$-component of $\beta$. Otherwise the pair $(\gamma,m(\gamma,\alpha))$ would be exceptional by Proposition~\ref{prop:partitionconditions}, contradicting Lemma~\ref{lem:exceptional} since $m(\gamma,\alpha) \ge M+2$.

(b) Suppose that $C_1$ has at least two ends asymptotic to covers of $\gamma$.

Note first that if $m_0>0$, or if $C_1$ has at least three ends asymptotic to covers of $\gamma$, then $e_\gamma(\mathcal{C}) \ge 4$, and so $T_\gamma(\mathcal{C}) \ge 9$. Thus we can assume that $m_0=0$ and that $C_1$ has exactly two ends asymptotic to covers of $\gamma$.

{\em Case 1:} Suppose that  $C_1$ has two positive ends asymptotic to covers of $\gamma$. We need to show that $T_\gamma(\mathcal{C})\ge 9$.

In this case $e_\gamma(\mathcal{C}) = 3$, so we need to show that $S_\gamma(\mathcal{C}) \ge 0$. Since $m(\gamma,\beta)=0$, it is enough to show that $\gamma$ is not a $p^-$-component of $\alpha$. But if $\gamma$ were a $p^-$-component of $\alpha$, then the pair $(\gamma,m(\gamma,\alpha))$ would be exceptional by Proposition~\ref{prop:partitionconditions}, contradicting Lemma~\ref{lem:exceptional} if $k\ge k_{T_q}'$. 

{\em Case 2:} Suppose that $C_1$ has a positive end and a negative end asymptotic to covers of $\gamma$. We need to show that $T_\gamma(\mathcal{C})\ge 8$.

In this case $e_\gamma(\mathcal{C})=2$, so we need to show that $S_\gamma(\mathcal{C})\ge 2$. By Proposition~\ref{prop:partitionconditions}, $p_\gamma^+(m(\gamma,\alpha))=(m(\gamma,\alpha))$, so $\gamma$ is a $p^+$-component and special component of $\alpha$. By Lemma~\ref{lem:partitionfacts}(b), $\gamma$ cannot be a $p^-$-component of $\alpha$. By Proposition~\ref{prop:partitionconditions} again, $p_\gamma^-(m(\gamma,\beta))=(m(\gamma,\beta))$. So by Lemma~\ref{lem:partitionfacts}(b), $\gamma$ cannot be a $p^+$-component of $\beta$, and by Lemma~\ref{lem:partitionfacts}(a), $\gamma$ cannot be a special component of $\beta$. Thus $S_\gamma(\mathcal{C})\ge 2$ by equation \eqref{eqn:Sgamma}.
\end{proof}

The next lemma gives a similar refinement of Lemma~\ref{lem:emp} when we assume in addition that $C_1$ has a high multiplicity negative end.

\begin{lem}
\label{lem:low}
If $k$ is sufficiently large then the following holds. Let $\mathcal{C} = \mathcal{C}_0 + C_1 \in \mathcal{M}^{J_k}(\alpha,\beta)$ be a $U$-curve such that $\mathcal{A}(\alpha)\le T_q$. Let $\gamma\in\mathcal{P}_k$ and suppose that $C_1$ has at least one high multiplicity negative end asymptotic to a cover of $\gamma$. Let $m_0$ denote the multiplicity of the trivial cylinder $\R\times\gamma$ in $\mathcal{C}_0$. Then:
\begin{itemize}
\item[(a)]
If $C_1$ has a unique end asymptotic to a cover of $\gamma$, then
$T_\gamma(\mathcal{C}) \ge 4$, with equality only if $m_0=0$.
\item[(b)]
If $C_1$ has at least two ends asymptotic to covers of $\gamma$, then $T_\gamma(\mathcal{C}) \ge 7$, with equality only if $C_1$ has exactly one positive and one negative end asymptotic to covers of $\gamma$, the covering multiplicity for the positive end is one, and $m_0=0$.
\end{itemize}
\end{lem}

\begin{proof}
(a) Suppose that $C_1$ has exactly one end asymptotic to a cover of $\gamma$.

{\em Case 0:} Suppose that $m_0=0$. We need to show that $T_\gamma(\mathcal{C})\ge 4$.

In this case $e_\gamma(\mathcal{C})=1$, so we need to show that $S_\gamma(\mathcal{C})\ge 1$. By Proposition~\ref{prop:partitionconditions}, we have $p_\gamma^-(m(\gamma,\beta))=(m(\gamma,\beta))$, so $p^-(\gamma,\beta)=1$. By Lemma~\ref{lem:partitionfacts}(b), we have $p^+(\gamma,\beta)=0$. By Lemma~\ref{lem:partitionfacts}(a), we have $s(\gamma,\beta)=0$. Hence $S_\gamma(\mathcal{C})=1$.

{\em Case 1:} Suppose that $m_0=1$. We need to show that $T_\gamma(\mathcal{C})\ge 5$.

In this case $e_\gamma(\mathcal{C})=2$, so we need to show that $S_\gamma(\mathcal{C}) \ge -1$. We have $p^-(\gamma,\alpha)=0$ by definition, because $m(\gamma,\alpha)=1$. By Proposition~\ref{prop:partitionconditions}, we have $p_\gamma^-(m(\gamma,\beta)) = (m(\gamma,\beta)-1,1)$. If $k\ge k_{T_q}''$ then it follows that $p^+(\gamma,\beta)=0$, or else the pair $(\beta,m(\gamma,\beta))$ would be exceptional, contradicting Lemma~\ref{lem:exceptional}. Hence there is only one possible negative term in equation \eqref{eqn:Sgamma}.

{\em Case 2:} Suppose that $m_0>1$. Again we need to show that $T_\gamma(\mathcal{C})\ge 5$. This is proved by a symmetric argument to Case 2 of the proof of Lemma~\ref{lem:high}(a).

(b) Suppose that $C_1$ has at least two ends asymptotic to covers of $\gamma$. As in the proof of Lemma~\ref{lem:high}(b), we can assume that $m_0=0$ and that $C_1$ has exactly two ends asymptotic to covers of $\gamma$.

{\em Case 1:} Suppose that $C_1$ has two negative ends asymptotic to covers of $\gamma$. We need to show that $T_\gamma(\mathcal{C})\ge 8$.

In this case $e_\gamma(\mathcal{C})=3$, so we need to show that $S_\gamma(\mathcal{C})\ge -1$. This holds if $k\ge k_{T_q}''$, because $p^+(\gamma,\beta)=0$, because otherwise the pair $(\gamma,m(\gamma,\beta))$ would be exceptional, contradicting Lemma~\ref{lem:exceptional}.

{\em Case 2:} Suppose that $C_1$ has exactly one positive and one negative end asymptotic to covers of $\gamma$. We need to show that $T_\gamma(\mathcal{C})\ge 7$, with equality only if $m(\gamma,\alpha)=1$.

In this case $e_\gamma(\mathcal{C})=2$, so we need to show that $S_\gamma(\mathcal{C})\ge 1$, with equality only if $m(\gamma,\alpha)=1$. By Proposition~\ref{prop:partitionconditions}, $p_\gamma^+(m(\gamma,\alpha))=(m(\gamma,\alpha))$ and $p_\gamma^-(m(\gamma,\beta))=(m(\gamma,\beta))$. Thus $\gamma$ is a $p^-$-component of $\beta$, and $\gamma$ is a special component of $\alpha$ if $m(\gamma,\alpha)>1$. By Lemma~\ref{lem:partitionfacts}(b), $\gamma$ is not a $p^+$-component of $\beta$ or a special component of $\beta$ or a $p^-$-component of $\alpha$. We are now done by equation \eqref{eqn:Sgamma}.
\end{proof}


\subsubsection{Nonnegativity of the score}

The following proposition is our main interest in the score: it asserts that low energy curves in our $U$-sequence have nonnegative score, and in fact positive score except in a few cases that can be classified.

\begin{lemma}
\label{lem:key}
If $k$ is sufficiently large then the following holds. Let $\mathcal{C} = \mathcal{C}_0 + C_1 \in \mathcal{M}^{J_k}(\alpha,\beta)$ be a $U$-curve such that $\mathcal{A}(\alpha)\le T_q$. Assume that:
\begin{itemize}
\item $\mathcal{A}(\alpha)-\mathcal{A}(\beta)<\epsilon'$, where the constant $\epsilon'$ was specified in \S\ref{sec:constants}.
\item
$C_1$ is not a cylinder.
\end{itemize}
Then:
\begin{itemize}
\item[(a)]
The total score satisfies $T(\mathcal{C}) \geq 0$.
\item[(b)]
If $T(\mathcal{C}) = 0$, then:
\begin{itemize}
\item[(1)]
$C_1$ has genus zero and exactly three ends.
\item[(2)]
There are unique, distinct simple Reeb orbits $\gamma_1,\gamma_2\in\mathcal{P}_k$ such that $C_1$ has a high multiplicity positive end asymptotic to a cover of $\gamma_1$ and a high multiplicity negative end asymptotic to a cover of $\gamma_2$.
\item[(3)]
Either $C_1$ has an additional positive end asymptotic to a simple Reeb orbit distinct from $\gamma_1$, or $C_1$ has an additional negative end asymptotic to a cover of a simple Reeb orbit distinct from $\gamma_2$.
\item[(4)]
If a trivial cylinder $\R\times\gamma$ appears in $\mathcal{C}_0$, and if $C_1$ has an end asymptotic to a cover of $\gamma$, then $\gamma=\gamma_1$, and the multiplicity of $\R\times\gamma_1$ in $\mathcal{C}_0$ is at least $2$.
\item[(5)]
$J_0(\mathcal{C})=2$ if $\R\times\gamma_1$ appears in $\mathcal{C}_0$, and $J_0(\mathcal{C})=1$ otherwise.
\end{itemize}
\end{itemize}
\end{lemma}

\begin{proof}
Let $\alpha(1)$ denote the positive orbit set of $C_1$, and let $\beta(1)$ denote the negative orbit set of $C_1$. Note that $C_1\in\mathcal{M}^{J_k}(\alpha(1),\beta(1))$ is also a $U$-curve.

Suppose first that $C_1$ has ends asymptotic to covers of at least $5$ distinct simple Reeb orbits of $\lambda_k$. Then it follows from Remark~\ref{rem:Tsum} and Lemma~\ref{lem:emp}(a) that $T(\mathcal{C})\ge 3$, so the proposition is true in this case.

It remains to consider the case where $C_1$ has ends asymptotic to covers of at most $4$ distinct simple Reeb orbits, i.e.\ $|\alpha(1)|+|\beta(1)|\le 4$. By Lemma~\ref{lem:nonloc}, if $k$ is sufficiently large then the inequality \eqref{eqn:nonloc} holds.  Assume that $k$ is sufficiently large in this sense and also sufficiently large so that Lemma~\ref{lem:threshold} applies.
Then by Lemma~\ref{lem:threshold}, there are distinct simple Reeb orbits $\gamma_1,\gamma_2\in\mathcal{P}_k$ such that $C_1$ has a high multiplicity positive end asymptotic to a cover of $\gamma_1$ and a high multiplicity negative end asymptotic to a cover of $\gamma_2$. We now consider two cases, depending on whether or not $C_1$ has ends asymptotic to covers of any additional simple Reeb orbits.

{\em Case 1:} Suppose that $C_1$ has an end asymptotic to a cover of a simple Reeb orbit $\gamma_3$ distinct from $\gamma_1,\gamma_2$. Let $m_i$ denote the multiplicity of the trivial cylinder $\R\times\gamma_i$ in $\mathcal{C}_0$ for $i=1,2,3$. Then by Remark~\ref{rem:Tsum} and Lemmas~\ref{lem:emp}(a), \ref{lem:high}, and \ref{lem:low}, we have
\begin{equation}
\label{eqn:TCbound}
\begin{split}
T(\mathcal{C})
&\ge 6g(C_1) - 12 + T_{\gamma_1}(\mathcal{C}) + T_{\gamma_2}(\mathcal{C}) + T_{\gamma_3}(\mathcal{C})\\
&\ge 6g(C_1) - 12  + 5 + 4 + 3\\
& \ge 0.
\end{split}
\end{equation}
This proves (a). To prove (b), suppose that equality holds in \eqref{eqn:TCbound}. Then $g(C_1)=0$, the curve $C_1$ does not have ends asymptotic to covers of any orbits other than $\gamma_1$, $\gamma_2$, or $\gamma_3$, and we have $T_{\gamma_1}(\mathcal{C})=5$, $T_{\gamma_2}(\mathcal{C})=4$, and $T_{\gamma_3}(\mathcal{C})=3$. In particular, the end asymptotic to a cover of $\gamma_3$ is not high multiplicity. By Lemma~\ref{lem:high}, $C_1$ has no other ends asymptotic to covers of $\gamma_1$, and $m_1\neq 1$. By Lemma~\ref{lem:low}, $C_1$ has no other ends asymptotic to covers of $\gamma_2$, and $m_2=0$. By Lemma~\ref{lem:emp}, $m_3=0$, and $C_1$ has only one end asymptotic to a cover of $\gamma_3$. Moreover, if this end is a positive end, then the covering multiplicity is $1$. Thus assertions (1)--(4) in the proposition hold. Assertion (5) follows from equation \eqref{eqn:j0c}.

{\em Case 2:} Suppose that $C_1$ does not have ends asymptotic to covers of any simple orbits other than $\gamma_1$ and $\gamma_2$. In this case, similarly to \eqref{eqn:TCbound}, we have
\[
T(\mathcal{C}) \ge 6g(C_1) - 3.
\]
We can assume that $C_1$ has genus $0$, or else $T(\mathcal{C})>0$ and we are done. Thus by Remark~\ref{rem:Tsum} we have
\begin{equation}
\label{eqn:TCbound2}
T(\mathcal{C}) = -12 + T_{\gamma_1}(\mathcal{C}) + T_{\gamma_2}(\mathcal{C}).
\end{equation}
Since we assumed that $C_1$ is not a cylinder, it must have at least two ends asymptotic to covers of $\gamma_1$ or at least two ends asymptotic to covers of $\gamma_2$.

Suppose first that $C_1$ has at least two ends asymptotic to covers of $\gamma_1$. Then by equation  \eqref{eqn:TCbound} and Lemmas~\ref{lem:high} and \ref{lem:low}, we have $T(\mathcal{C})\ge 0$, with equality only if $T_{\gamma_1}(\mathcal{C})=8$ and $T_{\gamma_2}(\mathcal{C})=4$. If the latter equalities hold, then $C_1$ has exactly one positive end and one negative end asymptotic to covers of $\gamma_1$, the curve $C_1$ has only one end asymptotic to a cover of $\gamma_2$, and $\mathcal{C}_0$ does not include the trivial cylinders $\R\times\gamma_1$ or $\R\times\gamma_2$. Thus $J_0(\mathcal{C})=1$ by equation \eqref{eqn:j0c}.

Finally suppose that $C_1$ has at least two ends asymptotic to covers of $\gamma_2$. Then by equation \eqref{eqn:TCbound} and Lemmas~\ref{lem:high} and \ref{lem:low}, we have $T(\mathcal{C})\ge 0$, with equality only if $T_{\gamma_1}(\mathcal{C})=5$ and $T_{\gamma_2}(\mathcal{C})=7$. If the latter equalities hold, then $C_1$ has only one end asymptotic to a cover of $\gamma_1$, and $m_1\neq 1$; moreover $C_1$ has exactly one positive and one negative end asymptotic to covers of $\gamma_2$, the covering multiplicity for the positive end is $1$, and $m_2=0$. Thus assertions (1)--(4) hold, and assertion (5) follows from \eqref{eqn:j0c} as before.
\end{proof}

 
 \subsection{Putting it all together}
 \label{sec:totality}
 
 We can now complete the proof of Proposition~\ref{prop:main}.
 
\subsubsection{Finding a low energy cylinder}
\label{sec:puttogether}

We first prove a lemma which provides the cylinders claimed by Proposition~\ref{prop:main}. Assume that $k$ is sufficiently large so that Lemma~\ref{lem:ucurves} is applicable, and let $\mathcal{C}(i)\in\mathcal{M}^{J_k}(\alpha(i),\alpha(i-1))$ for $p_0<i\le q$ be a chain of $U$-curves provided by Lemma~\ref{lem:ucurves}. Below we say that $\mathcal{C}(i)$ is {\bf low energy\/} if
\[
\mathcal{A}(\alpha(i)) - \mathcal{A}(\alpha(i-1)) < \epsilon';
\]
otherwise we say that $\mathcal{C}(i)$ is {\bf high energy\/}. Also, if we write a $U$-curve $\mathcal{C}$ in the usual manner as $\mathcal{C}_0+C_1$, then we refer to $C_1$ as the {\bf nontrivial component\/} of $\mathcal{C}$.

\begin{lemma}
\label{lem:lowenergycylinder}
If $k$ is sufficiently large, then there exists $i\in\{p_0+1,\ldots,q\}$ such that the $U$-curve $\mathcal{C}(i)$ is low energy, and the nontrivial component of $\mathcal{C}(i)$ is a cylinder.
\end{lemma}

\begin{proof}
Suppose to get a contradiction that there is no $i$ with $p_0<i\le q$ such that $\mathcal{C}(i)$ is low energy and the nontrivial component of $\mathcal{C}(i)$ is a cylinder. By this hypothesis and Lemma~\ref{lem:key}, we can decompose
\begin{equation}
\label{eqn:G0123}
\{p_0+1,\ldots,q\} = G_0 \sqcup G_1 \sqcup G_2 \sqcup G_3
\end{equation}
where
\[
\begin{split}
G_0 &= \{i \mid \text{$\mathcal{C}(i)$ is high energy}\},\\
G_1 &= \{i \mid \text{$\mathcal{C}(i)$ is low energy, $T(\mathcal{C}(i))=0$, $J_0(\mathcal{C}(i))=1$}\},\\
G_2 &= \{i \mid \text{$\mathcal{C}(i)$ is low energy, $T(\mathcal{C}(i))=0$, $J_0(\mathcal{C}(i))=2$}\},\\
G_3 &= \{i\mid \text{$\mathcal{C}(i)$ is low energy, $T(\mathcal{C}(i))>0$}\}.
\end{split}
\]
The strategy will be to put bounds on the cardinalities of the above four subsets, and to contradict \eqref{eqn:G0123} by showing that the total cardinality is too small. We proceed in five steps.

{\em Step 1.\/} We first put bounds on $|G_0|$.

Note that since $\mathcal{A}(\alpha(q))\le \delta_2q^{1/2}$ by \eqref{eqn:cqkbound}, we have
\begin{equation}
\label{number_of_bad_curves}
|G_0| \leq \delta_2 q^{1/2}/\epsilon' \, .
\end{equation}

Note also that for each $i$ we have the coarse bound 
\begin{equation}
\label{eqn:coarse}
T(\mathcal{C}(i)) \ge -9,
\end{equation}
by Remark~\ref{rem:Tsum} and Lemma~\ref{lem:emp}.
It follows from \eqref{eqn:coarse} and Lemma~\ref{lem:ucurves}(a) that
\begin{equation}
\label{contribution_bad}
\sum_{i \in G_0} T(\mathcal{C}(i)) \geq - 9 \delta_2 q^{1/2}/\epsilon' \, .
\end{equation}

{\em Step 2.\/} We now show that $|G_3| < q^{2/3}$.

To start, observe that
\begin{align}
\nonumber
\sum^{q}_{i=p_0+1} T(\mathcal{C}(i)) &= S(\alpha(q)) - S(\alpha(p_0))+ 3\sum^{q}_{i=p_0+1} y(\mathcal{C}(i)) \\
\nonumber
&= [S(\alpha(q)) - S(\alpha(p_0))] + 3 \left( -2 (q-p_0) + \sum^{q}_{i=p_0+1} J_0(\mathcal{C}(i)) \right)\\
\nonumber
&\le [S(\alpha(q)) - S(\alpha(p_0))] + 6 \delta_1 \delta_2 q^{1/2}
\\
\label{eqn:totalscorebound}
&\le 4 \delta_2 q^{1/2}/\ell +  6 \delta_1 \delta_2 q^{1/2}.
\end{align}
Here the first line uses a telescoping sum, the second line uses the fact that $y = -2 + J_0$, and the the third line uses the inequality \eqref{eqn:j0equation}. The fourth line holds because if $\alpha=\alpha(q)$ or $\alpha=\alpha(p_0)$ then we have $|S(\alpha))| \le 2 |\alpha|$ by definition, and $\mathcal{A}(\alpha) \ge | \alpha | \ell$; on the other hand, $\mathcal{A}(\alpha) \le \delta_2 q^{1/2}$ by \eqref{eqn:cqkbound}.

Now by definition, if $i\in G_3$ then $T(\mathcal{C}(i))\ge 1$, and if $i\in G_1\sqcup G_2$ then $T(\mathcal{C}(i))= 0$. Thus
\[
\begin{split}
|G_3| & \leq \sum_{i\in G_1\sqcup G_2 \sqcup G_3} T(\mathcal{C}(i)) \\
& = \sum_{i = p_0+1}^q T(\mathcal{C}(i)) - \sum_{i\in G_0} T(\mathcal{C}(i)) \\
&\leq 4 \delta_2 q^{1/2}/\ell +  6 \delta_1 \delta_2 q^{1/2} + 9 \delta_2 q^{1/2}/\epsilon' \\
& < q^{2/3}
\end{split}
\]
Here the second to last line uses \eqref{eqn:totalscorebound} and \eqref{contribution_bad}, while the last line uses \eqref{eqn:defnq1}.

{\em Step 3.\/} We now show that $|G_1| < q^{4/5}$.

To do so, we will use a slight variant of the total score. If $\mathcal{C}$ is a $U$-curve, define
\[
T'(\mathcal{C}) = 2 y(\mathcal{C}) + S(\mathcal{C}) = T(\mathcal{C}) - y(\mathcal{C}).
\]
If $i\in G_1$, then $T'(\mathcal{C}(i))= 1$. The reason is that by the definition of $G_1$, we have $T(\mathcal{C}(i))=0$ and $y(\mathcal{C}(i))=J_0(\mathcal{C}(i))-2=-1$.

To bound $T'$ of other $U$-curves, if $\gamma$ is a simple Reeb orbit of $\lambda_k$, and if $\mathcal{C}$ is a $U$-curve, then analogously to \eqref{eqn:Tgamma}, define
\[
T'_\gamma(\mathcal{C}) = S_\gamma(\mathcal{C}) + 2e_\gamma(\mathcal{C}).
\]
As in Remark~\ref{rem:Tsum}, we have
\begin{equation}
\label{eqn:T'sum}
T'(\mathcal{C}) = 4g(C_1) - 8 + \sum_{\gamma\in\mathcal{P}_k} T'_\gamma(\mathcal{C}).
\end{equation}
Moreover in the sum on the right, the contribution $T'_\gamma(\mathcal{C})$ is nonzero only if $C_1$ has an end asymptotic to a cover of $\gamma$. We now have the following variant of Lemma~\ref{lem:emp}(a):

\begin{lem}
\label{lem:emptp}
Let $\mathcal{C}=\mathcal{C}_0+C_1\in\mathcal{M}^{J_k}(\alpha,\beta)$ be a $U$-curve, and let $\gamma\in\mathcal{P}_k$. Suppose that $C_1$ has at least one end asymptotic to a cover of $\gamma$. Then $T_\gamma'(\mathcal{C}) \ge 1$.
\end{lem}

\begin{proof}
By hypothesis, $e_\gamma(\mathcal{C})>0$. If $e_\gamma(\mathcal{C}) \ge 2$, then the lemma is true automatically since $S_\gamma(\mathcal{C}) \geq -3$ by equation \eqref{eqn:Sgamma}.   It therefore remains to consider the case where $e_\gamma(\mathcal{C})=1$. In this case $\mathcal{C}_0$ does not include the trivial cylinder $\R\times\gamma$, and $C_1$ has exactly one end asymptotic to a cover of $\gamma$.

If this end is a positive end, then $\gamma$ is not a $p^-$-component of $\alpha$. If $m(\gamma,\alpha)=1$, this is true by definition since $M>1$. If $m(\gamma,\alpha)>1$, then by Proposition~\ref{prop:partitionconditions}, $p^+_{\gamma}(m(\gamma,\alpha))=(m(\gamma,\alpha))$, and also $\gamma$ has irrational rotation number since $\alpha$ is an ECH generator. Then by Lemma~\ref{lem:partitionfacts}(b), $p_\gamma^-(m(\gamma,\alpha))\neq (m(\gamma,\alpha))$, so $\gamma$ is not a $p^-$-component of $\alpha$. Since $\gamma$ is not a $p^-$-component of $\alpha$, it follows from equation \eqref{eqn:Sgamma} that $S_\gamma(\mathcal{C})\ge 0$, so $T'_\gamma(\mathcal{C})=S_\gamma(\mathcal{C})+2 \ge 2$.

Suppose now that the end of $C_1$ asymptotic to a cover of $\gamma_i$ is negative. As in the previous paragraph, $\gamma$ is not a $p^+$-component of $\beta$. It then follows from equation \eqref{eqn:Sgamma} that $S_\gamma(\mathcal{C})\ge -1$, so $T'_\gamma(\mathcal{C})\ge 1$.
\end{proof}

It follows from equation \eqref{eqn:T'sum} and the above lemma that every $U$-curve $\mathcal{C}$ satisfies $T'(\mathcal{C}) \ge -7$.

We can now finish the proof that $|G_1| < q^{4/5}$. Note that similarly to \eqref{eqn:totalscorebound}, we have
\begin{equation}
\label{eqn:totalT'bound}
\sum_{i=p_0+1}^q T'(\mathcal{C}(i)) \le 4\delta_2q^{1/2}/\ell + 4\delta_1\delta_2q^{1/2}.
\end{equation}
Now for $i\in G_2$, we have $T(\mathcal{C}(i))=y(\mathcal{C}(i))=0$ by definition, so $T'(\mathcal{C}(i))=0$. Thus
\[
\begin{split}
|G_1|
&\leq \sum_{i \in G_1\sqcup G_2} T'(\mathcal{C}(i))\\
&= \sum^{q}_{i = p_0+1} T'(\mathcal{C}(i)) - \sum_{i\in G_0} T'(\mathcal{C}(i)) - \sum_{i\in G_3} T'(\mathcal{C}(i)) \\
&\leq 4 \delta_2 q^{1/2}/\ell +  4 \delta_1 \delta_2 q^{1/2} + 7 \left(\delta_2q^{1/2}/\epsilon' + q^{2/3}\right)\\
&< q^{4/5}.
\end{split}
\]
Here the second to last line uses \eqref{number_of_bad_curves} and Step 2, while the last line uses the inequality \eqref{eqn:defnq2}.

{\em Step 4.\/} We now show that $|G_2| < q^{5/6}$. 

If $\alpha$ is an orbit set, define $K(\alpha)\le 0$ to be minus the number of simple Reeb orbits $\gamma$ such that $m(\gamma,\alpha)>1$. If $\mathcal{C}\in\mathcal{M}^{J_k}(\alpha,\beta)$ is a $U$-curve, define
\[
K(\mathcal{C}) = K(\alpha) - K(\beta) + 2y(\mathcal{C}).
\]

\begin{lem}
\label{lem:kA}
If $i\in G_2$, then $K(\mathcal{C}(i))\ge 1$.
\end{lem}

\begin{proof}
Write $\mathcal{C}=\mathcal{C}(i)\in\mathcal{M}^{J_k}(\alpha,\beta)$ where $\alpha=\alpha(i)$ and $\beta=\alpha(i-1)$. Since $y(\mathcal{C})=0$, we need to show that $K(\alpha)>K(\beta)$. We will adopt the notation from Lemma~\ref{lem:key}. By Lemma~\ref{lem:key}(b), the trivial cylinder $\R\times\gamma_1$ appears with multiplicity at least $2$ in $\mathcal{C}_0$, while there are no other trivial cylinders $\R\times\gamma$ in $\mathcal{C}_0$ such that $C_1$ has an end asymptotic to a cover of $\gamma$. Thus we can remove any trivial cylinders other than $\R\times\gamma_1$ from $\mathcal{C}_0$ without changing $K(\mathcal{C})$, and we will assume that this has been done. There are now two cases to consider.

The first case is where $C_1$ has two positive ends and one negative end. Then $m(\gamma_1,\alpha)>1$, while the other positive end is asymptotic to a simple Reeb orbit by Lemma~\ref{lem:key}(b)(3), so $K(\alpha)=-1$. On the other hand, since the trivial cylinder $\R\times\gamma_1$ appears with multiplicity at least $2$ in $\mathcal{C}_0$, and the negative end of $C_1$ asymptotic to a cover of $\gamma_2\neq \gamma_1$ is high multiplicity, we have $K(\beta)=-2$.

The second case is where $C_1$ has one positive end and two negative ends. Then $\gamma_1$ is the only simple Reeb orbit that appears in $\alpha$, so $K(\alpha)=-1$. On the other hand the distinct simple Reeb orbits $\gamma_1,\gamma_2$ both appear in $\beta$ with multiplicity greater than one, so $K(\beta)\in\{-2,-3\}$.
\end{proof}

We can now complete the proof that $|G_2|\le q^{5/6}$. To start, similarly to \eqref{eqn:totalscorebound}, we have
\begin{equation}
\label{eqn:totalKbound}
\sum_{i = p_0+1}^q K(\mathcal{C}(i)) \le 2\delta_2 q^{1/2}/\ell + 4 \delta_1 \delta_2 q^{1/2}.
\end{equation}
Also observe that for every $i$, we have $K(\mathcal{C}(i)) \ge -8$ by equation \eqref{eqn:j0c}. Thus we have
\[
\begin{split}
|G_2| &\le \sum_{i\in G_2}K(\mathcal{C}(i)) \\
&= \sum_{i=p_0+1}^q K(\mathcal{C}(i)) - \sum_{i\in G_0\sqcup G_1\sqcup G_3}K(\mathcal{C}(i))\\
&\le 2\delta_2 q^{1/2}/\ell + 4 \delta_1 \delta_2 q^{1/2} +8\left(\delta_2q^{1/2}/\epsilon' + q^{4/5} + q^{2/3}\right)\\
& < q^{5/6}.
\end{split}
\]
Here the second to last line uses Steps 1, 2, and 3, while the last line uses the inequality \eqref{eqn:defnq3}. 

{\em Step 5.\/} We now complete the proof of Lemma~\ref{lem:lowenergycylinder}. By the decomposition \eqref{eqn:G0123} and Steps 1--4, we have
\[
q - p_0 < \delta_2q^{1/2}/\epsilon' + q^{4/5} + q^{5/6} + q^{2/3}.
\]
This contradicts the inequality \eqref{eqn:defnq4}.
\end{proof}

\subsubsection{Conclusion}

\begin{proof}[Proof of Proposition~\ref{prop:main}.]
Continue with the setup from \S\ref{sec:setup}.

By Lemma~\ref{lem:lowenergycylinder}, if $k$ is sufficiently large (we pass to a subsequence so that this is true for all $k$), there exists a $U$-curve for $(\lambda_k,J_k)$ such that the positive orbit set has action $\le T_q$, and the nontrivial component is a low energy cylinder. Choose such a $U$-curve for each $k$, and denote the nontrivial component by $C_k$. Let $\alpha_k$ and $\beta_k$ denote the Reeb orbits for $\lambda_k$ to which the positive and negative ends of $C_k$ respectively are asymptotic.

Similarly to \S\ref{sec:setup}, it follows from Lemma~\ref{lem:pert} that for each $k$, there are unique Reeb orbits $\underline{\alpha_k}$ and $\underline{\beta_k}$ for $\lambda$ that are homotopic to $\alpha_k$ and $\beta_k$ respectively in $N$. Since there is a uniform upper bound on the action of $\alpha_k$ and $\beta_k$, there are only finitely many options for $\alpha_k$ and $\beta_k$. We can then pass to a subsequence so that there are Reeb orbits $\alpha_\infty$ and $\beta_\infty$ for $\lambda$ such that $\underline{\alpha_k}=\alpha_\infty$ and $\underline{\beta_k}=\beta_\infty$ for each $k$. Then the sequence $((\lambda_k,J_k,\alpha_k,\beta_k,C_k))$ satisfies all of the bullet points in Proposition~\ref{prop:main}, as well as assertions (a) and (b). To complete the proof of Proposition~\ref{prop:main}, we now prove assertions (c) and (d).

To prove assertion (c) in Proposition~\ref{prop:main}, note that by Lemma~\ref{lem:nonloc}, if $k$ is sufficiently large (we pass to a subsequence so that this is true for all $k$), then
\begin{equation}
\label{eqn:hbar}
\mathcal{A}(\alpha_k)-\mathcal{A}(\beta_k)> \overline{h}.
\end{equation}
It then follows from Lemma~\ref{lem:threshold} that the Reeb orbits $\alpha_k$ and $\beta_k$ have multiplicity $\ge M\ge 3$. In particular, since $C_k$ together with some trivial cylinders goes between ECH generators, the orbits $\alpha_k$ and $\beta_k$ are elliptic and so their rotation numbers are irrational.

Let $m$ denote the covering multiplicity of the Reeb orbit $\alpha_k$, let $\gamma$ denote the underlying simple Reeb orbit, and let $\theta\in\R/\Z$ denote the rotation number of $\gamma$, so that the rotation number of $\alpha_k$ is $m\theta$. By Proposition~\ref{prop:partitionconditions}, we have $p_\gamma^+(m)=(m)$. Now suppose that $\alpha_\infty$ is degenerate; then its rotation number must be $0\in\R/\Z$. Thus the rotation number of the simple orbit underlying $\alpha_\infty$ is a rational number, which we write in lowest terms as $a/b$, and $m$ is an integer multiple of $b$. Choose $\varepsilon>0$ sufficiently small that Example~\ref{example:nearlyrationalangle} is applicable, and also sufficiently small so that $m\varepsilon < 1/2$. By Lemma~\ref{lem:pert}, if $k$ is sufficiently large (we pass to a subsequence so that this is true for all $k$), then $|\theta-a/b|<\varepsilon$. By Example~\ref{example:nearlyrationalangle}, if $\theta-a/b\in(0,\varepsilon)$, then $p_\gamma^+(m)=(m/b,\ldots,m/b)$. Since $m\ge M$ and we assumed in \S\ref{sec:constants} that $M>B_0$, we have $m>b$, so the partition $p_\gamma^+(m)$ has at least two elements. This contradicts the fact that $p_\gamma^+(m)=(m)$. Thus we must have $\theta-a/b\in(-\varepsilon,0)$, and it follows that the rotation number of $\alpha_k$ is $m\theta\in (-1/2,0)$.

A symmetric argument shows that if $\beta_\infty$ is degenerate, then if $k$ is sufficiently large (we pass to a subsequence so that this is true for all $k$), then the rotation number of $\beta_k$ is in the interval $(0,1/2)$.

Finally, to prove assertion (d) in Proposition~\ref{prop:main}, we will invoke Proposition~\ref{prop:gsscriterion}. We already know that $C_1$ satisfies all of the criteria in Proposition~\ref{prop:gsscriterion} except possibly for (e), which we still need to show. That is, we need to show that the component of the moduli space $\M^{J_k}(\alpha_k,\beta_k)/\R$ containing $C_1$ is compact. To do so, we use a variant of the proof of \cite[Lem.\ 4.10(c)]{CGHP}.

Suppose to get a contradiction that this moduli space is not compact. Then there must exist a sequence of elements of the moduli space that converges to a $J_k$-holomorphic building from $\alpha_k$ to $\beta_k$ with at least two levels. No level of the building can include a holomorphic plane, because the minimum action of a Reeb orbit is greater than $\epsilon'$ which is greater than $\mathcal{A}(\alpha)-\mathcal{A}(\beta)$. Hence every level of the building is a cylinder. By Proposition~\ref{prop:lowindexcurves}, there are exactly two levels, and each level has ECH index $1$. Let $\gamma$ denote the intermediate Reeb orbit. Since $\alpha_k$ and $\beta_k$ are elliptic, it follows from the ECH index parity conditions discussed in \S\ref{sec:defECH} that $\gamma$ is a cover of a positive hyperbolic orbit. By the partition conditions in Proposition~\ref{prop:partitionconditions} and Example~\ref{example:smallangle}, the Reeb orbit $\gamma$ must be simple. By the inequality \eqref{eqn:hbar}, at least one of the two cylinders in the building must have the property that the action difference between the Reeb orbits at the positive and negative ends is greater than $\overline{h}/2$. It then follows from Lemma~\ref{lem:threshold} that $\mathcal{A}(\alpha)-\mathcal{A}(\beta)\ge \epsilon'$, which contradicts the fact that the cylinder $C_k$ is low energy.
\end{proof}

\section{Preliminaries about holomorphic curves in symplectizations}
\label{sec_prelim_2}

Equipped with our Birkhoff sections for the nondegenerate perturbations $\lambda_n$, we now turn to the rest of the proof.
Before proceeding, we first explain the general idea behind our argument.

\subsubsection*{Outline of the argument}  We have a sequence of holomorphic cylinders $C_n$ that give a Birkhoff section for $\lambda_n$ and we would like to prove a two or infinity result in the limit.  Our basic idea for this is to obtain a compact moduli space of holomorphic cylinders in the limit.  To prove this, we first prove a general automatic transversality result that holds even for degenerate contact forms as long as the curves under consideration are asymptotic to 
Reeb orbits at their punctures and converge exponentially fast to them; we leverage a weighted Fredholm theory for this.  Next, we need to show that the $C_n$ give rise to curves with the required exponential convergence; this is where condition (c) in Proposition~\ref{prop:main} is used.
We also obtain a kind of weak version of the usual SFT compactness theorem that suffices for our purposes.  Thus, we are able to obtain a well-defined nonempty transverse moduli space of cylinders in the limit.  To show that it is compact, the basic idea is as follows: a nonconvergent sequence of cylinders would give rise to an orbit as an asymptotic limit with positive intersection number with the relative homology class of the projection of the $C_n$; thus, close to the limit we would obtain a loop in the projection of the image of $C_n$ with positive intersection number, and this contradicts the fact that each $C_n$ has an embedded projection to the 3-manifold.  Equipped with the necessary moduli space of cylinders in the limit, we then lift the Reeb flow to the moduli space via an evaluation map; we show that the lifted flow has an annular global cross-section, and so we are justified in applying a theorem of Franks to conclude that if the original Reeb flow had only finitely many orbits, this lifted flow would have none, hence the original flow has precisely two. 

\vspace{2 mm}

Executing the above sketch will take up the remainder of our paper.
To proceed, let $\lambda$ be a contact form on a closed $3$-manifold $Y$, with Reeb vector field $R$ and contact structure $\xi = \ker \lambda$.
We do not make nondegeneracy assumptions on $\lambda$.
Denote by $$ \pi_\lambda : TY \to \xi $$ the projection along $R$.

\subsection{Generalities}

We begin by setting notation and collecting some known results.
In what follows, it will be convenient 
to write 
Reeb orbits as equivalence classes of pairs $(x,T)$, where $x: \mathbb{R} \to Y$ is a periodic trajectory of $R$, and $T > 0$ is a period of $x$.  It will also be convenient to 
make a choice of marked point on the images of all simple Reeb orbits; 
this 
uniquely determines $x$ by requiring that $x(0)$ is the marked point.
We will also want to define two different Reeb orbits as being {\bf geometrically distinct} if their images do not intersect.

Given a Reeb orbit $\gamma=(x,T)$, the bundle $x(T\cdot)^*\xi \to \R/\Z$ will be denoted by $\xi_\gamma$ from now on.   
Let $J:\xi \to \xi$ be a $d\lambda$-compatible complex structure; from this point onwards, we will want to distinguish between the complex structure on $\xi$ and the corresponding $\lambda$-compatible almost complex structure on the symplectization, so we will denote the latter by $\tilde{J}$.
The associated {\bf asymptotic operator} of $\gamma=(x,T)$, acting on sections of $\xi_\gamma$, is defined~by $$ A_\gamma : \eta \mapsto -J(\nabla_t\eta - \nabla_\eta R) $$ where $\nabla$ is any choice of symmetric connection on $Y$, and $\nabla_t$ denotes the induced covariant derivative of vector fields along the curve $t\mapsto x(Tt)$.
Note that $A_\gamma$ is independent of choice of connection, and that $\gamma$ is nondegenerate precisely when $0$ is not an eigenvalue of $A_\gamma$.
Note also that $A_\gamma$ depends on~$J$ but we do not make this dependence explicit in the notation, for the sake of simplicity.
The operator $A_\gamma$ defines a self-adjoint unbounded operator that acts on $W^{1,2}$-sections of~$\xi_\gamma$ whose spectrum $\sigma(A_\gamma) \subset \R$ is discrete, consists only of eigenvalues, and accumulates only at $\pm\infty$. 
If $\nu \in \sigma(A_\gamma)$ then every eigenvector for $\nu$ is nowhere vanishing and has the same winding number in a given symplectic trivialization $\tau$ of $\xi_\gamma$.
We denote this winding number by $$ \wind_\tau(\nu) \in \Z \, . $$
It is useful to note the following.  First of all, there are two eigenvalues with a given winding number, counted with multiplicities.
Next,
\[\nu_1\leq \nu_2 \Rightarrow \wind_\tau(\nu_1) \leq \wind_\tau(\nu_2),\] 
and, moreover, any integer is the winding number of some eigenvalue.

Given $\delta \in \R$, we denote by $\nu^{<\delta}(\gamma)$ the largest eigenvalue $\nu$ satisfying $\nu < \delta$, and by $\nu^{\geq\delta}(\gamma)$ the smallest eigenvalue~$\nu$ satisfying $\nu \geq \delta$.
Set 
\begin{align*}
p^\delta(\gamma) = \wind_\tau(\nu^{\geq\delta}(\gamma)) -  \wind_\tau(\nu^{<\delta}(\gamma)) \in \{0,1\} \, .
\end{align*}
The {\bf $\delta$-weighted} Conley-Zehnder index of $\gamma$ in the trivialization $\tau$ is
\begin{equation}
\CZ^\delta_\tau(\gamma) = 2 \wind_\tau(\nu^{<\delta}(\gamma)) + p^\delta(\gamma) \, .
\end{equation}
When $\delta=0$ we simply write $\CZ_\tau(\gamma)$.
See~\cite[Section~3]{props2} for more details.
When $\delta=0$ and~$\gamma$ is nondegenerate, this agrees with~\eqref{eqn:CZtau}.

\begin{remark}[Martinet tubes]
\label{rmk_Martinet_tubes}
Consider a periodic orbit $\gamma=(x,T)$ with primitive period~$T_0$.
A {\bf Martinet tube} around $\gamma$ is a smooth diffeomorphism $\Psi : U \to \R/\Z \times B$, where $U$ is an open neighborhood of $x(\R)$ and $B \subset \C$ is an open ball centered at the origin, such that $\Psi(x(T_0t)) = (t,0)$ and $\Psi_*\lambda = f(\theta,x_1+ix_2)(d\theta + x_1dx_2)$, for some function $f$ satisfying $f|_{\R/\Z\times\{0\}} \equiv T_0$ and $df|_{\R/\Z\times\{0\}} \equiv 0$.
\end{remark}

From now on we fix a closed connected orientable surface $S$, and $$ \Gamma_+ = \{z_1^+,\dots,z_{m_+}^+\} \qquad \Gamma_- = \{z_1^-,\dots,z_{m_-}^-\} $$ disjoint ordered subsets of $S$.
Denote $\Gamma = \Gamma_+ \cup \Gamma_-$, $\dot S = S \setminus \Gamma$.
Once a conformal structure $j$ on $S$ is chosen, for any $z_* \in \Gamma$ and any choice of holomorphic embedding $\psi : (\D,i,0) \to (S,j,z_*)$ we call the map $$ (s,t) \in [0,+\infty) \times \R/\Z \mapsto \psi(e^{-2\pi(s+it)}) \in \psi(\D\setminus\{0\}) $$ {\bf positive holomorphic polar coordinates} centered at $z_*$.
Similarly, we call $$ (s,t) \in (-\infty,0] \times \R/\Z \mapsto \psi(e^{2\pi(s+it)}) \in \psi(\D\setminus\{0\}) $$ {\bf negative holomorphic polar coordinates} centered at $z_*$.

We 
sometimes call a $\jtil$-holomorphic map a
{\bf pseudo-holomorphic} map when there is no need to refer to~$\jtil$. 
Consider a conformal structure $j$ on $S$, and a nonconstant pseudo-holomorphic map 
\begin{equation}
\label{curve_reference}
\util = (a,u) : (\dot S,j) \to (\R \times Y,\jtil)
\end{equation}
with finite Hofer energy $$ 0 < E(\util) = \sup_{\phi} \int_{\dot S} \util^*d(\phi\lambda) < \infty \, . $$ 
The supremum is taken over all smooth $\phi:\R\to[0,1]$ satisfying $\phi'\geq0$.
It is proved in~\cite{Hofer93} that for every $z_*\in\Gamma$ exactly one of the following holds:
\begin{itemize}
\item $z_*$ is a positive puncture: $a(z) \to +\infty$ as $z\to z_*$.
\item $z_*$ is a negative puncture: $a(z) \to -\infty$ as $z\to z_*$.
\item $z_*$ is a removable puncture: $\util$ can be smoothly continued across $z_*$.
\end{itemize}
We proceed assuming that $\Gamma_+$ consists of positive punctures and $\Gamma_-$ consists of negative punctures.
In what follows we will refer to the sign $\epsilon = \pm1$ of a puncture according to the above classification, i.e. punctures in $\Gamma_+$ have sign $\epsilon=+1$ and punctures in $\Gamma_-$ have sign $\epsilon=-1$.
We also consider positive holomorphic polar coordinates centered at each positive puncture, and negative holomorphic polar coordinates centered at each negative puncture.
These sets of coordinates will be all denoted by $(s,t)$, with no fear of ambiguity.
All definitions and statements below are independent of these choices.
We write $\util(s,t)$ to denote the map $\util$ near a puncture written in holomorphic polar coordinates.

\begin{theorem}[Hofer~\cite{Hofer93}]
\label{thm_Hofer93}
Consider a puncture of sign $\epsilon$, and any sequence $s_n \to +\infty$.
There exist a subsequence $s_{n_k}$, a Reeb orbit $\gamma=(x,T)$ and $t_0\in\R/\Z$ such that $u(\epsilon s_{n_k},t+t_0) \to x(Tt)$ in $C^\infty(\R/\Z,Y)$ as $k\to\infty$.
\end{theorem}

\begin{definition}
\label{def_asymptotics}
Consider a puncture $z_*$ of sign $\epsilon$.
\begin{itemize}
\item We say that $\gamma=(x,T)$ is an {\bf asymptotic limit} of $\util$ at $z_*$ if we can find $s_n \to +\infty$ and $t_0 \in \R/\Z$ such that $u(\epsilon s_n,t+t_0) \to x(Tt)$ in $C^0(\R/\Z,Y)$ as $n\to\infty$.
\item We say that {\bf $\util$ is asymptotic to $\gamma=(x,T)$ at $z_*$} if we can find $t_0 \in \R/\Z$ such that $u(s,t+t_0) \to x(Tt)$ in $C^0(\R/\Z,Y)$ as $\epsilon s \to +\infty$.
\end{itemize}
\end{definition}

\begin{remark}
Asymptotic limits need not be unique, see~\cite{SiefringMathAnn}.
To say that every asymptotic limit of $\util$ at $z_*$ is equal to $\gamma$ is not the same as saying that $\util$ is asymptotic to $\gamma$ at $z_*$.
\end{remark}

For the next definition, assume that $\util$ is asymptotic to $\gamma=(x,T)$ at a puncture $z_*$.
Let $m \in \N$ be the covering multiplicity of $\gamma$, and $\epsilon$ be the sign of the puncture at~$z_*$.
Choose $\Psi:U \to \R/\Z\times B$ a Martinet tube around $\gamma$.
If~$\epsilon s$ is large enough then $u(s,t) \in U$ and we can write in components $$ \Psi \circ u(s,t) = (\theta(s,t),z(s,t)) \, . $$

\begin{definition}
\label{def_asymp_formula}
We say that {\bf$\util$ has an asymptotic formula at $z_*$} if either $z(s,t)$ vanishes identically, or there exist $b>0$, $\mu \in \sigma(A_\gamma)$ such that $\epsilon\mu < 0$ and the following holds:
\begin{itemize}
\item For some $a_0,t_0 \in \R$ and some lift $\tilde\theta:\R\times\R \to \R$ of $\theta(s,t)$ $$ \lim_{\epsilon s \to +\infty} \sup_{t\in\R/\Z} \ e^{b\epsilon s} \left( |D^\beta[a(s,t)-Ts-a_0]| + |D^\beta[\tilde\theta(s,t+t_0)-mt]| \right) = 0 $$ for every $\beta=(\beta_1,\beta_2) \in \N_0 \times \N_0$.
\item There is an eigenvector of $\mu$, represented as a function $v:\R/\Z \to \C$ in the frame $\{ \partial_{x_1},\partial_{x_2} \}$ of $\xi$ along $\gamma$, such that 
\begin{equation}
\label{asymptotic_formula}
z(s,t+t_0) = e^{\mu s} (v(t)+\Delta(s,t))
\end{equation}
for some $\Delta(s,t)$ satisfying $\sup_t|D^\beta \Delta(s,t)|\to 0$ as $\epsilon s \to +\infty$, for every multi-index $\beta=(\beta_1,\beta_2) \in \N_0 \times \N_0$.
\end{itemize}
The eigenvalue $\mu$ is called the {\bf asymptotic eigenvalue} of $\util$ at $z_*$, and the eigenvector in~\eqref{asymptotic_formula} is called the asymptotic eigenvector of $\util$ at $z_*$.
We say that $\util$ has a {\bf nontrivial asymptotic formula} at $z_*$ if it has an asymptotic formula at $z_*$ and $z(s,t)$ does not vanish identically.
\end{definition}

The above definition does not depend on the choices of holomorphic polar coordinates and Martinet tube.

\begin{theorem}[Hofer, Wysocki and Zehnder~\cite{props1}]
\label{thm_HWZ_asymptotics_nondeg_case}
If an asymptotic limit $\gamma$ of $\util$ at the puncture $z_*$ is a nondegenerate periodic orbit then $\util$ is asymptotic to $\gamma$ at $z_*$, and has an asymptotic formula at~$z_*$.
\end{theorem}

Choose any Riemannian metric on $Y$, and denote the associated exponential map by $\exp$.

\begin{definition}[\cite{fast}]
\label{def_nondeg_puncture}
Let $z_*$ be a puncture, and let $\epsilon$ be its sign.
We say that $z_*$ is a {\bf nondegenerate puncture} if $\util$ is asymptotic to a Reeb orbit $\gamma=(x,T)$ at $z_*$, i.e. $u(s,\cdot) \to x(T(\cdot+t_0))$ in $C^0$ as $\epsilon s \to +\infty$ for some $t_0$, and:
\begin{itemize}
\item There exists $a_0 \in \R$ such that $\sup_t|a(s,t) - Ts - a_0|$ as $\epsilon s\to+\infty$. 
\item If $\pi_\lambda \circ du$ does not vanish identically, then $\pi_\lambda \circ du(s,t) \neq 0$ for $\epsilon s \gg 1$.
\item Let $\zeta(s,t) \in T_{x(T(t+t_0))}Y$ be defined by $\exp(\zeta(s,t)) = u(s,t)$, $\epsilon s\gg1$. 
There exists $b>0$ such that $\sup_t e^{b\epsilon s}|\zeta(s,t)| \to 0$ as $\epsilon s\to+\infty$.
\end{itemize}
\end{definition}

We emphasize that the Reeb orbit associated to a nondegenerate puncture need not be nondegenerate.
Obviously, if $\util$ has an asymptotic formula at $z_*$ then $z_*$ is a nondegenerate puncture of~$\util$.
The converse is the content of the following statement.

\begin{theorem}[\cite{fast}, Corollary~6.6]
\label{thm_nondeg_implies_asymp_formula}
If $z_*$ is a nondegenerate puncture of $\util$, then $\util$ has an asymptotic formula at $z_*$.
\end{theorem}

Theeorem~\ref{thm_nondeg_implies_asymp_formula} is a consequence of arguments originally given by Hofer, Wysocki and Zehnder~\cite{props1} to prove Theorem~\ref{thm_HWZ_asymptotics_nondeg_case}.
The form of the asymptotic formula~\eqref{asymptotic_formula} presented above first appeared in~\cite{SiefringCPAM}.
It should be remarked that the above discussion has a version in higher dimensions.

\subsection{Weighted Fredholm theory and automatic transversality}
\label{sec_moduli_spaces}

In this section, we prove what we will need to know about the Fredholm theory for a pseudo-holomorphic curve with exponential convergence to Reeb orbits at its punctures; we do not assume that the Reeb orbits at these punctures are nondegenerate.    

We assume in this section that $\util$ is a pseudo-holomorphic map as in~\eqref{curve_reference} with an asymptotic formula at every puncture.
Assume also that $\pi_\lambda \circ du$ does not vanish identically on $\dot S$.
By connectedness of~$S$, the latter is equivalent to saying that $\util$ is not a (possibly branched) cover of a trivial cylinder. 
It follows from Carleman's similarity principle, and the fact that $\pi_\lambda \circ du$ satisfies a Cauchy-Riemann type equation in suitable local representations, that every zero of $\pi_\lambda \circ du$ is isolated and contributes positively to the total algebraic count.
By Theorem~\ref{thm_nondeg_implies_asymp_formula} and Definition~\ref{def_asymp_formula}, it follows that $\util$ has a nontrivial asymptotic formula at every puncture, and that there are only finitely many zeros of $\pi_\lambda \circ du$.
Denote by $\gamma^\pm_i$ the asymptotic limit of $\util$ at $z^\pm_i$, and by $\mu^\pm_i \in \sigma(A_{\gamma^\pm_i})$ the asymptotic eigenvalue of $\util$ at $z^\pm_i$.
Note that $\mu^+_i < 0$ and $\mu^-_i > 0$.

The invariants $\wind_\pi$ and $\wind_\infty$ were introduced in~\cite{props2}.
The invariant $\wind_\pi(\util)$ is defined as the total algebraic count of zeros of $\pi_\lambda \circ du$. 
In particular
\begin{equation}
\label{wind_pi_zero_transverse_immersion}
\wind_\pi(\util) = 0 \qquad \Rightarrow \qquad \text{$u$ is an immersion transverse to $R$.}
\end{equation}
A global symplectic trivialization $\tau$ of $u^*\xi$ induces symplectic trivializations $\tau^\pm_i$ on $\xi_{\gamma^\pm_i}$, up to homotopy. 
The invariant 
\begin{equation}
\wind_\infty(\util) = \sum_{i=1}^{m_+} \wind_{\tau^+_i}(\mu^+_i) - \sum_{i=1}^{m_-} \wind_{\tau^-_i}(\mu^-_i)
\end{equation}
is independent of the choice of $\tau$.
Standard degree theory, see~\cite{props2}, shows that 
\begin{equation}
\wind_\pi(\util) = \wind_\infty(\util) - 2 + 2g + \#\Gamma
\end{equation}
where $g$ is the genus of $S$.

Choose $\delta>0$ small enough such that
\begin{equation}
\label{delta_props}
[-\delta,0) \cap \sigma(A_{\gamma^+_i}) = \emptyset \ \forall i\in\{1,\dots,m^+\}, \qquad (0,\delta] \cap \sigma(A_{\gamma^-_i}) = \emptyset \ \forall i\in\{1,\dots,m^-\} \, .
\end{equation}
Hence $\mu^+_i < -\delta$ and $\mu^-_i > \delta$.
We consider a $\delta$-weighted Fredholm index defined as
\begin{equation}
\label{def_Fredholm_index}
\ind_\delta(\util) = - 2 + 2g + m_+ + m_- + \sum_{i=1}^{m_+} \CZ_{\tau^+_i}^{-\delta}(\gamma^+_i) - \sum_{i=1}^{m_-} \CZ_{\tau^-_i}^{\delta}(\gamma^-_i) \, . 
\end{equation}
Denote $\Gamma_+^{\rm odd} = \{ z^+_i \mid \CZ^{-\delta}_{\tau^+_i}(\gamma^+_i) \ \text{is odd} \}$, and $\Gamma_-^{\rm odd} = \{ z^-_i \mid \CZ^{\delta}_{\tau^-_i}(\gamma^-_i) \ \text{is odd} \}$. 
These sets are independent of the choice of $\delta$ with the properties above, and also independent of $\tau$.
A puncture in $\Gamma^{\rm odd} = \Gamma_+^{\rm odd} \cup \Gamma_-^{\rm odd}$ will be called an odd puncture.

\begin{lemma}
\label{lemma_ineq_wind_pi}
The inequality $$ 2\wind_\pi(\util) \leq \ind_\delta(\util) - 2 + 2g + \#\Gamma - \#\Gamma^{\rm odd} $$ holds. 
In particular, if $g=0$, $\ind_\delta(\util) = 2$ and $\#\Gamma \leq \#\Gamma^{\rm odd}+1$ then $\wind_\pi(\util) = 0$ and $u$ is an immersion transverse to the Reeb vector field.
Moreover, if $g=0$, $\ind_\delta(\util) = 2$ and $\#\Gamma = \#\Gamma^{\rm odd}$ then
$$
\begin{aligned}
& \sum_{i=1}^{m_+} \wind_{\tau_i^+}(\mu_i^+) = \sum_{i=1}^{m_+} \wind_{\tau_i^+}(\nu^{<-\delta}(\gamma^+_i)) \\
& \sum_{i=1}^{m_-} \wind_{\tau_i^-}(\mu_i^-) = \sum_{i=1}^{m_-} \wind_{\tau_i^-}(\nu^{\geq\delta}(\gamma^-_i))
\end{aligned}
$$
\end{lemma}

\begin{proof}
Note that the identities $$ \begin{aligned} & 2\wind_{\tau_i^-}(\mu^-_i) \geq 2 \wind_{\tau_i^-}(\nu^{\geq\delta}(\gamma_i^-)) = \CZ^\delta_{\tau_i^-}(\gamma_i^-) + p^\delta(\gamma_i^-) \\ & 2\wind_{\tau_i^+}(\mu^+_i) \leq 2 \wind_{\tau_i^-}(\nu^{<-\delta}(\gamma_i^+)) = \CZ^{-\delta}_{\tau_i^+}(\gamma_i^+) - p^{-\delta}(\gamma_i^+) \end{aligned} $$ follow from the definitions of the quantities involved, and from the basic properties of spectra of asymptotic operators described before.
It follows that
$$
\begin{aligned}
\sum_{i=1}^{m_+} 2\wind_{\tau_i^+}(\mu^+_i) & \leq - \#\Gamma^{\rm odd}_+ + \sum_{i=1}^{m_+} \CZ^{-\delta}_{\tau_i^+}(\gamma_i^+) \, , \\
- \sum_{i=1}^{m_-} 2\wind_{\tau_i^-}(\mu_i^-) & \leq -\#\Gamma^{\rm odd}_- - \sum_{i=1}^{m_-} \CZ^{\delta}_{\tau_i^-}(\gamma_i^-) \, .
\end{aligned}
$$
Using these identities we estimate:
\begin{equation}
\label{sandwich_wind_pi}
\begin{aligned}
& 2 \wind_\pi(\util) = 2\wind_\infty(\util) - 4 + 4g + 2\#\Gamma \\
& = \left[ \sum_{i=1}^{m_+} 2\wind_{\tau_i^+}(\mu^+_i) - \sum_{i=1}^{m_-} 2\wind_{\tau_i^-}(\mu^-_i) \right] - 4 - 4g + 2\#\Gamma \\
& \leq \left[ \sum_{i=1}^{m_+} \CZ^{-\delta}_{\tau_i^+}(\gamma_i^+) - \sum_{i=1}^{m_-} \CZ^{\delta}_{\tau_i^-}(\gamma_i^-) \right] - \#\Gamma^{\rm odd} - 4 + 4g + 2\#\Gamma \\
& = \ind_\delta(\util) + 2 - 2g - \#\Gamma - \#\Gamma^{\rm odd} - 4 + 4g + 2\#\Gamma \\
&= \ind_\delta(\util) - 2 + 2g + \#\Gamma - \#\Gamma^{\rm odd}
\end{aligned}
\end{equation}
as desired.
If conditions $g=0$, $\ind_\delta(\util)=2$ and $\#\Gamma \leq \#\Gamma^{\rm odd}+1$ hold then $\ind_\delta(\util) - 2 + 2g + \#\Gamma - \#\Gamma^{\rm odd} \leq 1$, hence $\wind_\pi(\util)=0$. 
The fact that $u$ is an immersion transverse to the Reeb vector field follows from~\eqref{wind_pi_zero_transverse_immersion}.
Finally, assume that $g=0$, $\ind_\delta(\util)=2$ and $\#\Gamma = \#\Gamma^{\rm odd}$.
The right-hand side in~\eqref{sandwich_wind_pi} vanishes and we get 
$$
\begin{aligned}
& \sum_{i=1}^{m_+} 2\wind_{\tau_i^+}(\mu^+_i) - \sum_{i=1}^{m_-} 2\wind_{\tau_i^-}(\mu^-_i) \\
& \qquad = \sum_{i=1}^{m_+} 2\wind_{\tau_i^+}(\nu^{<-\delta}(\gamma^+_i)) - \sum_{i=1}^{m_-} 2\wind_{\tau_i^-}(\nu^{\geq\delta}(\gamma^-_i))
\end{aligned}
$$
Dividing by two we arrive at
$$
\begin{aligned}
& \sum_{i=1}^{m_+} \wind_{\tau_i^+}(\mu^+_i) - \wind_{\tau_i^+}(\nu^{<-\delta}(\gamma^+_i)) \\
& \qquad = 
\sum_{i=1}^{m_-} \wind_{\tau_i^-}(\mu^-_i) - \wind_{\tau_i^-}(\nu^{\geq\delta}(\gamma^-_i))
\end{aligned}
$$
The left-hand side is nonpositive and the right-hand side is nonnegative.
Hence both sides are equal to zero.
\end{proof}

Next we discuss Fredholm theory.
Let $g$ be the genus of $S$.
Consider the moduli space
\begin{equation}
\M_{J,g}(\gamma_1^+,\dots,\gamma_{m_+}^+;\gamma_1^-,\dots,\gamma_{m_-}^-)
\end{equation}
consisting of equivalence classes of pairs $(\util,j)$, where $j$ is a conformal structure on $S$, and $\util:(\dot S,j)\to(\R\times Y,\jtil)$ is nonconstant pseudo-holomorphic map with finite Hofer energy, as in~\eqref{curve_reference}, and such that:
\begin{itemize}
\item For every $i \in \{1,\dots,m_+\}$, $z_i^+$ is a positive nondegenerate puncture where $\util$ is asymptotic to $\gamma^+_i$.
\item For every $i \in \{1,\dots,m_-\}$, $z_i^-$ is a negative nondegenerate puncture where $\util$ is asymptotic to $\gamma^-_i$.
\end{itemize}
Two such pairs $(\util=(a,u),j)$ and $(\util'=(a',u'),j')$ are declared equivalent if there exists a biholomorphism $\phi:(S,j)\to (S,j')$ that fixes each point of~$\Gamma$, and a constant $c\in\R$ such that $(c+a'\circ\phi,u'\circ\phi) = (a,u)$.
Once $g$, $\delta$, $\Gamma_\pm$ and $\gamma_i^\pm$ are all fixed, we write $\M_J$ to denote the above moduli space, for simplicity.
A curve $C \in \M_J$ is said to be embedded, immersed, or somewhere injective if it can be represented by a pair $(\util,j)$ where $\util$ is an embedding, an immersion, or a somewhere injective map, respectively.
The equivalence class of a pair $(\util,j)$ will be denoted by $[\util,j]$.

\begin{remark}
If $g=0$ and $m_++m_-\leq 3$ then we can disregard the conformal structure~$j$ since in this case, up to a diffeomorphism, there is only one conformal structure.
This becomes relevant in later sections where we will be dealing with cylinders. 
\end{remark}

If $C = [\util,j] \in \M_J$ is immersed then one can follow~\cite{props3} and build the functional analytic set-up for a Fredholm theory based on a certain space of sections of the normal bundle~$N_\util$ of~$\util$.
This Fredholm theory can then be used to model a neighborhood of $C$ in $\M_J$.
Let us describe the details we need.

Fix choices of positive/negative holomorphic polar coordinates $(s,t)$ centered at the positive/negative punctures.
Up to homotopy, we can assume that $N_\util$ is a $\jtil$-invariant bundle, and that 
\begin{equation}
\label{normal_bundle_ctctstr_near_puncts}
\text{$N_\util$ agrees with $u^*\xi$ near the punctures.} 
\end{equation}
Let $\tau_\xi$ be a global symplectic trivialization of $u^*\xi$, and $\tau_N$ be a global complex trivialization of $N_\util$.
Up to homotopy we can assume that at a given puncture $z_*$, of sign $\epsilon$, the trivializations~$\tau_\xi$ and~$\tau_N$ converge as $\epsilon s \to \infty$ to trivializations $\tau_\xi(\gamma^\pm_i)$ and $\tau_N(\gamma^\pm_i)$ of $\xi_{\gamma^\pm_i}$, respectively. 
We write $$ \wind(\tau_\xi(\gamma^\pm_i),\tau_N(\gamma^\pm_i)) \in \Z $$ to denote the winding number of a constant vector in $\tau_\xi(\gamma^\pm_i)$ represented in~$\tau_N(\gamma^\pm_i)$.
To compute this winding number the periodic orbit is traversed along the Reeb flow.
In other words, this is the winding number of $\tau_\xi(\gamma^\pm_i)$ relative to $\tau_N(\gamma^\pm_i)$.
The next lemma is proved in~\cite{props3}.

\begin{lemma}
\label{lemma_winding_ctctstr_normal}
Under the above assumptions, the identity
$$ \sum_{i=1}^{m_+} \wind(\tau_\xi(\gamma^+_i),\tau_N(\gamma^+_i)) - \sum_{i=1}^{m_-} \wind(\tau_\xi(\gamma^-_i),\tau_N(\gamma^-_i)) = - 2 + 2g + \#\Gamma $$
holds.
\end{lemma}

\begin{proof}
Near $z_i^+$ we consider positive holomorphic polar coordinates, and near~$z_i^-$ we consider negative holomorphic polar coordinates.
As usual, there is a sign $\epsilon = \pm1$ associated to each puncture $z^\pm_i$.
All these sets of coordinates are denoted by $(s,t)$ with no fear of ambiguity.
For $r\gg1$ large enough we denote by $K_r \subset \dot S$ the complement of the union of the sets described by $\{ (s,t) \mid \epsilon s > r\}$ near each puncture. 
Then $K_r$ is compact, and its boundary consists of circles.
The component $\partial_i^\pm K_r$ of $\partial K_r$ near $z_i^\pm$ is oriented by the parametrization $t \mapsto (\epsilon r,t)$.
Hence, if $\partial K_r$ is oriented as the boundary of $K_r$ we have $\partial K_r = \sum_{i=1}^{m_+} \partial_i^+ K_r - \sum_{i=1}^{m_-} \partial_i^- K_r$.

We see $\xi$ and $R$ as $\R$-invariant objects in $\R\times Y$.
By Definition~\ref{def_asymp_formula}, near each puncture the vector $\partial_s\util$ becomes arbitrarily close to $T\partial_a|_\util$, and the vector $\partial_t\util$ becomes arbitrarily close to $TR|_\util$, respectively, as $\epsilon s \to +\infty$.
Hence, if $r\gg1$ is large enough, there is a vector subbundle $E_\util \subset \util^*T(\R\times Y)$ that complements $u^*\xi$, agrees with ${\rm span}\{{\partial_a}|_\util,R|_\util\}$ on $K_{r-1}$, and agrees with $T\util$ on $\dot S \setminus K_r$.
Moreover, $E_\util$ has a global frame $\tau_E$ that agrees with $\{{\partial_a}|_\util,R|_\util\}$ on $K_{r-1}$, and with $\{ \partial_s\util , \partial_t\util \}$ on $\dot S \setminus K_{r}$.
By~\eqref{normal_bundle_ctctstr_near_puncts}, there is no loss of generality to assume that $N_\util$ agrees with $u^*\xi$ on $\dot S \setminus K_r$.
Hence, we have splittings $$ \util^*T(\R\times Y) = E_\util \oplus u^*\xi = T\util \oplus N_\util $$ whose corresponding summands coincide near the punctures.
Consider a global trivialization $\tau_T$ of $T\util$. 
Note that $$ \sum_{i=1}^{m_+} \wind_{\partial^+_iK_r}(\tau_E,\tau_T) - \sum_{i=1}^{m_-} \wind_{\partial^-_iK_r}(\tau_E,\tau_T) = 2 - 2g - \#\Gamma \, , $$ where $\wind_{\partial^\pm_iK_r}$ denotes the winding number along the curve $\partial^\pm_iK_r$ as oriented by $t$.
This is true because near each puncture the vector $\epsilon\partial_s\util$ points radially towards the puncture.

Recall the global trivializations $\tau_\xi$ and $\tau_N$ of $u^*\xi$ and of $N_\util$, respectively. 
Hence, we have two competing global trivializations $\tau_E \oplus \tau_\xi$ and $\tau_T \oplus \tau_N$ of $\util^*T(\R\times Y)$.
It follows that 
$$
\begin{aligned}
0 & = \left( \sum_{i=1}^{m_+} \wind_{\partial^+_iK_r}(\tau_E,\tau_T) + \sum_{i=1}^{m_+} \wind_{\partial^+_iK_r}(\tau_\xi,\tau_N) \right) \\
& \qquad - \left( \sum_{i=1}^{m_-} \wind_{\partial^-_iK_r}(\tau_E,\tau_T) + \sum_{i=1}^{m_-} \wind_{\partial^-_iK_r}(\tau_\xi,\tau_N) \right) \\
& = 2 - 2g - \#\Gamma + \left( \sum_{i=1}^{m_+} \wind_{\partial^+_iK_r}(\tau_\xi,\tau_N) - \sum_{i=1}^{m_-} \wind_{\partial^-_iK_r}(\tau_\xi,\tau_N) \right) \, .
\end{aligned}
$$
The desired conclusion follows from taking the limit as $r \to +\infty$. 
\end{proof}

To construct a weighted Fredholm theory, Hofer, Wysocki and Zehnder consider in~\cite{props3} a certain Banach space of sections of $N_\util$.
A section in this space can be represented near a puncture of sign~$\epsilon$, with respect to holomorphic polar coordinates and trivialization~$\tau_N$ as above, as a $\C$-valued function $(s,t) \mapsto w(s,t)$ that decays, together with its first derivatives, like $e^{-\epsilon\delta s}$ as $\epsilon s \to +\infty$.
A smooth Fredholm map is then defined on a neighborhood of zero of this Banach space, in such a way that exponentiating sections in its zero locus gives rise to images of curves in $\M_J$ near $C$ in a suitable topology.
The linearization at zero is called the linearized Cauchy-Riemann operator at $\util$, and will be denoted by $D_\util$.
Its Fredholm index is equal to $\ind_\delta(\util)$.
For precise definitions we refer to~\cite[Section~6]{props3}.

\begin{lemma}[Automatic transversality]
\label{lemma_automatic_transversality}
Assume that $g=0$, $\ind_\delta(\util)=2$ and $p^{\mp\delta}(\gamma^\pm_i) = 1$ for all $\gamma_i^\pm$.
Then $C$ is immersed and $D_\util$ is surjective.
\end{lemma}

\begin{proof}
We know from Lemma~\ref{lemma_ineq_wind_pi} that $C$ is immersed.
Hence, we can consider the functional analytic set-up from~\cite{props3} as outlined above.
Let $\zeta \in \ker D_\util$ be a section that does not vanish identically.
In~\cite{props3} it is shown that $D_\util$ is a Cauchy-Riemann type operator.
Moreover, it can be represented near a puncture $z^\pm_i$ using holomorphic polar coordinates $(s,t)$ as above, and the trivialization $\tau_N$, by $\partial_s + J(s,t)\partial_t+A(s,t)$ where $-J(s,t)\partial_t-A(s,t)$ converges as $\epsilon s\to+\infty$ to a representation of the asymptotic operator $A_{\gamma^\pm_i}$ in the frame $\tau_N(\gamma^\pm_i)$.
Hence, if we represent $\zeta$ as a $\C$-valued function $w(s,t)$ near $z^\pm_i$ using $(s,t)$ and $\tau_N$, then we can apply~\cite[Theorem~A.1]{SiefringCPAM} or~\cite[Theorem~6.1]{fast} to obtain an asymptotic formula $$ w(s,t) = e^{\kappa s}(v(t)+R(s,t)) \qquad (\epsilon s \sim +\infty) \, . $$
Here $\kappa \in \sigma(A_{\gamma^\pm_i})$ satisfies $\epsilon \kappa < 0$, and $v(t)$ is the representation in $\tau_N(\gamma^\pm_i)$ of a corresponding eigensection.
A crucial, but simple, remark at this point is that the inequalities below follow from the definition of $\delta$
$$
\begin{aligned}
& \epsilon = +1 \quad \Rightarrow \quad \wind_{\tau_N(\gamma^+_i)}(\kappa) \leq \wind_{\tau_N(\gamma^+_i)}(\nu^{<-\delta}(\gamma_i^+)) \\
& \epsilon = -1 \quad \Rightarrow \quad \wind_{\tau_N(\gamma^-_i)}(\kappa) \geq \wind_{\tau_N(\gamma^-_i)}(\nu^{\geq\delta}(\gamma_i^-)) \, .
\end{aligned}
$$
Standard degree theory and Lemma~\ref{lemma_winding_ctctstr_normal} together imply that the total algebraic count $\#\zeta^{-1}(0)$ of zeros of $\zeta$ satisfies 
$$ 
\begin{aligned}
& 2 \#\zeta^{-1}(0) \\
& \leq \sum_{i=1}^{m_+} 2\wind_{\tau_N(\gamma^+_i)}(\nu^{<-\delta}(\gamma_i^+)) - \sum_{i=1}^{m_+} 2\wind_{\tau_N(\gamma^-_i)}(\nu^{\geq\delta}(\gamma_i^-)) \\
& = - 4 + 2\#\Gamma + \sum_{i=1}^{m_+} 2\wind_{\tau_\xi(\gamma^+_i)}(\nu^{<-\delta}(\gamma_i^+)) - \sum_{i=1}^{m_+} 2\wind_{\tau_\xi(\gamma^-_i)}(\nu^{\geq\delta}(\gamma_i^-)) \\
&= - 4 + 2\#\Gamma - \#\Gamma^{\rm odd} + \sum_{i=1}^{m_+} \CZ_{\tau_\xi(\gamma^+_i)}^{-\delta} (\gamma_i^+) - \sum_{i=1}^{m_+} \CZ_{\tau_\xi(\gamma^-_i)}^{\delta}(\gamma_i^-) \\
&= -2 -2 + \#\Gamma + \sum_{i=1}^{m_+} \CZ_{\tau_\xi(\gamma^+_i)}^{-\delta}(\gamma_i^+) - \sum_{i=1}^{m_+} \CZ_{\tau_\xi(\gamma^-_i)}^{\delta}(\gamma_i^-) \\
&= -2 + \ind_\delta(\util) = 0 \, .
\end{aligned}
$$
Hence, sections in $\ker D_\util$ never vanish, or vanish identically.
Assume, by contradiction, that the cokernel of $D_\util$ is not zero.
Then there are at least three linearly independent sections in $\ker D_\util$ because the index of $D_\util$ is two.
Since $N_\util$ has rank two, this implies that a linear combination of them would vanish at some point.
But above we proved that the only way in which a section can vanish at some point is by vanishing identically.
This contradiction shows that $D_\util$ has a trivial cokernel or, in other words, that $D_\util$ is surjective.
\end{proof}

\section{Asymptotic analysis and compactness}
\label{sec_asymptotic_analysis}
 
In this section, we collect the information we glean from the particular kind of exponential convergence of each $C_n$.  More precisely, we extract a limiting cylinder with nondegenerate punctures, and apply some basic asymptotic analysis to establish some very useful facts about the image of its projection in the three-manifold; see Proposition~\ref{prop_main_asymptotic_analysis}.   We also prove a corresponding compactness statement.  Our arguments here are valid for any sequence of curves satisfying certain axioms, see (*) and (A)-(G) below, and we structure our proofs in this level of generality.

We draw freely from the notation and definitions established in Section~\ref{prelim_sec}.
Consider a sequence $g_n \in C^\infty(Y)$ such that $$ \text{$g_n \to 1$ in $C^\infty$}. $$ 
The Reeb vector fields of $\lambda_n = g_n\lambda$ and $\lambda$ are denoted by $R_n$ and $R$, respectively.
It should be stressed that we make no genericity assumptions on $\lambda_n$ or on $\lambda$.
Consider $d\lambda_n$-compatible, equivalently $d\lambda$-compatible, complex structures $J_n,J$ on $\xi$ satisfying $$ J_n\to J \ \text{in} \ C^\infty \, . $$
The induced $\R$-invariant almost complex structures on $\R\times Y$ induced by $\lambda_n,J_n$ and $\lambda,J$ are denoted by $\jtil_n$ and $\jtil$, respectively.

The cylinder $\R \times \R/\Z$ is equipped with coordinates $s\in\R$, $t\in\R/\Z$, and with its standard conformal structure obtained from the identification $\R\times\R/\Z \simeq \C \setminus \{0\}$, $(s,t) \simeq e^{2\pi(s+it)}$.
Assume that we are given nonconstant finite-energy pseudo-holomorphic maps
\begin{equation}
\label{seq_u_n}
\util_n : \R \times \R/\Z \to (\R\times Y,\jtil_n)
\end{equation}
and 
\begin{equation}
\label{limit_map}
\util : \R \times \R/\Z \to (\R\times Y,\jtil)
\end{equation}
with a positive puncture at $s=+\infty$ and a negative puncture at $s=-\infty$.
Assume that 
\begin{itemize}
\item[($*$)] $\util_n$ has nondegenerate punctures.
\end{itemize}
Let $\util_n$ be asymptotic to $\gamma^+_{n} = (x^+_{n},T^+_{n})$ at $s=+\infty$, and to $\gamma^-_{n} = (x^-_{n},T^-_{n})$ at $s=-\infty$, where $\gamma^\pm_n$ are Reeb orbits of $\lambda_n$. 
Assume further that that
\begin{itemize}
\item[(A)] $\util_n \to \util$ in $C^\infty_\loc(\R\times\R/\Z)$.

\item[(B)] There are Reeb orbits $\gamma^\pm = (x^\pm,T^\pm)$ of $\lambda$ such that $\gamma^\pm_{n} \to \gamma^\pm$ as $n\to\infty$.
More precisely, $x^\pm_{n}(T^\pm_{n}(\cdot+t^\pm_n)) \to x^\pm(T^\pm\cdot)$ in $C^\infty(\R/\Z,Y)$ as $n\to\infty$, for suitable choices of~$t^\pm_n \in \R/\Z$.

\item[(C)] If $c:\R\to Y$ is a periodic trajectory of $R$ with primitive period $\leq T^+$ then $g_n$ satisfies $g_n|_{c(\R)} \equiv 1$ and $dg_n|_{c(\R)} \equiv 0$.
In particular, $R_n$ coincides with $R$ on $c(\R)$, and~$c$ is also a periodic Reeb trajectory of~$\lambda_n$, for all $n$.

\item[(D)] The following hold for every $n$: 
\begin{itemize}
\item $u_n (\R \times \R/\Z) \subset Y \setminus (x^+_n(\R) \cup x^-_n(\R))$.
\item If $\gamma=(x,T)$ is a periodic $\lambda_n$-Reeb orbit in $Y \setminus (x^+_n(\R) \cup x^-_n(\R))$ with $T \leq T_+$ then the intersection number of the loop $x(T\cdot)$ with the map $u_n$ is strictly positive.
\item If~$\beta$ is a loop in $\R \times \R/\Z$ then $u_n \circ \beta$ has intersection number zero with~$u_n$.
\end{itemize}

\item[(E)] Let $T_{\min}(\lambda)>0$ be the minimal period of a $\lambda$-Reeb orbit. 
There exists $r>0$ such that $$ r \leq \int_{\R\times\R/\Z} u_n^*d\lambda_n \leq \frac{1}{2}T_{\min}(\lambda) \qquad \forall n \, . $$

\item[(F)] The point $u(0,0)$ does not belong to a $\lambda$-Reeb orbit.

\item[(G)] Denote asymptotic operators by $A^\pm_n = A_{\gamma^\pm_n}$, $A^\pm=A_{\gamma^\pm}$, and asymptotic eigenvalues of $\util_n$ at $s=\pm\infty$ by $\mu^\pm_n \in \sigma(A^\pm_n)$. We assume that there exists $\delta>0$ such that: $[-\delta,0) \cap \sigma(A^+) = \emptyset$, $(0,\delta] \cap \sigma(A^-) = \emptyset$, $\ind_\delta(\util_n)=2$, $\mu^+_n < -\delta$, $\mu^-_n > \delta$, $p^{-\delta}(\gamma_n^+) = p^\delta(\gamma_n^-) = 1$ $\forall n$.

\end{itemize}

We now state the asymptotic analysis result that is a main goal of this section.

\begin{proposition}
\label{prop_main_asymptotic_analysis}

Let $g_n,\lambda_n,\lambda,J_n,J,\util_n,\util$ be as above. Assume that ($*$) and (A)-(G) hold. Then:
\begin{itemize}
\item[(a)] $\util$ has a nontrivial asymptotic formula at both punctures, $\util$ is asymptotic to $\gamma^+$ at $s=+\infty$, and $\util$ is asymptotic to $\gamma^-$ at $s=-\infty$.
\item[(b)] $\ind_\delta(\util)=2$ and $p^{-\delta}(\gamma^+)=p^\delta(\gamma^-)=1$.
\item[(c)] The asymptotic eigenvalues $\mu^+$ and $\mu^-$ of $\util$ at $s=+\infty$ and at $s = -\infty$, respectively, satisfy $\mu^+<-\delta$, $\mu^->\delta$ and  
\begin{equation*}
\begin{aligned}
\wind(\mu^+) = \wind(\nu^{<-\delta}(\gamma^+)) \qquad \wind(\mu^-) = \wind(\nu^{\geq\delta}(\gamma^-))
\end{aligned}
\end{equation*}
in any trivialization of $\gamma^+$ and $\gamma^-$.
\item[(d)] The following holds: 
\begin{itemize}
\item $u (\R \times \R/\Z) \subset Y \setminus (x^+(\R) \cup x^-(\R))$.
\item If $\gamma=(x,T)$ is a $\lambda$-Reeb orbit in $Y \setminus (x^+(\R) \cup x^-(\R))$ such that $T \leq T^+$, then the intersection number of the loop $x(T\cdot)$ with the map $u$ is strictly positive.
\item If~$\beta$ is a loop in $\R \times \R/\Z$ then $u \circ \beta$ has intersection number zero with~$u$.
\end{itemize}
\end{itemize}
\end{proposition}

Proving the above proposition will take the bulk of this section, and we split the proof into several steps.  We start with a simple lemma.

\begin{lemma}
\label{lemma_zeros_pi_du}
$\pi_\lambda \circ du$ does not vanish.
\end{lemma}

\begin{proof}
From (E) we know that $\pi_\lambda \circ du$ does not vanish identically.
Hence zeros of $\pi_\lambda \circ du$ are isolated.
If there is a zero then standard degree theory allows one to find zeros of $\pi_{\lambda_n} \circ du_n$ for $n$ large enough.
But~(G) and Lemma~\ref{lemma_ineq_wind_pi} together imply that $\pi_{\lambda_n} \circ du_n$ does not vanish.
This contradiction shows that $\pi_\lambda \circ du$ does not vanish.
\end{proof}

\subsection{Localizing the ends}

Next, we prove the following lemma, which allows us to localize the argument, and we collect some consequences; this is the topic of this subsection.

\begin{lemma}
\label{lemma_localizing_the_ends}
For every $(\R/\Z)$-invariant open $C^\infty$-neighborhood $\mathcal{U}^+$ of $x^+(T^+\cdot)$ there exist $n_+,s_+$ such that for every $n \geq n_+$ and $s\geq s_+$ we have $u_n(s,\cdot) \in \mathcal{U}^+$.
Similarly, for every $(\R/\Z)$-invariant open $C^\infty$-neighborhood $\mathcal{U}^-$ of $x^-(T^-\cdot)$ there exist $n_-,s_-$ such that for every $n\geq n_-$ and $s\leq s_-$ we have $u_n(s,\cdot) \in \mathcal{U}^-$.
\end{lemma}

\begin{proof}
We will prove the statement for the positive end, the argument for the negative end is analogous and will be left to the reader.
By Stokes theorem, (B) and (E)
\begin{equation}
\label{ctct_area_u_estimate}
T^+ - T^- = \lim_{n\to\infty} T^+_n - T^-_n = \lim_{n\to\infty} \int_{\R\times\R/\Z} u_n^*d\lambda_n \leq \frac{1}{2} T_{\min}(\lambda)
\end{equation}
and by Stokes theorem together with $(u_n^*d\lambda_n)/(ds\wedge dt) \geq 0$
\begin{equation}
\label{action_s_constant}
T^-_n  \leq \int_{\R/\Z} u_n(s,\cdot)^*\lambda_n \leq T^+_n \qquad \forall s\in\R \, .
\end{equation}
We now claim that for every pair of sequences $n_j,s_j \to\infty$ the sequence $\vtil_j :\R\times\R/\Z \to \R\times Y$ defined by $$ \vtil_j (s,t) = (a_{n_j}(s+s_j,t) - a_{n_j}(s_j,0),u_{n_j}(s+s_j,t)) $$ has a subsequence that $C^\infty_\loc$-converges to a trivial cylinder over $\gamma^+$, i.e. to a map of the form $(s,t) \mapsto (T^+s+a_0,x^+(T^+(t+t_0))$.
With this claim at hand one can finish the proof of the lemma as follows.
Let $\mathcal{U}^+$ be an open $(\R/\Z)$-invariant $C^\infty$-neighborhood of $x^+(T^+\cdot)$.
Assume, by contradiction, that there are sequences $n_j,s_j\to\infty$ such that $u_{n_j}(s_j,\cdot) \not \in \mathcal{U}^+$ for all $j$.
Then $\vtil_j (s,t) = (a_{n_j}(s+s_j,t) - a_{n_j}(s_j,0),u_{n_j}(s+s_j,t))$ has a $C^\infty_\loc$-convergent subsequence to a trivial cylinder $\vtil=(b,v)$ over $\gamma^+$. 
Hence, $\lim_{j} u_{n_j}(s_j,\cdot) = v(0,\cdot) \in \mathcal{U}^+$, and $j\gg1 \Rightarrow u_{n_j}(s_j,\cdot) \in \mathcal{U}^+$ since $\mathcal{U}^+$ is open.
This contradiction concludes the argument.

We now argue to establish the above claim.
Note that $d\vtil_j$ is $C^0_\loc$-bounded since, otherwise, we would find a bubbling-off point and, after suitably rescaling and taking a limit, one would get a nonconstant $\jtil$-holomorphic plane with finite Hofer energy and total $d\lambda$-area at most $(1/2) T_{\min}(\lambda)$, but this is impossible.
Since $\vtil_j(0,0)$ belongs to $\{0\} \times Y$, we get $C^1_\loc$-bounds for~$\vtil_j$. 
Elliptic bootstrapping implies $C^\infty_\loc$-bounds. 
Hence, we obtain the desired subsequence convergent to some $\jtil$-holomorphic map $\vtil$ satisfying $E(\vtil) \leq \sup_nE(\util_n)<\infty$.
We still denote the convergent subsequence by $\vtil_j$.
The map $\vtil$ is not constant since by~\eqref{action_s_constant}
\begin{equation}
\label{action_estimate_loops_v}
\int_{\R/\Z} v(s,\cdot)^*\lambda = \lim_{j\to\infty} \int_{\R/\Z} u_{n_j}(s_{j}+s,\cdot)^*\lambda_{n_j} \in [T^-,T^+] \qquad \forall s\in\R \, .
\end{equation}
This also implies that $s=+\infty$ is a positive puncture and $s=-\infty$ is a negative puncture of~$\vtil$; otherwise $\int_{\R/\Z} v(s_*,\cdot)^*\lambda<T^-$ would hold for some~$s_*$.

Consider $\gamma=(x,T)$ an asymptotic limit of $\vtil$ at the negative puncture. 
By the above estimate we know that $T \in [T^-,T^+]$.
If $T = T^-$ then for every $\alpha>0$ we can find some $\underline s \in \R$ such that $\int_{\R/\Z} v(\underline s,\cdot)^*\lambda < T^-+\alpha$. 
With~$s$ fixed arbitrarily we could take $j\gg1$ such that $\underline s + s_{j}>s$ and estimate $$ T^-_{n_j} \leq \int_{\R/\Z} u_{n_j}(s,\cdot)^*\lambda_{n_j} \leq \int_{\R/\Z} u_{n_j}(\underline s+s_{j},\cdot)^*\lambda_{n_j} \to \int_{\R/\Z} v(\underline s,\cdot)^*\lambda < T^-+\alpha \, . $$ 
Passing to the limit as $j\to\infty$ $$ T^- \leq \int_{\R/\Z} u(s,\cdot)^*\lambda \leq T^- + \alpha \quad \forall s $$ where (B) was used.
Using Stokes, and passing to the limit as $\alpha \to 0$ we get $$ \int_{\R\times\R/\Z} u^*d\lambda = 0 $$ which means that $\util$ is a trivial cylinder over a periodic orbit, in contradiction to (F). 
We proved that $T>T^-$.
It follows from~\eqref{ctct_area_u_estimate} and~\eqref{action_estimate_loops_v} that $\gamma$ is geometrically distinct from $\gamma^-$.

We now argue by contradiction to show that $x(\R) \subset x^+(\R)$.
If we have $x(\R) \subset Y \setminus x^+(\R)$ then $x(\R) \subset Y \setminus (x^+(\R)\cup x^-(\R))$.
If $n$ is large enough then $S = \{x(T\cdot+t_0)\mid t_0\in\R/\Z\}$ has an open neighborhood $\mathcal{W} \subset C^\infty(\R/\Z,Y)$ such that every loop in $\mathcal{W}$ does not intersect the link $x^+_n(\R) \cup x^-_n(\R)$.
Otherwise, by the Arzel\`a-Ascoli theorem, we would find a sequence of loops $c_{n_j}$ converging to a loop in $S$, such that each $c_{n_j}$ intersects $x^+_{n_j}(\R) \cup x^-_{n_j}(\R)$.
Passing to the limit as $j\to\infty$, and using (B), we would find a loop in $S$ that intersects $x^+(\R) \cup x^-(\R)$, but we work under the assumption that this is not the case.
After shrinking $\mathcal{W}$, we can assume in addition that $\mathcal{W}$ is path-connected.
By assumptions (C) and (D) every loop in $S$ has positive intersection number with $u_n$. Since loops in $\mathcal{W}$ do not intersect the asymptotic limits of $\util_n$, we can use the path-connectedness of $\mathcal{W}$, and the homotopy invariance of intersection numbers, to conclude that every loop in $\mathcal{W}$ has positive intersection number with $u_n$.
We can find $\underline s \sim -\infty$ such that  $v(\underline s,\cdot)$ is close to $S$, hence belongs to $\mathcal{W}$.
Hence, for $j$ large enough $u_{n_j}(s_{j}+\underline s,\cdot)$ also belongs to $\mathcal{W}$ and, consequently, has positive intersection number with $u_{n_j}$.
But, by (D), this intersection number is zero.
This contradiction shows that $x(\R) \subset x^+(\R)$.
It follows from this,~\eqref{ctct_area_u_estimate} and~\eqref{action_estimate_loops_v} that $\gamma=\gamma^+$.
By~\eqref{action_estimate_loops_v} we get $\int_{\R/\Z} v(s,\cdot)^*\lambda = T^+$ for every $s$, hence $v^*d\lambda$ vanishes identically, and $\vtil$ is the desired trivial cylinder over $\gamma^+$.
\end{proof}

The above proof made only use of assumptions (A)-(F).

\begin{corollary}
\label{cor_localizing_the_ends}
If $U_+,U_- \subset Y$ are open neighborhoods of $x^+(\R),x^-(\R)$, respectively, then there exist $N \in \N$ and $\underline s>0$ such that if $n\geq N$ then 
$$
\begin{aligned}
u([\underline s,+\infty) \times \R/\Z) & \cup \bigcup_n u_n([\underline s,+\infty) \times \R/\Z) \subset U_+ \, , \\
u((-\infty,-\underline s] \times \R/\Z) & \cup \bigcup_n u_n((-\infty,-\underline s] \times \R/\Z) \subset U_- \, .
\end{aligned}
$$
\end{corollary}

\begin{corollary}
\label{cor_asymptotics_nondeg_case}
If $\gamma^+$ is nondegenerate then $\util$ is asymptotic to $\gamma^+$ at $s=+\infty$, and has an asymptotic formula at $s=+\infty$. Similarly, if $\gamma^-$ is nondegenerate then $\util$ is asymptotic to $\gamma^-$ at $s=-\infty$, and has an asymptotic formula at $s=-\infty$.
\end{corollary}

\begin{proof}
Assume that $\gamma^+$ is nondegenerate.
By Corollary~\ref{cor_localizing_the_ends} the only asymptotic limit of $\util$ at $s=+\infty$ is $\gamma^+$.
The conclusion at $s=+\infty$ follows from a direct application of Theorem~\ref{thm_HWZ_asymptotics_nondeg_case}.
The same argument applies to $s=-\infty$ if $\gamma^-$ is nondegenerate.
\end{proof}

\subsection{Special coordinates}

To proceed, it will be helpful to construct special coordinates to analyze the ends of the cylinders $\util_n$; this is the goal of this subsection.
We only analyze positive punctures, the construction for the negative punctures is analogous.
Let $\ubar{\gamma}^+ = (x^+,\ubar{T}^+)$ and $\ubar{\gamma}^+_n = (x^+_n,\ubar{T}^+_n)$ be the primitive Reeb orbits underlying $\gamma^+$ and $\gamma^+_n$, respectively.
Let $k_+ \in \N$ denote the covering multiplicity of $\gamma^+$.
Assume that ${\gamma}^+$ is degenerate.
A direct application of~\cite[Proposition~1]{Bangert} implies that for $n$ large enough there is a divisor $j_+$ of $k_+ = l_+j_+$ such that 
\begin{equation}
\gamma^+_n = (\ubar{\gamma}^+_n)^{l_+} \ \ \forall n \, , \qquad\ubar{\gamma}^+_n \to (\ubar{\gamma}^+)^{j_+} \, .
\end{equation}
Moreover, by~\cite[Proposition~1]{Bangert}, $j_+$ is equal to one or to the multiplicity of the first degenerate iterate of~$\ubar{\gamma}^+$.
In particular, we have $\ubar{T}^+_n \to j_+\ubar{T}^+$.

\begin{lemma}
\label{lemma_special_coordinates_covering}
There is an open tubular neighborhood $N_+$ of $x^+(\R)$ in $Y$, a smooth covering map $P_+ : \R/j_+\Z \times B \to N_+$ where $B \subset \C$ is an open ball centered at the origin, and local diffeomorphisms $P_n : \R/j_+\Z \times B \to N_+$ with the following properties.
\begin{itemize}
\item[(a)] The action of the group of deck transformations of $P_+$ is the action by $\Z/j_+\Z$ on the first component, and $P_n \to P_+$ in $C^\infty_\loc$ as $n\to+\infty$.
\item[(b)] There are lifts $\tilde x^+_n : \R \to \R/j_+\Z \times B$ of $x^+_n : \R \to Y$ by $P_n$, and $\tilde x^+:\R \to \R/j_+\Z \times B$ of $x^+:\R \to Y$ by $P_+$ such that
$$
\tilde x^+_n(t) = (j_+t/\ubar{T}^+_n,0) \qquad \text{and} \qquad \tilde x^+(t) = (t/\ubar{T}^+,0) \, .
$$
\item[(c)] There are smooth functions $\tilde f_n,\tilde f:\R/j_+\Z \times B \to \R$ such that $\tilde\lambda_n = P_n^*\lambda_n$ and $\tilde\lambda = P_+^*\lambda$ 
can be written as
$$
\tilde \lambda_n = \tilde f_n(\theta,z)(d\theta + x_1dx_2) \, , \qquad \tilde \lambda = \tilde f(\theta,z)(d\theta + x_1dx_2) \, .
$$
Moreover, $\tilde f_n,\tilde f$ satisfy
$$
\begin{aligned}
& \tilde f_n(\theta,0) = \ubar{T}^+_n/j_+, \ \tilde f(\theta,0) = \ubar{T}^+ \ \forall \theta \\
& d\tilde f_n(\theta,0) = d\tilde f(\theta,0) = 0 \ \forall \theta \\
& \text{$\tilde f_n \to \tilde f$ in $C^\infty_\loc(\R/j_+\Z \times B)$.}
\end{aligned}
$$
\end{itemize}
\end{lemma}

\begin{proof}
Consider a Martinet tube $\Psi : V \to \R/\Z \times B$ for $\lambda$ and $\gamma^+$.
This means that $\Psi$ is a diffeomorphism defined on an open neighborhood $V$ of $x^+(\R)$ such that 
$$
\begin{aligned}
& \Psi(x^+(\ubar{T}^+t))=(t,0) \\
& \Psi_*\lambda = h(\theta,z)(d\theta + x_1dx_2) \\
& h(\theta,0) = \ubar{T}^+, \ dh(\theta,0) = 0 \ \forall \theta \, .
\end{aligned}
$$ 
Here $B \subset \C$ is an open ball centered at the origin.
The contact structure $\xi = \ker\lambda$ gets represented as $\xi_0 = \Psi_*\xi = \ker (d\theta + x_1dx_2)$.
It follows from~(C) that 
$$
\Psi_*\lambda_n = h_n(\theta,z)(d\theta + x_1dx_2)
$$ 
where $h_n(\theta,0) = \ubar{T}^+$, $dh_n(\theta,0) = 0$ $\forall \theta$, and $h_n \to h$ in $C^\infty_\loc$.
Consider the $(\Z/j_+\Z)$-invariant covering map
$$
P : \R/j_+\Z \times B \to \R/\Z \times B \, , \quad (\theta+j_+\Z,z) \mapsto (\theta+\Z,z) \, .
$$
Use this map to lift $\Psi_*\lambda,\Psi_*\lambda_n$ to contact forms $\tilde\lambda = P^*(\Psi_*\lambda)$, $\tilde\lambda_n = P^*(\Psi_*\lambda_n)$ on $\R/j_+\Z \times B$, respectively, written as
$$
\tilde\lambda_n = \tilde h_n(\theta,z)(d\theta + x_1dx_2) \qquad \tilde\lambda = \tilde f(\theta,z)(d\theta + x_1dx_2)
$$
where $d\tilde f,d\tilde h_n$ vanish on $\R/j_+\Z \times \{0\}$, $\tilde f,\tilde h_n$ are equal to $\ubar{T}^+$ on $\R/j_+\Z \times \{0\}$, and $\tilde h_n \to \tilde f$ in $C^\infty_\loc$.
The loops $$ c_n : t \in \R/j_+\Z \mapsto \Psi \circ x^+_n(\ubar{T}^+_n t/j_+) 
$$ 
are constant-speed reparametrizations of simply covered Reeb orbits of $\Psi_*\lambda_n$.
Perhaps replacing~$t$ by~$t+t_n$ for suitable $t_n$, there is no loss of generality to assume that $c_n$ converges to $t \in \R/j_+\Z \mapsto (t,0) \in \R/\Z \times B$ in~$C^\infty$.
Here we used assumption~(B).
If we denote the Reeb vector field of $\Psi_*\lambda_n$ by $R_{\Psi_*\lambda_n}$ then
$$
\frac{d}{dt} \ c_n(t) = (\ubar{T}^+_n/j_+) R_{\Psi_*\lambda_n}(c_n(t))
$$
A sequence of lifts of $c_n$ by $P$ $$ \tilde c_n : \R/j_+\Z \to \R/j_+\Z \times B $$ can be chosen in such a way that it $C^\infty$-converges to the map $t \mapsto (t,0)$.
If we denote the Reeb vector field of $\tilde\lambda_n = P^*(\Psi_*\lambda_n)$ by $R_{\tilde\lambda_n}$ then
$$
\frac{d}{dt} \ \tilde c_n(t) = (\ubar{T}^+_n/j_+) R_{\tilde\lambda_n}(\tilde c_n(t))
$$
One can now find co-orientation preserving contactomorphisms 
$$
\phi_n : (\R/j_+\Z \times B,\xi_0) \to (\R/j_+\Z \times B,\xi_0)
$$ 
such that $$ \phi_n \to id \ \ \text{in $C^\infty$,} \quad \supp(\phi_n) \to \R/j_+\Z \times \{0\}, \quad \phi_n \circ \tilde c_n(t) = (t,0) \ \forall n,t \, . $$
The contact forms 
\begin{equation}
(\phi_n)_* \tilde\lambda_n = (\phi_n)_*(\tilde h_n(\theta,z)(d\theta + x_1dx_2)) = \tilde f_n(\theta,z)(d\theta + x_1dx_2)
\end{equation}
have $t \in \R/j_+\Z \mapsto \phi_n \circ \tilde c_n(t) = (t,0)$ as a Reeb orbit of primitive period $\ubar{T}^+_n$.
Moreover, if we denote the Reeb vector field of $(\phi_n)_*\tilde\lambda_n$ by $R_{(\phi_n)_*\tilde\lambda_n}$ then 
$$
\begin{aligned}
(1,0) & = \frac{d}{dt} \ \phi_n\circ\tilde c_n(t) = (\ubar{T}^+_n/j_+) R_{(\phi_n)_*\tilde\lambda_n}(\phi_n \circ\tilde c_n(t)) \\
& = (\ubar{T}^+_n/j_+) R_{(\phi_n)_*\tilde\lambda_n}(t,0)
\end{aligned}
$$
Thus $R_{(\phi_n)_*\tilde\lambda_n}(t,0)$ is, for each $n$, a constant multiple of $(1,0)$ and 
\begin{equation}
\label{hat_h_derivatives_at_core_circle}
d\tilde f_n(\theta,0) = 0 \, , \ \tilde f_n(\theta,0) = \ubar{T}^+_n/j_+ \ \forall \theta \, , \qquad \tilde f_n \to \tilde f \ \text{in} \ C^\infty_\loc
\end{equation}
Set $P_+ = \Psi^{-1} \circ P$, $P_n = \Psi^{-1} \circ P \circ (\phi_n)^{-1}$ defined on $\R/j_+\Z \times B$, and~$N_+ = \Psi^{-1}(\R/\Z \times B)$.
\end{proof}

\subsection{Uniform exponential decay}

We now begin the process of putting all of this together to prove Proposition~\ref{prop_main_asymptotic_analysis};  the goal of this section is to prove item (a).

Consider the neighborhood $N_+$ of $x^+(\R)$ and the smooth maps $P_+,P_n$ given by Lemma~\ref{lemma_special_coordinates_covering}.
The lift of the contact structure by $P_n$ or by $P_+$ is equal to $\ker (d\theta + x_1dx_2)$.
The lifts of $J_n$ and of $J$ by $P_n$ and by $P_+$, respectively, can be represented as smooth functions
$$ 
\hat J_n : \R/j_+\Z \times B \to \R^{2\times 2} \qquad \hat J : \R/j_+\Z \times B \to \R^{2\times 2} 
$$ 
with respect to the frame $\{\partial_{x_1},-x_1\partial_\theta+\partial_{x_2}\}$.
Note that $\hat J$ is invariant by the action of $\Z/j_+\Z$ on the first coordinate.
We get $\hat J_n \to \hat J$ in $C^\infty_\loc$ from $J_n \to J$ in $C^\infty_\loc$.
Similarly, write 
$$
\tilde R_n : \R/j_+\Z \times B \to \R^3 \qquad \tilde R : \R/j_+\Z \times B \to \R^3
$$
for the lifts of the Reeb vector fields of $\lambda_n$ and $\lambda$ by $P_n$ and by $P_+$, respectively.
The function $\tilde R$ is $(\Z/j_+\Z)$-invariant and $\tilde R_n \to \tilde R$ in $C^\infty_\loc$.
We write $D_2\tilde R_n,D_2\tilde R$ for the partial derivative in the second component.

Let $\ubar{s} > 0$ be given by an application of Corollary~\ref{cor_localizing_the_ends} to the neighborhood~$N_+$ of $x^+(\R)$.
It follows from Lemma~\ref{lemma_special_coordinates_covering} that a loop in $N_+$ can be lifted by~$P_n$ or by $P$ to $\R/j_+\Z \times B$ as a loop if, and only if, it is homotopic to $t \in \R/\Z \mapsto x^+(mj_+\ubar{T}^+t)$, for some $m \in \Z$.
For every $s \geq \ubar{s}$ and $n$ the loop $t \mapsto \R/\Z \mapsto u_n(s,t)$ is freely homotopic to $t \in \R/\Z \mapsto x^+(k_+\ubar{T}^+t)$.
Hence, since $k_+ = j_+l_+$, $(s,t) \in [\ubar{s},+\infty) \times \R/\Z \mapsto u_n(s,t) \in N_+$ can be lifted by $P_n$ as smooth maps 
$$
(s,t) \in [\ubar{s},+\infty) \times \R/\Z \mapsto (\theta_n(s,t),z_n(s,t)) \in \R/j_+\Z \times B .
$$
Similarly, $(s,t) \in [\ubar{s},+\infty) \times \R/\Z \mapsto u(s,t) \in N_+$ can be lifted by $P_+$ as a smooth map
$$
(s,t) \in [\ubar{s},+\infty) \times \R/\Z \mapsto (\theta(s,t),z(s,t)) \in \R/j_+\Z \times B.
$$
Moreover, since $\util_n \to \util$ in $C^\infty_\loc$, these lifts can be arranged to satisfy
$$
(\theta_n(s,t),z_n(s,t)) \to (\theta(s,t),z(s,t)) \qquad C^\infty_\loc([\ubar{s},+\infty) \times \R/\Z) \, .
$$
Write in components $$ z_n(s,t) = x_{1,n}(s,t) + ix_{2,n}(s,t) \qquad z(s,t) = x_{1}(s,t) + ix_{2}(s,t) \, . $$
Abbreviate 
$$
j_n(s,t) = \hat J_n(\theta_n(s,t),z_n(s,t)) \quad j(s,t) = \hat J(\theta(s,t),z(s,t))
$$
for simplicity.
Denote 
$$
\begin{aligned}
D_n(s,t) &= \int_0^1 D_2\tilde R_n(\theta_n(s,t),vz_n(s,t))dv \\
D(s,t) &= \int_0^1 D_2\tilde R(\theta(s,t),vz(s,t))dv
\end{aligned}
$$
and
\begin{equation}
\label{formula_S_n}
\begin{aligned}
S(s,t) &= [\partial_ta \ I - \partial_sa \ j]D(s,t) \\
S_n(s,t) &= [\partial_ta_n \ I - \partial_sa_n \ j_n]D_n(s,t)
\end{aligned}
\end{equation}
where $a_n(s,t),a(s,t)$ are the $\R$-components of $\util_n(s,t),\util(s,t)$.
The Cauchy-Riemann equations for the lifts of $\util_n,\util$ by $P_n,P_+$ get written as 
\begin{equation}
\label{CR_local_coords_u_n}
\begin{aligned}
& \partial_sa_n - \tilde f_n(\theta_n,z_n)(\partial_t\theta_n+x_{1,n}\partial_tx_{2,n}) = 0 \, , \\
& \partial_ta_n + \tilde f_n(\theta_n,z_n)(\partial_s\theta_n+x_{1,n}\partial_sx_{2,n}) = 0 \, , \\
& \partial_sz_n + j_n\partial_tz_n +S_nz_n = 0 \, ,
\end{aligned}
\end{equation}
and, similarly,
\begin{equation}
\label{CR_local_coords_u}
\begin{aligned}
& \partial_sa - \tilde f(\theta,z)(\partial_t\theta+x_{1}\partial_tx_{2}) = 0 \, , \\
& \partial_ta + \tilde f(\theta,z)(\partial_s\theta+x_{1}\partial_sx_{2}) = 0 \, , \\
& \partial_sz + j\partial_tz +Sz = 0 \, .
\end{aligned}
\end{equation}
Above, and also in what follows, we see derivatives of $\theta_n,\theta$ as $\R$-valued functions.

\begin{lemma}
\label{lemma_uniform_bounds_ends}
We have
\begin{align}
& \lim_{\tau\to+\infty} \sup_{n} \ \|D^\beta z_n(\tau,\cdot)\|_{L^\infty(\R/\Z)} = 0 \quad \forall \beta \label{unif_estimates_z_n_comp} \\
& \lim_{\tau\to+\infty} \sup_{n} \ \left( \begin{aligned} & \|D^\beta [a_n-T^+_ns](\tau,\cdot)\|_{L^\infty(\R/\Z)} \\ & \quad + \|D^\beta [\theta_n-k_+t](\tau,\cdot)\|_{L^\infty(\R/\Z)} \end{aligned} \right) = 0 \quad \text{$\forall \beta \neq (0,0)$} \label{unif_estimates_derivatives_symp_comp}
\end{align}
where $\beta = (\beta_1,\beta_2) \in \N_0 \times \N_0$ and $D^\beta = \partial^{\beta_1}_s\partial^{\beta_2}_t$.
\end{lemma}

\begin{proof}
By Lemma~\ref{lemma_localizing_the_ends},~\eqref{unif_estimates_z_n_comp} holds for all $\beta = (0,\beta_2)$.
We claim that 
\begin{equation}
\label{uniform_gradient_bounds}
\sup_{s \in [\ubar{s},+\infty), \, t\in\R/\Z, \, n} |\nabla \util_n(s,t)| < \infty \, . 
\end{equation}
If not then, up to a subsequence, we find bubbling-off points around which we could rescale and find a nonconstant holomorphic plane with finite Hofer energy. 
Such a plane would need to have $d\lambda$-area at least equal to $T_{\min}(\lambda)$, but (E) forces the $d\lambda$-area of this plane to be at most $\frac{1}{2}T_{\min}(\lambda)$.
This contradiction proves~\eqref{uniform_gradient_bounds}.
Standard boot-strapping arguments show that~\eqref{uniform_gradient_bounds} implies uniform bounds 
\begin{equation}
\label{uniform_higher_derivatives_bounds}
\sup_{s \in [\ubar{s},+\infty), \, t\in\R/\Z, \, n} |D^\alpha\nabla\util_n(s,t)| < \infty \quad \forall \alpha
\end{equation}
for the higher-order derivatives, see~\cite[Lemma~7.1]{convex}.
It follows from~\eqref{uniform_higher_derivatives_bounds},~\eqref{formula_S_n}, and~\eqref{CR_local_coords_u_n}, that~\eqref{unif_estimates_z_n_comp} holds when $|\beta| \leq 1$.

Choose lifts $\tilde\theta_n : \R \times \R \to \R$ of $\theta_n$.
If~\eqref{unif_estimates_derivatives_symp_comp} does not hold then one finds $\beta \neq (0,0)$, $\epsilon > 0$ and sequences $s_m \in \R$, $t_m \in \R/\Z$, $n_m \in \N$ satisfying $s_m \to +\infty$, $n_m \to +\infty$ and 
\begin{equation}
\label{contradiction_control_derivatives_symp_direction}
|D^\beta[a_{n_m}-T^+_{n_m}s](s_m,t_m)| + |D^\beta[\tilde\theta_{n_m}-k_+t](s_m,t_m)| \geq \epsilon \, .
\end{equation}
Define $h_m : \C \to \C$ by
$$
\begin{aligned}
h_m(s+it) & = \frac{1}{T_{n_m}^+} (a_{n_m}(s+s_m,t+t_m) - a_{n_m}(s_m,t_m)) \\
& \quad + \frac{i}{k_+} (\tilde\theta_{n_m}(s+s_m,t+t_m) - \tilde\theta_{n_m}(s_m,t_m)) \, .
\end{aligned}
$$
By~\eqref{uniform_higher_derivatives_bounds} $\nabla h_m$ is $C^\infty_\loc$-bounded. 
Together with $h_m(0)=0$ we get $C^\infty_\loc$-bounds for $h_m$.
Hence, up to a taking a subsequence, we can assume without loss of generality that $h_m$ is $C^\infty_\loc$ convergent to some $h : \C \to \C$.
Equations~\eqref{CR_local_coords_u_n}, together with~\eqref{unif_estimates_z_n_comp} for $|\beta| \leq 1$, imply that $h$ is holomorphic.
Moreover,~\eqref{uniform_higher_derivatives_bounds} implies that $h$ has bounded derivatives.
By Lemma~\ref{lemma_localizing_the_ends}, $\partial_th \equiv 1$ and $h$ is not constant.
By Liouville's theorem, $h(s+it)$ must be a degree-one complex polynomial in $s+it$ satisfying $h(0)=0$ and $\partial_th \equiv 1$.
Hence, $h(s+it) = s+it$.
By~\eqref{contradiction_control_derivatives_symp_direction} we have $|D^\beta[h_m-s-it](0)|$ bounded away from zero, but we have just proved that $h_m-s-it$ $C^\infty_\loc$-converges to $0$.
This contradiction concludes the proof that~\eqref{unif_estimates_derivatives_symp_comp} holds.

The proof that~\eqref{unif_estimates_z_n_comp} holds for all $\beta$ can now be completed by differentiating the equation for $z_{n}$ in~\eqref{CR_local_coords_u_n} and plugging  the bounds for $S_n$ that follow from~\eqref{unif_estimates_derivatives_symp_comp}.
\end{proof}

Let $$ \hat M_n : \R/j_+\Z \times B \to Sp(2) \qquad \hat M : \R/j_+\Z \times B \to Sp(2) $$ be smooth functions satisfying $\hat M_n\hat J_n = i\hat M_n$, $\hat M\hat J = i\hat M$ pointwise.
Since $\hat J_n \to \hat J$, we can arrange these functions to satisfy
\begin{equation*}
\hat M_n \to \hat M \quad \text{in} \quad C^\infty_\loc \, .
\end{equation*}
Consider the smooth functions $\R/\Z \to \R^{2\times 2}$
$$
\begin{aligned}
j^\infty(t) &= \hat J(k_+t,0) \\
D^\infty(t) &= D_2 \tilde R(k_+t,0) \\
S^\infty(t) &= -T_+ j^\infty(t) D^\infty(t) \\
M^\infty(t) &= M(k_+t,0) \, .
\end{aligned}
$$
As before, we abbreviate 
$$
\begin{aligned}
M_n(s,t) = \hat M_n(\theta_n(s,t),z_n(s,t)) \, , \qquad
M(s,t) = \hat M(\theta(s,t),z(s,t)) \, .
\end{aligned}
$$
The next lemma is~\cite[Lemma~4.12]{elliptic} or~\cite[Lemma~3.33]{HSW}.

\begin{lemma}
\label{lemma_data_for_CR_asymptotic}
For every pair of sequences $s_m \to \infty$, $n_m \to \infty$ there exist $c\in[0,1)$ and a sequence of indices~$m_l\to\infty$ such that 
$$ \begin{aligned}
& \lim_{l\to+\infty} \|D^\beta[j_{n_{m_l}}(s,t) - j^\infty(t+c)](s_{m_l},\cdot)\|_{L^\infty(\R/\Z)} = 0 \\
& \lim_{l\to+\infty} \|D^\beta[D_{n_{m_l}}(s,t) - D^\infty(t+c)](s_{m_l},\cdot)\|_{L^\infty(\R/\Z)} = 0 \\
& \lim_{l\to+\infty} \|D^\beta[S_{n_{m_l}}(s,t) - S^\infty(t+c)](s_{m_l},\cdot)\|_{L^\infty(\R/\Z)} = 0 \\
& \lim_{l\to+\infty} \|D^\beta[M_{n_{m_l}}(s,t) - M^\infty(t+c)](s_{m_l},\cdot)\|_{L^\infty(\R/\Z)} = 0
\end{aligned} $$
hold for every partial derivative $D^\beta = \partial^{\beta_1}_s\partial^{\beta_2}_t$.
\end{lemma}

Define functions
\begin{equation*}
\begin{aligned}
\zeta_n(s,t) &= M_nz_n \, , \\
\Lambda_n(s,t) &= (M_nS_n-\partial_sM_n-i\partial_tM_n)M_n^{-1} \, , \\
\Lambda^\infty(t) &= (M^\infty S^\infty - i\partial_tM^\infty)(M^\infty)^{-1} \, .
\end{aligned}
\end{equation*}

\begin{lemma}
The following hold:
\begin{itemize}
\item[(a)] 
$\partial_s\zeta_n + i\partial_t\zeta_n + \Lambda_n\zeta_n = 0$.
\item[(b)] For every pair of sequences $s_m \to \infty$, $n_m \to \infty$ there exist $c\in[0,1)$ and a sequence of indices~$m_l\to\infty$ such that 
$$
\lim_{l\to+\infty} \|D^\beta[\Lambda_{n_{m_l}}(s,t) - \Lambda^\infty(t+c)](s_{m_l},\cdot)\|_{L^\infty(\R/\Z)} = 0 \, .
$$
\end{itemize}
\end{lemma}

\begin{proof}
A direct computation shows (a), and (b) is a direct consequence of Lemma~\ref{lemma_data_for_CR_asymptotic}.
\end{proof}

\begin{remark}
\label{rmk_Banach_Holder}
Consider $r	\in \N_0$, $\alpha \in (0,1)$, $d<0$ and a finite-dimensional vector space~$V$.
Consider the space $C^{r,\alpha,d}_0([\ubar{s},+\infty)\times\R/\Z,V)$ of functions $h:[\ubar{s},+\infty)\times\R/\Z\to V$ of class $C^{r,\alpha}_\loc$ such that for every~$\beta$ satisfying $|\beta|\leq r$ the $C^{0,\alpha}$-norm of $e^{-d s}D^\beta h$ on $[\rho,+\infty)\times\R/\Z$ is finite and decays to zero as $\rho\to+\infty$.
The norm $\|h\|_{r,\alpha,d} = \|e^{-d s}h\|_{C^{r,\alpha}([\ubar{s},+\infty)\times\R/\Z)}$ turns this space into a Banach space.
\end{remark}

\begin{proposition}[Proposition~4.15 in~\cite{elliptic}]
\label{prop_asymp_control}
Suppose that 
$$
K_n : [0,+\infty) \times \R/\Z \to \R^{2k\times 2k} \qquad K^\infty : \R/\Z \to \R^{2k\times 2k}
$$
are smooth maps satisfying
\begin{itemize}
\item $K^\infty(t)$ is symmetric, for all $t$.
\item For every pair of sequences $s_m \to \infty$, $n_m \to \infty$ there exist $c\in[0,1)$ and a sequence of indices~$m_l\to\infty$ such that 
$$
\lim_{l\to+\infty} \|D^\beta[K_{n_{m_l}}(s,t) - K^\infty(t+c)](s_{m_l},\cdot)\|_{L^\infty(\R/\Z)} = 0 \, .
$$
\end{itemize}
Consider the unbounded self-adjoint operator $L = -J_0\partial_t - K^\infty$ on $L^2(\R/\Z,\R^{2k})$.
Consider $r \in \N$, $\alpha \in (0,1)$.
Let $\delta>0$ be such that $-\delta$ does not belong to the spectrum of $L$.
Consider a $C^\infty_\loc$-bounded sequence of maps $$ X_n \in C^{r,\alpha,-\delta}_0([0,+\infty) \times \R/\Z,\R^{2k}) \cap C^\infty([0,+\infty) \times \R/\Z,\R^{2k}) $$ satisfying $$ \partial_sX_n + J_0\partial_tX_n + K_nX_n = 0 \, . $$ 
Then $X_n$ has a convergent subsequence in $C^{r,\alpha,-\delta}_0([0,+\infty) \times \R/\Z,\R^{2k})$.
\end{proposition}

Fix $\alpha\in(0,1)$.
Let $r\in\N$ be arbitrary, and recall the number $\delta>0$ from hypothesis~(G).
We write $C^{r,\alpha,-\delta}_0$ instead of $C^{r,\alpha,-\delta}_0([\ubar{s},+\infty)\times\R/\Z)$ for simplicity.
By assumption,~$\util_n$ has nondegenerate punctures in the sense of Definition~\ref{def_nondeg_puncture}.
By Theorem~\ref{thm_nondeg_implies_asymp_formula}, $\util_n$ has an asymptotic formula at both punctures.
This implies that $z_n \in C^{r,\alpha,-\delta}_0$ since, by~(G), the corresponding asymptotic eigenvalue $\mu_n^+$ of $\util_n$ at the positive puncture satisfies $\mu^+_n < -\delta$.
Moreover, from Lemma~\ref{lemma_uniform_bounds_ends} it follows that 
\begin{equation}
\label{bounds_derivatives_M_n}
\limsup_{\tau\to+\infty} \left( \sup_n \| D^\beta M_n(\tau,\cdot) \|_{L^\infty(\R/\Z)} \right) < \infty \qquad \forall \beta \, . 
\end{equation}
This and the fact that $z_n$ belongs to $C^{r,\alpha,-\delta}_0$ together imply that $\zeta_n \in C^{r,\alpha,-\delta}_0$.
Note that for each $c$ the operator $-i\partial_t-M^\infty(t+c)$ is representation of $A^+$ in a suitable frame and, consequently, does not have~$-\delta$ in the spectrum.
Finally, since $z_n \to z$ in $C^\infty_\loc$, we get $$ \zeta_n \to \zeta := M(s,t)z(s,t) \quad \text{in} \quad C^\infty_\loc \, . $$
Consequently, $\zeta_n$ is $C^\infty_\loc$-bounded.
We have checked all assumptions of Proposition~\ref{prop_asymp_control}, which implies $\zeta \in C^{r,\alpha,-\delta}_0$ and $\zeta_n \to \zeta$ in $C^{r,\alpha,-\delta}_0$.
With ~\eqref{bounds_derivatives_M_n} this gives that $z(s,t) \in C^{r,\alpha,-\delta}_0$ and that $z_n \to z \text{ in } C^{r,\alpha,-\delta}_0$.
Since $r \in \N$ is arbitrary, we conclude the proof of the following statement.

\begin{lemma}
\label{lemma_exponential_decay_z_component}
We have:
$$
\begin{aligned}
& \lim_{\tau\to+\infty} e^{\delta\tau} \left( \|D^\beta z(\tau,\cdot)\|_{L^\infty(\R/\Z)} + \sup_n \|D^\beta z_n(\tau,\cdot)\|_{L^\infty(\R/\Z)} \right) = 0 \qquad \forall \beta \, .
\end{aligned}
$$
\end{lemma}

Since each $\util_n$ has an asymptotic formula, we can find constants $c_n,d_n \in \R$ such that the function $V_n : [\ubar{s},+\infty) \times \R/\Z \to \R^2$ defined by
$$
V_n(s,t) = \begin{pmatrix} a_n(s,t)/T^+_n - s - c_n \\ \tilde \theta_n(s,t)/k_+ - t - d_n \end{pmatrix}
$$
satisfies $\|V_n(s,\cdot)\|_{L^\infty(\R/\Z)} \to 0$ as $s \to +\infty$, for each $n$.
Equations~\eqref{CR_local_coords_u_n} can be written as
\begin{equation}
\label{CR_eqns_symp_comps}
\partial_sV_n(s,t) + J_0\partial_tV_n(s,t) + B_n(s,t)z_n(s,t) = 0
\end{equation}
where $B_n(s,t)$ is a $2 \times 2$ matrix-valued function whose entries consist of smooth functions defined on $\R/j_+\Z \times B$ composed with the map $(s,t) \mapsto (\theta_n(s,t),z_n(s,t))$.
From the bounds given by Lemma~\ref{lemma_uniform_bounds_ends} we get
$$
\limsup_{\tau \to +\infty} \ \sup_n \|D^\beta B_n(\tau,\cdot)\|_{L^\infty(\R/\Z)} < \infty \qquad \forall \beta \, .
$$
Together with Lemma~\ref{lemma_exponential_decay_z_component} this yields
\begin{equation}
\label{decay_RHS_symp_comps}
\limsup_{\tau \to +\infty} \ \sup_n e^{\delta\tau} \|D^\beta [B_nz_n](\tau,\cdot)\|_{L^\infty(\R/\Z)} = 0 \qquad \forall \beta \, .
\end{equation}
It follows as in~\cite[Lemma~3.36]{HSW}, which is a version of~\cite[Lemma~6.3]{fast} for sequences, that one can use~\eqref{CR_eqns_symp_comps} and~\eqref{decay_RHS_symp_comps} to find $b>0$ such that 
\begin{equation}
\lim_{\tau \to +\infty} \ \sup_n e^{b \tau} \|D^\beta V_n(\tau,\cdot)\|_{L^\infty(\R/\Z)} = 0 \quad \text{if} \quad |\beta| \geq 1 \, .
\end{equation}
With $\beta=(1,0)$ we get
$$
\lim_{\tau \to +\infty} \ \sup_n e^{b \tau} \left( \|\partial_sa_n(\tau,\cdot) - T_n^+ \|_{L^\infty(\R/\Z)} + \|\partial_s\tilde\theta_n(\tau,\cdot)\|_{L^\infty(\R/\Z)} \right) = 0
$$
Taking the limit as $n\to\infty$
$$
\lim_{\tau \to +\infty} e^{b \tau} \left( \|\partial_sa(\tau,\cdot) - T^+ \|_{L^\infty(\R/\Z)} + \|\partial_s\tilde\theta(\tau,\cdot)\|_{L^\infty(\R/\Z)} \right) = 0 \, .
$$
It follows from this exponential decay that the integrals 
$$
\begin{aligned}
& c = a(\ubar{s},0) - T^+\ubar{s} + \int_{\ubar{s}}^{+\infty} \partial_sa(s,0)-T^+ \ ds  \\
& d = \tilde \theta(\ubar{s},0) + \int_{\ubar{s}}^{+\infty} \partial_s\tilde \theta(s,0) \ ds
\end{aligned}
$$
are convergent, and 
$$
a(s,0) - T^+s - c \to 0 \qquad \tilde\theta(s,0) - d \to 0 \, .
$$
Together with Lemma~\ref{lemma_zeros_pi_du} and Lemma~\ref{lemma_exponential_decay_z_component}, this concludes the proof that $s=+\infty$ is a nondegenerate puncture of $\util$ in the sense of Definition~\ref{def_nondeg_puncture}, in case $\gamma^+$ is a degenerate periodic orbit.
The case where $\gamma^+$ is nondegenerate was taken care by Corollary~\ref{cor_asymptotics_nondeg_case}.
The analysis at the puncture $s=-\infty$ is identical.
An application of Theorem~\ref{thm_nondeg_implies_asymp_formula} concludes the proof of (a) in Proposition~\ref{prop_main_asymptotic_analysis}.

\subsection{Asymptotic data of the cylinder $\util$}

Having proved (a), we now prove (b) and (c) in Proposition~\ref{prop_main_asymptotic_analysis}.

\begin{lemma}
Choose symplectic trivializations of $\xi$ along $\gamma^+,\gamma^-$ that are homotopic to symplectic trivializations of $\xi$ along $\gamma^+_n,\gamma^-_n$, respectively.
In these trivializations, if $n$ is large enough then 
$$ \begin{aligned} 
\wind(\nu^{<-\delta}(\gamma^+)) & = \wind(\nu^{<-\delta}(\gamma^+_n)) \\
\wind(\nu^{\geq\delta}(\gamma^-)) & = \wind(\nu^{\geq\delta}(\gamma^-_n)) \, .
\end{aligned}
$$
Moreover, if $n$ is large enough then 
$$
p^{-\delta}(\gamma^+) = p^{-\delta}(\gamma^+_n) = 1 \, , \qquad p^{\delta}(\gamma^-) = p^{\delta}(\gamma^-_n) = 1 \, .
$$
\end{lemma}

\begin{proof}
We only prove the statements for $\gamma^+,\gamma_n^+$, the corresponding statements for $\gamma^-,\gamma_n^-$ can be proved analogously.
Represent asymptotic operators $A^+_n,A^+$ in the coordinates given by Lemma~\ref{lemma_special_coordinates_covering}, with respect to the frame $\{\partial_{x_1},\partial_{x_2}\}$ of $\xi$ along $\R/j_+\Z \times \{0\}$.
By (b) and (c) of Lemma~\ref{lemma_special_coordinates_covering}, $A^+_n$ gets represented as $-j_n(t)\partial_t-B_n(t)$, $A^+$ gets represented as $-j(t)\partial_t-B(t)$, where $j_n \to j$ and $B_n \to B$ in $C^\infty$.
Hence, the spectrum of $A^+_n$ is a small perturbation of the spectrum of $A^+$, and winding numbers of nearby eigenvalues in any given frame coincide.
In particular, $\sigma(A^+_n)$ is bounded away from $-\delta$, $\nu^{<-\delta}(A^+_n) \to \alpha \in \sigma(A^+)$, $\nu^{\geq-\delta}(A^+_n) \to \beta \in \sigma(A^+)$ with $\alpha < -\delta$ and $\beta > -\delta$.
We used that, by (G), $-\delta \not \in \sigma(A^+)$.
If $\alpha < \nu^{<-\delta}(A^+)$ then there is an eigenvalue of $A^+$ in $(\alpha,-\delta)$, and for $n$ large enough there is an eigenvalue of $A^+_n$ in $(\nu^{<-\delta}(A^+_n),-\delta)$ which is impossible.
This shows that $\alpha = \nu^{<-\delta}(A^+)$. 
Similarly, one proves that $\beta = \nu^{\geq-\delta}(A^+)$.
Hence $$ n \gg 1 \qquad \Rightarrow \qquad \begin{aligned} & \wind(\nu^{<-\delta}(A^+)) = \wind(\nu^{<-\delta}(A^+_n)) \\ & \wind(\nu^{\geq-\delta}(A^+)) = \wind(\nu^{\geq-\delta}(A^+_n)) \end{aligned} $$ from where it follows that $n \gg 1 \Rightarrow p^{-\delta}(\gamma^+) = p^{-\delta}(\gamma^+_n)$.
\end{proof}

\begin{corollary}
$\ind_\delta(\util) = \ind_\delta(\util_n) = 2$.
\end{corollary}

The above lemma and corollary give (b) in~Proposition~\ref{prop_main_asymptotic_analysis}.
Let $\mu^+$ be the asymptotic eigenvalue of $\util$ at $s=+\infty$, and let $\mu^-$ be the asymptotic eigenvalue of $\util$ at $s=-\infty$.
These exist since we prove in (a) that $\util$ has a nontrivial asymptotic formula at both punctures.
Apply Lemma~\ref{lemma_ineq_wind_pi} to the curve~$\util$ to get $\wind(\mu^+) = \wind(\nu^{<-\delta}(\gamma^+))$ in some (hence any) trivialization along $\gamma^+$, and $\wind(\mu^-) = \wind(\nu^{\geq\delta}(\gamma^-))$ in some (hence any) trivialization along $\gamma^-$.
This proves (c).

To conclude, we need to prove~(d), and we now do this.
Suppose, by contradiction, that $\exists (s_0,t_0)$ such that $u(s_0,t_0) \in x^+(\R) \cup x^-(\R)$.
For simplicity assume $u(s_0,t_0) \in x^+(\R)$, the case $u(s_0,t_0) \in x^-(\R)$ is handled analogously.
By (a), (b) and (c) we can apply Lemma~\ref{lemma_ineq_wind_pi} to conclude that $u$ is an immersion transverse to $R$.
We find $0 < \epsilon < 1/2$ and a compact neighborhood $K$ of $x^+(\R)$ such that, if we denote by $\Delta_\epsilon$ the closed $\epsilon$-disk centered at $(s_0,t_0)$, then the following holds:
\begin{itemize}
\item[(i)] $u(\partial\Delta_\epsilon) \cap K = \emptyset$.
\item[(ii)] $\exists n_0 \in \N$ such that $n \geq n_0 \Rightarrow x_n(\R) \subset K$.
\item[(iii)] $\exists n_1,N \in \N$ such that $n \geq n_1 \Rightarrow 0 < {\rm int}(x_n^+(T^+_n\cdot),u|_{\Delta_\epsilon}) \leq N$.
\item[(iv)] $\exists n_2 \in \N$ such that if $n \geq n_2$ then $u_n|_{\partial\Delta_\epsilon}$ is homotopic to $u|_{\partial\Delta_\epsilon}$ in $Y \setminus K$.
\end{itemize}
These properties follows from the transversality between $u$ and $R$, and from assumptions (A) and (B).
From homotopy invariance of intersection numbers we get
$$
n \geq \max\{n_0,n_1,n_2\} \Rightarrow {\rm int}(u_n|_{\Delta_\epsilon},x_n^+(T^+_n\cdot)) = {\rm int}(u|_{\Delta_\epsilon},x_n^+(T^+_n\cdot)) > 0 \, .
$$
It follows that $u_n(\R\times\R/\Z)$ intersects $x^+_n(\R)$, in contradiction to~(D).
We proved that 
$$
u(\R\times\R/\Z) \cap (x^+(\R) \cup x^-(\R)) = \emptyset \, .
$$

Now let $\gamma=(x,T)$ be a Reeb orbit of $\lambda$ in $Y \setminus (x^+(\R) \cup x^-(\R))$ that satisfies $T \leq T_+$.
Choose a compact neighborhood $K$ of $x^+(\R) \cup x^-(\R)$ such that $x(\R) \cap K = \emptyset$.
By~(B) we find $n_0 \in \N$ such that $$ n \geq n_0 \Rightarrow x^+_n(\R) \cup x^-_n(\R) \subset K \, . $$
By Corollary~\ref{cor_localizing_the_ends} we find $h > 0$ such that $$ |s| \geq h \Rightarrow u(\{s\}\times\R/\Z) \subset K \ \text{and} \ u_n(\{s\}\times\R/\Z) \subset K \ \forall n\, . $$
By~(C), $\gamma=(x,T)$ is a Reeb orbit of $\lambda_n$.
By~(D), ${\rm int}(x(T\cdot),u_n) > 0$ for all $n$.
We find $(s_n,t_n)$ such that $u_n(s_n,t_n) \in x(\R)$.
We get $|s_n| < h \ \forall n$ and, up to a subsequence, we may assume $(s_n,t_n) \to (s_*,t_*)$.
From~(A) we get $u(s_*,t_*) \in x(\R)$.
From the transversality between $u$ and $R$ provided by Lemma~\ref{lemma_ineq_wind_pi} we get ${\rm int}(x(T\cdot),u)>0$, as desired.

Consider a loop $\beta \subset \R \times \R/\Z$.
Since $u(\R\times\R/\Z) \cap (x^+(\R) \cap x^-(\R)) = \emptyset$ we find a compact neighborhood $F$ of $u(\beta)$, a compact neighborhood $K$ of of $x^+(\R) \cap x^-(\R)$ and $n_0 \in \N$ such that $F \cap K = \emptyset$ and 
$$ n \geq n_0 \Rightarrow \text{$u_n(\beta) \subset F$ and $u_n\circ\beta$ is homotopic to $u\circ\beta$ in $F$.} $$
Here we used~(A).
By Corollary~\ref{cor_localizing_the_ends} we find $h > 0$ such that $$ |s| \geq h \Rightarrow u(\{s\}\times\R/\Z) \subset \mathring{K} \ \text{and} \ u_n(\{s\}\times\R/\Z) \subset \mathring{K} \ \forall n\, . $$
Hence, there exists $n_1 \in \N$ such that $$ n \geq n_1 \Rightarrow \text{$u_n(\pm h,\cdot)$ is homotopic to $u(\pm h,\cdot)$ in $K$.} $$
Using the homotopy invariance of intersection numbers
$$ 
\begin{aligned}
n \geq \max\{n_0,n_1\} \quad \Rightarrow \quad {\rm int}(u \circ \beta,u) 
& = {\rm int}(u \circ \beta,u|_{[-h,h] \times \R/\Z}) \\
& = {\rm int}(u_n \circ \beta,u_n|_{[-h,h] \times \R/\Z}) \\
& = 0
\end{aligned}
$$
as desired.
The proof of Proposition~\ref{prop_main_asymptotic_analysis} is completed. 

\subsection{Compactness}

Before moving on, we also prove the compactness statement we will need.
Continue with the notation  $g_n,\lambda_n,\lambda,J_n,J$ as above, and 
consider nonconstant finite-energy $\jtil_n$-holomorphic maps $\util_n = (a_n,u_n)$ as in~\eqref{seq_u_n}.

\begin{lemma}
\label{lemma_compactness}
Assume that hypotheses ($*$), (B) and (E) hold.
There exist sequences $n_j \in \N$, $c_j \in \R$, $(s_j,t_j) \in \R \times \R/\Z$ such that $n_j \to \infty$ and the sequence $$ (s,t) \mapsto (a_{n_j}(s+s_j,t+t_j)+c_j,u_j(s+s_j,t+t_j)) $$ is $C^\infty_\loc$-convergent to a nonconstant finite-energy $\jtil$-holomorphic map
\begin{equation}
\label{limit_compactness_lemma}
\util =(a,u) : \R \times \R/\Z \to \R \times Y
\end{equation}
with a positive puncture at $s=+\infty$ and a negative puncture at $s=-\infty$.
Moreover, if $\lambda$ has finitely many simple Reeb orbits then $u(0,0)$ does not belong to a Reeb orbit of $\lambda$.
\end{lemma}

\begin{proof}
From~(B) we know that $\{\util_n\}$ has uniformly bounded Hofer energy.
From~(E) it follows that $d\util_n$ is $C^0_\loc$-bounded: otherwise we would find a nonconstant finite-energy $\jtil$-holomorphic plane with total $d\lambda$-area strictly less than $T_{\min}(\lambda)$, which is impossible.
Define $\vtil_n = (a_n-a_n(0,0),u_n)$.
Then $\vtil_n$ is $C^1_\loc$-bounded.
Elliptic boot-strapping implies $C^\infty_\loc$-bounds.
Hence there is a $C^\infty_\loc$-convergent subsequence $\vtil_{n_j} \to \util = (a,u) : \R \times \R/\Z \to \R \times Y$ to a $\jtil$-holomorphic map with finite Hofer energy.
Since $$ T^-_{n_j} \leq \int_{\R/\Z} u_{n_j}(s,\cdot)^*\lambda_{n_j} \leq T^+_{n_j} \quad \Rightarrow \quad T^- \leq \int_{\R/\Z} u(s,\cdot)^*\lambda \leq T^+ \quad \forall s $$
one gets that $\util$ is  nonconstant, has a positive puncture at $s=+\infty$ and a negative puncture at $s=-\infty$.

Assume that $\lambda$ has finitely many simple Reeb orbits.
From this,~(B) and~(E) we get $0 < T^+-T^- < T_{\min}(\lambda)$.
Hence $\gamma^+,\gamma^-$ are geometrically distinct.
By~(B), the distance between $\gamma^+_{n}$ and $\gamma^-_{n}$ is bounded away from zero, uniformly~in $n\gg1$.
Hence, one finds $(s_n,t_n)$ such that $u_n(s_n,t_n)$ converges to a point in the complement of the set of Reeb orbits of $\lambda$.
One can now run the previous argument with $\util_n(s+s_n,t+t_n)$ in the place of $\util_n$ to get the desired conclusion.
\end{proof}

\section{Two or infinitely many simple Reeb orbits}

We now put everything together to prove our main theorem.  To prove our theorem, we can assume that there are finitely many simple Reeb orbits, and this will be our standing assumption throughout this section.  

Recall the contact structure $\xi = \ker \lambda$.  The following lemma collects what we need to know from Proposition~\ref{prop:main}.

\begin{lemma}
\label{lem:fromprop}
There exists a $d\lambda$-compatible complex structure $J : \xi \to \xi$ and sequences $g_n,J_n,\util_n$ with the following properties.
\begin{enumerate}

\item[(a)] $g_n \in C^\infty(Y)$, $g_n \to 1$ in $C^\infty$, $\lambda_n = g_n\lambda$ is nondegenerate $\forall n$.

\item[(b)] $J_n : \xi \to \xi$ are $d\lambda_n$-compatible (equivalently $d\lambda$-compatible) and $J_n \to J$ in $C^\infty$.

\item[(c)] $\util_n = (a_n,u_n) : \R \times \R/\Z \to \R \times Y$ are non-constant finite-energy $\jtil_n$-holomorphic maps, with a positive puncture at $s=+\infty$ and a negative puncture at $s=-\infty$.

\item[(d)] The asymptotic limits $\gamma^\pm_n=(x^\pm_n,T^\pm_n)$ of $\util_n$ at $s=\pm\infty$ satisfy $\sup_n T^+_n < +\infty$, and $r \leq T^+_n - T^-_n \leq \frac{1}{2} T_{\min}(\lambda)$ for some $r>0$.
Moreover, there are Reeb orbits $\gamma^\pm=(x^\pm,T^\pm)$ of $\lambda$ such that $\gamma^\pm_n \to \gamma^\pm$, in the sense that $x^\pm_n(T^\pm_n(\cdot+t^\pm_n)) \to x^\pm(T^\pm\cdot)$ in $C^\infty(\R/\Z,Y)$ for suitable $t^\pm_n \in \R/\Z$.

\item[(e)] If $x : \R \to Y$ is a periodic $\lambda$-Reeb trajectory with period at most~$T^+$ then $g_n|_{x(\R)} \equiv 1$, $dg_n|_{x(\R)} \equiv 0$ $\forall n$.

\item[(f)] For every $n$, $\util_n(\R\times\R/\Z)$ defines a Birkhoff section for the Reeb flow of~$\lambda_n$.
In particular, $u_n$ is a proper embedding into $Y \setminus (x^+_n(\R) \cup x^-_n(\R))$ transverse to the Reeb vector field of $\lambda_n$.

\item[(g)] Denote asymptotic operators by $A^\pm_n = A_{\gamma^\pm_n}$, $A^\pm=A_{\gamma^\pm}$, and asymptotic eigenvalues of $\util_n$ at $s=\pm\infty$ by $\mu^\pm_n \in \sigma(A^\pm_n)$. For some $\delta>0$ we have $[-\delta,0) \cap \sigma(A^+) = \emptyset$, $(0,\delta] \cap \sigma(A^-) = \emptyset$, $\ind_\delta(\util_n)=2$, $\mu^+_n < -\delta$, $\mu^-_n > \delta$, $p^{-\delta}(\gamma_n^+) = p^\delta(\gamma_n^-) = 1$ $\forall n$.
\end{enumerate}

\end{lemma}

Before giving the proof, a couple of remarks are in order.
The almost complex structures $\jtil_n$ are constructed from $J_n,\lambda_n$ as in Section~\ref{sec_prelim_2}.
By~(a), the cylinders $\util_n$ have an asymptotic formula at both punctures.
By~(d), $\gamma^+_n$ is geometrically distinct from $\gamma^-_n$, and $\gamma^+$ is geometrically distinct from $\gamma^-$ for all $n$ large enough.
Hence, each $\util_n$ is not a trivial cylinder and has a non-trivial asymptotic formula at both punctures when $n$ is large enough.
This allows one to consider asymptotic eigenvalues in~(g).

\begin{proof}[Proof of Lemma~\ref{lem:fromprop}]

This follows from Proposition~\ref{prop:main}; all but $(g)$ are essentially immediate so we only explain (g).
 The reasoning to obtain~(g) from Proposition~\ref{prop:main} is as follows.
We analyze only the positive puncture, the analysis of the negative puncture is left to the reader.
Consider a trivialization of $\xi$ along the simple Reeb orbit underlying $\gamma^+$.
This trivialization can be iterated, thus inducing a trivialization of~$\xi$ along $\gamma^+$, and also homotopy classes of trivializations of~$\xi$ along $\gamma^+_n$ for all $n$ large enough.
In the argument that follows, winding numbers and rotation numbers are computed with respect to these trivializations.
By the asymptotic formula for~$\util_n$ at both punctures, the winding number $\wind(\mu^+_n)$ of $\mu^+_n$ is smaller than the rotation number $\theta_n$ of $\gamma^+_n$. 
Note that $\theta_n \to \theta$ where $\theta$ is the rotation number of~$\gamma^+$.
Moreover, up to a subsequence, there exists an eigenvalue of $A^+$ such that $\mu^+_n \to \mu$.
From $\mu^+_n<0$ we get $\mu\leq0$.
If $\gamma^+$ is nondegenerate then $\theta \not\in \Z$, $\mu<0$, and existence of the gap $[-\delta,0)$ in the spectrum of $A^+$ for some $0<\delta<-\mu$ follows.
If $\gamma^+$ is degenerate then $\theta \in \Z$ is the winding number~$\wind(0)$ of the eigenvalue $0$  of $A^+$.
In this case Proposition~\ref{prop:main} tells us that $\theta_n-\theta<0$, and from $\wind(\mu^+_n)<\theta_n$ we get $\wind(\mu^+_n) \leq \theta -1$.
From this we can conclude that $\mu<0$, since otherwise $\wind(\mu^+_n)$ would converge to $\wind(0) = \theta$ and $\wind(\mu^+_n)$ would be equal to the integer $\theta$ for $n$ large enough.
Again, existence of the gap $[-\delta,0)$ in the spectrum of $A^+$ for some $0<\delta<-\mu$ follows.
The only statement left to be proved for~(g) is $p^{-\delta}(\gamma_n^+) = p^\delta(\gamma_n^-) = 1$.
The cylinders given by Proposition~\ref{prop:main} have high multiplicity ends; this follows from Lemma~\ref{lem:threshold} as explained in Part 2 of the proof of Proposition~\ref{prop:main}. Since these cylinders are part of a curve counted by the U-map in ECH, and such a curve connects orbit sets that are ECH generators and can only have hyperbolic orbits with multiplicity one, we get that $\gamma^+_n$ is elliptic. 
\end{proof}

Recall the moduli space $$ \M = \M_{J,0}(\gamma^+,\gamma^-) $$ defined in Section~\ref{sec_moduli_spaces}.
It is 
crucial to keep in mind that, by definition, finite-energy maps representing curves in $\M$ have nondegenerate punctures.
Hence, by Theorem~\ref{thm_nondeg_implies_asymp_formula}, these maps have an asymptotic formula at both punctures.
Since $\gamma^+$ is geometrically distinct from $\gamma^-$, curves in $\M$ are not trivial cylinders, and we further get a non-trivial asymptotic formula at both punctures.
In particular, at the positive puncture of any curve in $\M$ there is an asymptotic eigenvalue in $(-\infty,-\delta)$, and at the negative puncture there is an asymptotic eigenvalue in $(\delta,+\infty)$.
In what follows $\M$ is equipped with the quotient topology induced by the $C^\infty_\loc$-topology on the space of smooth maps $\R\times\R/\Z \to \R\times Y$.  

The following is a standard fact we will need:

\begin{lemma}
\label{rmk_ends}
Let $C_n \in \mathcal{M}$ be a convergent sequence.
Let $\vtil_n=(b_n,v_n)$ be representatives of $C_n$ that $C^\infty_{\rm loc}$-converge to a representative of the limiting curve.
Then, for every neighborhood $N$ of $x^+(\R) \cup x^-(\R)$ in~$Y$ we find $s_*>0$ and~$n_*$ such that if $|s| \geq s_*$ and $n\geq n_*$ then $v_n(\{s\}\times\R/\Z) \subset N$.
\end{lemma}
\begin{proof}
The proof follows a standard recipe: if not then we would find $n_j \to \infty$, $s_j$ with $|s_j|\to+\infty$, and $t_j \in \R/\Z$ such that $v_{n_j}(s_j,t_j)$ converges to a point in the complement of the union of the finitely many simple Reeb orbits of~$\lambda$.
Reparametrizing to center the maps at $(s_j,t_j)$, and using the fact no bubbling-off points can arise due to low $d\lambda$-area, we would obtain a cylinder with zero $d\lambda$-area through a point not in any periodic orbit, which is impossible.
\end{proof}

With the above preliminaries behind us, we can now prove the crucial fact that the moduli space $\M$ has a compact component; we prove this, along with a few other useful facts, in the following lemma.

\begin{lemma}
\label{prop_moduli_spaces}
The space $\M$ has a compact connected nonempty component $\mathcal{Y}$ which is a smooth $1$-manifold whose elements are represented by maps $\util = (a,u) : \R \times \R/\Z \to \R \times Y$ that are $\tilde J$-holomorphic and have the following properties:
\begin{itemize}
\item[(i)] $u$ is an immersion transverse to the Reeb vector field of $\lambda$.
\item[(ii)] The map $\util$ has a non-trivial asymptotic formula at both punctures, is asymptotic to~$\gamma^+$ at~$s=+\infty$, and to~$\gamma^-$ at~$s=-\infty$. 
The asymptotic eigenvalue $\mu^+$ at $+\infty$ satisfies $$ \mu^+ < -\delta \qquad \wind(\mu^+) = \wind(\nu^{<-\delta}(\gamma^+)) $$ in some (hence any) trivialization along $\gamma^+$.
The asymptotic eigenvalue~$\mu^-$ at $-\infty$ satisfies $$ \mu^- > \delta \qquad \wind(\mu^-) = \wind(\nu^{\geq\delta}(\gamma^-)) $$  in some (hence any) trivialization along $\gamma^-$.
\item[(iii)] ${\rm ind}_\delta(\util) = 2$, $p^{-\delta}(\gamma^+) = p^\delta(\gamma^-) = 1$.
\end{itemize}
\end{lemma}

\begin{proof}

In the course of the proof, it will be helpful to establish the following two additional properties:

\begin{itemize}
\item[(iv)] $u(\R\times\R/\Z) \subset Y \setminus (x^+(\R) \cup x^-(\R))$ and every Reeb orbit of $\lambda$ of period $\leq T^+$ has strictly positive intersection number with $u$.
\item[(v)] If $\beta$ is a loop in $\R \times \R/\Z$ then $u\circ\beta$ has intersection number zero with~$u$.
\end{itemize}
We now explain why all of (i) - (v) hold.

The curves $\util_n$ satisfy assumptions~($*$), (B), (C), (D), (E) and (G) from Section~\ref{sec_asymptotic_analysis}.
Lemma~\ref{lemma_compactness} implies that $\{\util_n\}$ is $C^\infty_\loc$-convergent to a $\jtil$-holomorphic finite-energy map $\util=(a,u)$ in~\eqref{limit_compactness_lemma} with a positive puncture at $s=+\infty$ and a negative puncture at $s=-\infty$, and that $u(0,0)$ does not belong to a  Reeb orbit of $\lambda$.
This shows that also hypotheses (A) and (F) from Section~\ref{sec_asymptotic_analysis} are satisfied.
Proposition~\ref{prop_main_asymptotic_analysis} implies that all properties (i)-(v) as in the statement hold for the map~$\util$.
In particular, $\util$ represents a curve in~$\M$.

We denote by $\mathcal{Y}$ the connected component of $\M$ containing the curve represented by the map $\util$ obtained above.
Properties (iv)-(v) are shared by all curves in $\mathcal{Y}$; this follows from the definition of $\M$ and from Lemma~\ref{lemma_ineq_wind_pi}.
Lemma~\ref{lemma_automatic_transversality} gives $\mathcal{Y}$ the structure of a smooth $1$-dimensional manifold when equipped with a topology induced from the functional analytic set-up used in~\cite{props3} to define the Fredholm theory discussed in Section~\ref{sec_moduli_spaces}.
It follows essentially from Proposition~\ref{prop_asymp_control} that the topology on $\M$ agrees with the quotient topology induced from the $C^\infty_{\loc}$-topology.
Thus $\mathcal{Y}$ is a smooth $1$-manifold with the latter topology.

We will show that $\mathcal{Y}$ is compact.
To this end we will first show that every curve in $\mathcal{Y}$ is represented by a map $\vtil = (b,v)$ satisfying properties~(i)-(v).
We already know this to be the case for properties~(iv)-(v).
Lemma~\ref{lemma_ineq_wind_pi} gives~(i) for all curves in~$\mathcal{Y}$.

We claim that every curve in~$\mathcal{Y}$ projects to $Y \setminus (x^+(\R) \cup x^-(\R))$.
To this end, consider the set $\mathcal{Z} \subset \mathcal{Y}$ of curves that project to $Y \setminus (x^+(\R) \cup x^-(\R))$.
We will show that $\mathcal{Z}$ is open and closed.

Let us first show that $\mathcal{Z}$ is open.
Let $C \in \mathcal{Z}$ be represented by a map $\vtil = (b,v)$ such that $v(\R \times \R/\Z)$ does not intersect $x^+(\R)\cup x^-(\R)$.
Suppose, by contradiction, that there exists $C_n \in \mathcal{Y}$, $C_n \to C$ represented by maps $\vtil_n=(b_n,v_n)$, and $(s_n,t_n)$ for which we have that $v_n(s_n,t_n) \in x^+(\R)\cup x^-(\R)$.
Up to translating in the $\R$-direction and reparametrizing, we can assume $\vtil_n \to \vtil$ in $C^\infty_\loc$.
Up to a subsequence, we can further assume that 
one of the following holds: $v_n(s_n,t_n) \in x^+(\R) \ \forall n$ or $v_n(s_n,t_n) \in x^-(\R) \ \forall n$. 
Assume $v_n(s_n,t_n) \in x^+(\R) \ \forall n$, the other case is handled analogously.
We claim that $\inf_n s_n > -\infty$.
If not, since $\gamma^+$ and $\gamma^-$ are geometrically distinct,
we find $s'_n \to -\infty$, $t'_n$ such that $v_n(s'_n,t'_n) \to p$ for some $p \in Y \setminus (x^+(\R)\cup x^-(\R))$.
This contradicts Lemma~\ref{rmk_ends}.
Now we claim that $\sup_n s_n < +\infty$.
Consider a tubular neighborhood $N^+$ of $x^+(\R)$ equipped with coordinates $N^+ \simeq \R/\Z \times \C$ such that $x^+(T^+t) \simeq (t,0)$.
We can define the linking number, relative to the coordinates on $N^+$, of a loop $(c_1(t),c_2(t))$ in $\R/\Z \times (\C\setminus\{0\}) \simeq N^+ \setminus x^+(\R)$ with $x^+(T^+\cdot)$ as the winding number of~$c_2(t)$.
Since $\vtil$ has a non-trivial asymptotic formula at $s=+\infty$, we find $h>0$ such that if $s \geq h$ then the loop $v(s,\cdot)$ lies in $N^+ \setminus x^+(\R)$ and its linking number is $\wind(\nu^{\leq -\delta}(\gamma^+))$, where the latter winding is computed in a trivialization along $\gamma^+$ induced from the coordinates on $N^+$.
Since $\vtil_n \to \vtil$, we can find $n_0$ such that if $n \geq n_0$ then $v_n(h,\cdot)$ lies in $N^+ \setminus x^+(\R)$ and has linking number $\wind(\nu^{\leq -\delta}(\gamma^+))$.
However, if $\sup_ns_n=+\infty$ then we can assume, up to selection of a subsequence, that $s_n \to +\infty$, in particular there is $n_1$ such that $n \geq n_1$ implies that $s_n > h$.
By the non-trivial asymptotic formula for $\vtil_n$ and 
properties (iv)-(v) for~$\vtil_n$, we find for each $n \geq n_1$ some $\tau_n > s_n$ such that $v_n(\tau_n,\cdot)$ lies in $N^+ \setminus x^+(\R)$ and has linking number $\wind(\nu^{\leq -\delta}(\gamma^+))$.
Summarizing, if $n \geq \max\{n_0,n_1\}$ then both loops $v_n(h,\cdot)$, $v_n(\tau_n,\cdot)$ lie in $N^+ \setminus x^+(\R)$ and have the same linking number with $x^+(T^+\cdot)$.
By positivity of intersections, the intersection number of $\vtil_n|_{[h,\tau_n]\times\R/\Z}$ with the trivial cylinder over $x^+(\R)$ is strictly positive, when $n \geq \max\{n_0,n_1\}$.
It follows that $v_n(h,\cdot)$ and $v_n(\tau_n,\cdot)$ have different linking numbers with $x^+(T^+\cdot)$.
This contradiction concludes the proof that $\sup_ns_n<+\infty$.
We have now proved that $s_n$ is bounded.
Up to a subsequence, $(s_n,t_n) \to (s_*,t_*)$ and $v(s_*,t_*) \in x^+(\R)$, which is a contradiction the assumption on~$\vtil$.
Hence,~$\mathcal{Z}$ is open.

The complement of $\mathcal{Z}$ in $\mathcal{Y}$ is obviously open, by the transversality between projections to~$Y$ of curves in $\mathcal{Y}$ and the Reeb vector field (Lemma~\ref{lemma_ineq_wind_pi}).
By connectedness, $\mathcal{Z} = \mathcal{Y}$.
We have proved that all curves in $\mathcal{Y}$ project to~$Y \setminus (x^+(\R) \cup x^-(\R))$, and satisfy properties (i), (iv) and (v).
Properties~(iv) and~(v) for representatives of all curves in $\mathcal{Y}$ now follow from the fact that there is one curve in $\mathcal{Y}$ represented by a map satisfying~(i)-(v).

Finally, compactness of $\mathcal{Y}$ is a direct application of Lemma~\ref{lemma_compactness} and Proposition~\ref{prop_main_asymptotic_analysis} to the constant sequences $g_n=1$, $J_n=J$, in view of the properties obtained above for the maps representing curves in $\mathcal{Y}$.
\end{proof}

At this point, one could attempt to argue by proving an analogue of \cite[Prop. 3.2]{CGHP}; however, one would have to show that the curves in $\mathcal{Y}$ are somewhere injective, and a priori the $\tilde{u}_n$ need not limit to something somewhere injective.
To bypass this point, we implement a small workaround, namely we lift to the moduli space of cylinders via the evaluation map.    The remainder of the paper fills in the details of this and proves Theorem~\ref{thm:main_intro}.  

By the uniformization theorem, 
$\M$ can be described as the set of equivalence classes of finite-energy $\jtil$-holomorphic maps $\util : \R \times \R/\Z \to \R \times Y$ asymptotic to $\gamma^+$ at the positive puncture $s=+\infty$ and to $\gamma^-$ at the negative puncture $s=-\infty$, with an asymptotic formula at both punctures. 
Here $\R \times \R/\Z$ is equipped with the conformal structure obtained by pulling back multiplication by~$i$ in $\C\setminus\{0\}$ via the map $(s,t) \mapsto e^{2\pi(s+it)}$.
Two such maps $\util_0 = (a_0,u_0)$, $\util_1=(a_1,u_1)$ are declared equivalent if there exist $(\Delta s,\Delta t) \in \R\times\R/\Z$, $c\in\R$ such that 
\begin{equation}
\label{rep_curve_moduli_cyls}
(a_1(s,t)+c,u_1(s,t)) = (a_0(s+\Delta s,t+\Delta t),u_0(s+\Delta s,t+\Delta t)) \quad \forall (s,t) \, .
\end{equation}
The moduli space of cylinders with one marked point $\M^{(1)}$ is defined as a set of equivalence classes of pairs $(\util,z)$ where $\util$ is a finite-energy cylinder that represents a curve in $\M$ and $z \in \R \times \R/\Z$.
Two pairs $(\util_0,z_0=(s_0,t_0))$, $(\util_1,z_1=(s_1,t_1))$ are declared equivalent if there exist $(\Delta s,\Delta t) \in \R\times\R/\Z$, $c \in \R$ such that~\eqref{rep_curve_moduli_cyls} holds and $(s_0,t_0) = (s_1 + \Delta s,t_1 + \Delta t)$. 
The equivalence class of $(\util,z)$ will be denoted by $[\util,z]$.
There is a forgetful map $$ F : \M^{(1)} \to \M \, , \qquad F([\util,z]) = [\util] $$ 
and an evaluation map $$ {\rm ev} : \M^{(1)} \to Y \, , \qquad {\rm ev}([\util=(a,u),z]) = u(z) \, . $$
In the following we consider the connected component $\mathcal{Y} \subset \M$ given by Lemma~\ref{prop_moduli_spaces}.
It follows from this proposition that $$ {\rm ev}(F^{-1}(\mathcal{Y})) \subset Y \setminus (x^+(\R) \cup x^-(\R)) \, . $$
Moreover, the automatic transversality from Lemma~\ref{lemma_automatic_transversality} and Lemma~\ref{prop_moduli_spaces} together guarantee that $F^{-1}(\mathcal{Y})$ has the structure of a smooth $3$-manifold such that $F$ is a trivial fibration over~$\mathcal{Y}$ with fiber $\R\times\R/\Z$.
Note also that~$\mathcal{Y}$ is diffeomorphic to the circle, since it is a compact connected smooth~$1$-manifold without boundary.

\begin{lemma}
\label{lemma_ev_proper_local_diffeo}
The map ${\rm ev}|_{F^{-1}(\mathcal{Y})} : F^{-1}(\mathcal{Y}) \to Y \setminus (x^+(\R) \cup x^-(\R))$ is a local diffeomorphism and a proper map.
\end{lemma}

\begin{proof}
First we prove that ${\rm ev}$ is a local diffeomorphism.
Consider an arbitrary element of $F^{-1}(\mathcal{Y})$, and represent it by a pair $(\util=(a,u),z_0)$.
The tangent space at $\util$ of the corresponding space of holomorphic maps modulo reparametrization can be identified with the $2$-dimensional vector space $\ker D_\util$, where $D_\util$ is the linearized Cauchy-Riemann operator at~$\util$ considered in Section~\ref{sec_moduli_spaces}.
Differentiating the $(\R,+)$ action at $\util$ we get $\zeta_0 \in \ker D_\util$, $\zeta_0 \neq 0$.
The tangent space $T_{[\util]}\mathcal{Y}$ of $\mathcal{Y}$ at $[\util] = F([\util,z_0])$ can be identified with $(\ker D_\util) / \R\zeta_0$.
The proof of Lemma~\ref{lemma_automatic_transversality} shows that sections in $\ker D_{\util}$ either vanish identically or are nowhere vanishing.
Any section $\zeta \in \ker D_\util$ has components $\zeta = (\zeta^\R,\zeta^Y)$ according to the splitting $T(\R \times Y) = T\R \times TY$.
We claim that if $z \in \R\times\R/\Z$, $\zeta \in \ker D_\util$, $\zeta\neq0$ and $\zeta^Y(z) \in {\rm im} \ du(z)$ then $\zeta(z) =(\alpha,0) + d\util(z)v$ for some $\alpha \in \R \setminus\{0\}$, $v \in T_z(\R\times\R/\Z)$.
In fact, if $\zeta^Y(z) =du(z)v$ for some $v$ then $\zeta(z) = (\alpha,0) + d\util(z)v$ where $\alpha = \zeta^\R(z)-da(z)v$.
We must have $\alpha\neq0$ since $\zeta(z) \neq 0$ is in the normal bundle.
The claim is proved.
Let $\zeta_1 \in \ker D_\util$ be such that $\{\zeta_0,\zeta_1\}$ is a basis of $\ker D_\util$.
We claim that $\zeta_1^Y(z) \not\in {\rm im} \ du(z)$ for every $z$.
We argue indirectly for this. Suppose we find~$z_*$ such that $\zeta_1^Y(z_*) \in {\rm im} \ du(z_*)$.
From the claim above $\zeta_1(z_*) = (\alpha_1,0) + V_1$ for some $V_1 \in {\rm im} \ d\util(z_*)$ and some $\alpha_1 \in\R\setminus\{0\}$.
Similarly, since $\zeta_0^Y(z)$ belongs to ${\rm im} \ du(z)$ for every $z$, we can write $\zeta_0(z_*) = (\alpha_0,0) + V_0$ with $V_0 \in {\rm im} \ d\util(z_*)$ and $\alpha_0 \in\R\setminus \{0\}$.
It follows that $\alpha_0\zeta_1(z_*)-\alpha_1\zeta_0(z_*) \in {\rm im} \ d\util(z_*)$, from where it follows that $\alpha_0\zeta_1(z_*)-\alpha_1\zeta_0(z_*)=0$. 
Hence, $\alpha_0\zeta_1-\alpha_1\zeta_0\equiv 0$ since it vanishes at one point, and this is a contradiction.
Surjectivity of the differential of ${\rm ev}$ at $[\util,z]$ follows since $T_{[\util]}\mathcal{Y}$ is spanned by $\zeta_1 + \R\zeta_0$, and~$u$ is an immersion.
The proof that ${\rm ev}$ is a local diffeomorphism is complete.

To prove that ${\rm ev}|_{F^{-1}(\mathcal{Y})}$ is a proper map into $Y \setminus (x^+(\R) \cup x^-(\R))$, we need to consider a sequence $[\util_n=(a_n,u_n),z_n] \in \mathcal{Y}$ satisfying $u_n(z_n) \to p$, $p \not\in x^+(\R) \cup x^-(\R)$, and prove that some subsequence of $[\util_n,z_n]$ converges in~$\mathcal{Y}$.
From the compactness of $\mathcal{Y}$ given by Lemma~\ref{prop_moduli_spaces} we can assume, up to selection of a subsequence, that $[\util_n] \to [\util] \in \mathcal{Y}$.
Up to reparametrizing and translating in the $\R$-direction, we can further assume $\util_n \to \util$ in $C^\infty_\loc$.
Let $V$ be a closed neighborhood of $x^+(\R) \cup x^-(\R)$ such that $p \not\in V$.
From Lemma~\ref{rmk_ends} we find $\ubar{s} > 0$ such that $|s| > \ubar{s} \Rightarrow u_n(s,t) \in V \ \forall (n,t)$.
Hence $z_n = (s_n,t_n)$ satisfies $|s_n| \leq \ubar{s}$ when $n\gg1$ and, up to a subsequence, we may further assume that $z_n \to z_*$ for some $z_* \in \R\times\R/\Z$.
Hence $[\util_n,z_n] \to [\util,z_*]$.
\end{proof}

We can now put all of this together to prove the main theorem of our paper.

\begin{proof}[Proof of Theorem~\ref{thm:main_intro}]

Recall that we are assuming that the Reeb flow of $\lambda$ has finitely many simple Reeb orbits.  We will now conclude that there are exactly two simple Reeb orbits.

Recall that to each Reeb orbit $\gamma=(x,T)$ of $\lambda$, and each (homotopy class of) symplectic trivialization $\tau$ of~$\xi_\gamma$, there is an associated rotation number ${\rm rot}_\tau(\gamma) \in \R$ uniquely characterized by $$ {\rm rot}_\tau(\gamma) = \lim_{k\to\infty} \frac{\CZ_{\tau^k}(\gamma^k)}{2k}  $$ where $\tau^k$ denotes the trivialization of $\xi_{\gamma^k}$ induced from $\tau$.
It is not assumed here that~$\tau$ comes from a trivialization over the simple Reeb orbit underlying $\gamma$.

Denote $L = x^+(\R) \cup x^-(\R)$.
Denote the primitive period of $x^+$ by $\ubar{T}^+$, and the multiplicity of $\gamma^+$ by $k_+$.
The simple Reeb orbit underlying $\gamma^+$ is $\ubar{\gamma}^+ = (x^+,\ubar{T}^+)$, and we have $\gamma^+ = (\ubar{\gamma}^+)^{k_+} = (x^+,T^+ = k_+\ubar{T}^+)$.

Consider a Martinet tube $\Psi : U^+ \to \R/\Z \times B$ around $\gamma^+$, with coordinates $(\theta,x_1+ix_2)$.
If the argument of $x_1+ix_2$ is denoted by $\varphi \in \R/2\pi\Z$ then, by~(iv) in Lemma~\ref{prop_moduli_spaces}, the slope of $\util$ at the positive puncture with respect to coordinates $(\theta,\frac{\varphi}{2\pi}) \in \R/\Z \times \R/\Z$ is $(k_+,\wind_{\tau^{k_+}}(\nu^{<-\delta}(\gamma^+)))$, where~$\tau$ is the trivialization of $\xi_{\ubar{\gamma}^+}$ induced by the frame $\{\partial_{x_1},\partial_{x_2}\}$, and $\tau^{k_+}$ is the trivialization of $\xi_{\gamma^+}$ induced by $\tau$.

The Reeb vector field~$R$ of $\lambda$ is complete on $Y \setminus L$.
For each $\util = (a,u)$ that represents a curve in $\mathcal{Y}$, the $Y$-component $u$ is transverse to $R$. 
Hence, a trajectory of $R$ has non-negative intersection number with $u$, and any intersection with $u(\R\times\R/\Z)$ counts positively to the intersection number.
We split the remaining arguments into a few claims.  

We want to lift the Reeb flow via the evaluation map, and find a global cross-section for the lifted flow. Our first claim collects what we need to know away from the boundary:

\begin{itemize}
\item[(C1)] For each open neighborhood $W$ of $L$ there exists $\ell_W > 0$ such that if $b - a > \ell_W$ then for every $[\util=(a,u)] \in \mathcal{Y}$ and every piece of trajectory $c : [a,b] \to Y \setminus W$ we have $c([a,b]) \cap u(\R\times\R/\Z) \neq \emptyset$.
\end{itemize}
To prove (C1), consider the map ${\rm ev} : F^{-1}(\mathcal{Y}) \to Y \setminus L$.
This map is a proper local diffeomorphism, by Lemma~\ref{lemma_ev_proper_local_diffeo}.
Denote by $\tilde R$ the lift of $R$ by ${\rm ev}$.
Hence $\tilde R$ is a complete vector field.
Moreover, $F : F^{-1}(\mathcal{Y}) \to \mathcal{Y}$ is a fibration over~$\mathcal{Y}$ and $\tilde R$ is transverse to the fibers.
Therefore we can find a global parameter $s \in \R/\Z$ on $\mathcal{Y}$ satisfying $\iota_{\tilde R}d(s \circ F) > 0$.
Since $Y \setminus W$ is compact in $Y \setminus L$, ${\rm ev}^{-1}(Y \setminus W)$ is a compact subset of $F^{-1}(\mathcal{Y})$.
Let $\Delta$ be the infimum of $\iota_{\tilde R}d(s \circ F)$ on ${\rm ev}^{-1}(Y \setminus W)$.
By compactness, this infimum is a minimum.
Hence, $\Delta>0$.
Let $c:[a,b]\to Y\setminus W$ be a piece of a trajectory of $R$.
Consider a lift $\tilde c:[a,b] \to {\rm ev}^{-1}(Y \setminus W)$ of $c$ by ${\rm ev}$.
Then $\tilde c$ is a piece of trajectory $\tilde R$.
Moreover, a lift $[a,b] \to \R$ of $s \circ F \circ \tilde c$ oscillates more than $\Delta(b-a)$.
This oscillation is strictly larger than $1$ if $b-a > \Delta^{-1}$.
Hence, $\tilde c$ intersects every fiber of $F$ when $b-a > \Delta^{-1}$.
Claim~(C1) is proved if we set $\ell_W = \Delta^{-1}$.

Our next claim collects what we need to know near the boundary.

\begin{itemize}
\item[(C2)] There is a neighborhood $W_*$ of $L$ and $\ell_*>0$ such that if $b-a>\ell_*$ then for every $[\util=(a,u)] \in \mathcal{Y}$ and every piece of Reeb trajectory $c : [a,b] \to W_*$ we have $c([a,b]) \cap u(\R\times\R/\Z) \neq \emptyset$.
\end{itemize}
We prove (C2).
For every $T>0$ there is an open neighborhood $W_*^+ = W_*^+(T)$ of $x^+(\R)$ with coordinates $(\theta,z=x_1+ix_2=|z|e^{i\varphi}) \in \R/\Z \times \C$ such that:
\begin{itemize}
\item $x^+(t) = (t/\ubar{T}^+,0)$.
\item $d\theta \wedge dx_1 \wedge dx_2$ is positive relatively to $\lambda \wedge d\lambda$.
\item For every piece Reeb trajectory of time-length $T$ contained in $W_*^+$, a lift to $\R$ of $\varphi/2\pi$ along the trajectory oscillates not less than ${\rm rot}(\ubar{\gamma}^+)\lfloor T/\ubar{T}^+ \rfloor - 1$.
\end{itemize}
Here ${\rm rot}(\ubar{\gamma}^+)$ is computed with respect to the frame $\{\partial_{x_1},\partial_{x_2}\}$.
As remarked above, for an arbitrary element of~$\mathcal{Y}$ represented by a pseudo-holomorphic map $\util=(a,u)$, if $\rho\gg1$ the loop $t \in \R/\Z \mapsto u(\rho,t) \in W_*^+ \setminus x^+(\R)$ has slope $(k_+,\wind(\nu^{<-\delta}(\gamma_+))$ in the basis $\{d\theta,\frac{d\varphi}{2\pi}\}$.
Here the winding is computed in a trivialization aligned with the frame $\{\partial_{x_1},\partial_{x_2}\}$.
Hence, the argument of a Reeb trajectory of time-length~$T$ contained in $W_*^+$ relative to $u([\rho,+\infty) \times \R/\Z)$ has a lift to $\R$ that oscillates 
$$ 
\begin{aligned} 
D &= 2\pi \left( {\rm rot}(\ubar{\gamma}^+) - \frac{\wind(\nu^{<-\delta}(\gamma_+))}{k_+} \right) \lfloor T/\ubar{T}^+ \rfloor - O(1) \\ 
& = 2\pi \left( \frac{{\rm rot}(\gamma^+) - \wind(\nu^{<-\delta}(\gamma_+))}{k_+} \right) \lfloor T/\ubar{T}^+ \rfloor - O(1) \ \, .
\end{aligned} 
$$
Here $O(1)$ denotes a positive constant independent of $T$ and of $\util$.
We know that 
$$ {\rm rot}(\gamma^+) > \wind(\nu^{<-\delta}(\gamma_+)) $$ 
since $\nu^{<-\delta}(\gamma_+)$ is a negative eigenvalue of the asymptotic operator at $\gamma_+$.
Hence, we can fix~$T$ such that $D>3\pi$, and conclude that any piece of Reeb trajectory contained in $W_*^+=W_*^+(T)$ of time-length $T$ intersects the image of $u$.
An analogous argument at the negative puncture will produce a neighborhood $W_*^-$ of $x^-(\R)$ with similar properties. 
Take $W_* = W_*^+ \cup W_*^-$ and~(C2) is proved.

We now put (C1) and (C2) together to prove the following:

\begin{itemize}
\item[(C3)] There exists $\ell>0$ such that if $b-a>\ell$ then for every piece of Reeb trajectory $c:[a,b] \to Y\setminus L$ and every $[\util=(a,u)] \in \mathcal{Y}$ we have $c([a,b]) \cap u(\R\times\R/\Z) \neq \emptyset$.
\end{itemize}
Let $\phi^t$ denote the Reeb flow of $\lambda$.
To prove (C3), consider a neighborhood $W_*$ of $L$ and $\ell_*>0$ given by (C2). Let $W \subset W_*$ be a neighborhood of $L$ such that if $p\in W$ then $\phi^{[0,\ell_*+1]}(p) \subset W_*$.
Let $\ell_W>0$ be given by (C1).
Consider $p \in Y \setminus L$ arbitrary.
If $\phi^{[0,\ell_W+1]}(p) \cap W = \emptyset$ then, by~(C1), $\phi^{[0,\ell_W+1]}(p)$ intersects the projection to $Y$ of every curve in $\mathcal{Y}$.
If $\phi^{[0,\ell_W+1]}(p) \cap W \neq \emptyset$ then we have a piece of trajectory inside $\phi^{[0,\ell_W+\ell_*+2]}(p)$ of time-length $\ell_*+1$ that is contained in $W_*$.
By (C2) it must intersect the projection to $Y$ of every curve in $\mathcal{Y}$.
We proved that one can take $\ell = \ell_W+\ell_*+2$ for (C3).

Let us now complete the proof.
Lift the Reeb flow $\phi^t$ of~$\lambda$ from $Y\setminus L$ to a flow $\psi^t$ on $F^{-1}(\mathcal{Y})$ via the proper local diffeomorphism ${\rm ev}|_{F^{-1}(\mathcal{Y})}$. 
By Lemma~\ref{prop_moduli_spaces}, the flow $\psi^t$ is transverse to the fibers of the fibration $F:F^{-1}(\mathcal{Y}) \to \mathcal{Y}$.
By (C3) each fiber is a global cross-section for~$\psi^t$.
Lemma~\ref{prop_moduli_spaces}~(i) also tells us that the~$\phi^t$-invariant $2$-form $d\lambda$ lifts to a~$\psi^t$-invariant $2$-form on $F^{-1}(\mathcal{Y})$ that defines an area-form of finite total area on each fiber of the map~$F$.
Franks' results from~\cite{Franks} guarantee that~$\psi^t$ has either infinitely many simple periodic orbits, or no simple periodic orbits at all.

Periodic orbits of~$\psi^t$ are mapped onto periodic orbit of~$\phi^t$ by~$F$.
Conversely, properness of the local diffeomorphism ${\rm ev}|_{F^{-1}(\mathcal{Y})}$ implies that a periodic orbit of~$\phi^t$ in $Y\setminus L$ lifts to finitely many periodic orbits of~$\psi^t$.
It follows that $\psi^t$ has finitely many simple periodic orbits, hence no periodic orbits at all.
This implies that~$\phi^t$ has no periodic orbits on~$Y\setminus L$.
\end{proof}

\bigskip

{\footnotesize

{\sc Dan Cristofaro-Gardiner}

University of Maryland, College Park

{\em dcristof@umd.edu}

\medskip

{\sc Umberto Hryniewicz}

RWTH Aachen, Jakobstrasse 2, Aachen 52064, Germany

{\em hryniewicz@mathga.rwth-aachen.de}

\medskip

{\sc Michael Hutchings}

University of California, Berkeley

{\em hutching@math.berkeley.edu\/}

\medskip

{\sc Hui Liu}

School of Mathematics and Statistics, Wuhan University, Wuhan 430072, Hubei, P.R. China

{\em huiliu00031514@whu.edu.cn}

}


\begin{thebibliography}{99}

\bibitem{paiva} J. C. \'Alvarez Paiva, \textit{Some problems on Finsler geometry}. Handbook of
Differential Geometry. {\bf 2} (2006), 1–33, Elsevier/North-Holland, Amsterdam.

\bibitem{Bangert} V. Bangert, \textit{On the lengths of closed geodesics on almost round spheres}, Math. Z. {\bf 191\/} (1986), 549--558.

\bibitem{Bangert2} V. Bangert, \textit{On the existence of closed geodesics on two-spheres}, Internat. J. Math., {\bf 4}(1) (1993), 1--10.

\bibitem{BangertLong} V. Bangert and Y. Long, \textit{The existence of two closed geodesics on every Finsler 2-sphere},  Math. Ann. {\bf 346} (2010), 335-366.

\bibitem{BurnsMatveev} K. Burns and S. Matveev, \textit{Open problems and questions about geodesics}, Ergodic Theory Dynam. Systems {\bf 41\/}, no. 3 (2021), 641--684.

\bibitem{CDHR} V. Colin, P. Dehornoy, U. Hryniewicz, and A. Rechtman, {\em Generic properties of 3-dimensional Reeb flows: Birkhoff sections and entropy\/}, arXiv:2202.01506.

\bibitem{CDR} V. Colin, P. Dehornoy, and A. Rechtman, \textit{On the existence of supporting broken book decompositions for contact forms in dimension $3$}, Invent. Math. {\bf 231\/} (2023), 1489--1539.

\bibitem{CGHHL} D. Cristofaro-Gardiner, U. Hryniewicz, M. Hutchings and H. Liu, \textit{Contact three-manifolds with exactly two simple Reeb orbits}, Geom. Topol. {\bf 27} (2023), 3801--3831.

\bibitem{CGH} D. Cristofaro-Gardiner and M. Hutchings, \textit{From one Reeb orbit to two.} J. Differential Geom. {\bf 102} (2016), 25--36.

\bibitem{CGHP} D. Cristofaro-Gardiner, M. Hutchings and D. Pomerleano, {\em Torsion contact forms in three dimensions have two or infinitely many Reeb orbits}, Geom. Topol. {\bf 23} (2019), 3601--3645.

\bibitem{ECHasymptotics} D. Cristofaro-Gardiner, M. Hutchings and V. G. B. Ramos, \textit{The asymptotics of ECH capacities.} Invent. Math. {\bf 199\/} (2015), 187--214.


\bibitem{dw} A. Doan and T. Walpuski, {\em Castelnuovo's bound and rigidity in almost complex geometry\/}, Adv. Math. {\bf 379} (2021).


\bibitem{EGH} Y. Eliashberg, Y, A. Givental and H. Hofer, \textit{Introduction to symplectic field theory}, Geom. Funct. Anal., 560--673 (2000)

\bibitem{Franks} J. Franks, \textit{Geodesics on $S^2$ and periodic points of annulus homeomorphisms}, Invent. Math. 108, no. 2, 403--418 (1992)

\bibitem{Grayson} M. A. Grayson, \textit{Shortening embedded curves}, Ann. of Math., 129 (1): 71--111 (1989) 

\bibitem{Ginzburg} V. L. Ginzburg, \textit{The Conley conjecture}, Ann. of Math. (2) 172, no. 2, 1127--1180 (2010)

\bibitem{GHHM} V. L. Ginzburg, D. Hein, U. Hryniewicz and L. Macarini, \textit{Closed Reeb orbits on the sphere and symplectically degenerate maxima}, Acta Math. Vietnam. 38, no. 1, 55--78 (2013)


\bibitem{Hadamard} J. Hadamard, J., \textit{Les surfaces \`a courbures opposes et leur lignes g\'eodesiques}, J Math. Pures Appl. (5), 4, 27-73 (1898)

\bibitem{Hingston} N. Hingston, \textit{Subharmonic solutions of Hamiltonian equations on tori}, Ann. of Math. (2) 170, no. 2, 529--560 (2009)

\bibitem{Hofer93} H. Hofer, \textit{Pseudoholomorphic curves in symplectizations with applications to the Weinstein conjecture in dimension three.} Invent. Math. 114, 515--563 (1993)

\bibitem{props1} H. Hofer, K. Wysocki and E. Zehnder, \textit{Properties of pseudo-holomorphic curves in symplectisations. I. Asymptotics.} Ann. Inst. H. Poincar\'e C Anal. Non Lin\'eaire 13, 337--379 (1996)

\bibitem{props2} H. Hofer, K. Wysocki and E. Zehnder, \textit{Properties of pseudo-holomorphic curves in symplectisations. II. Embedding controls and algebraic invariants.} Geom. Funct. Anal. 5, 270--328 (1995)

\bibitem{props3} H. Hofer, K. Wysocki and E. Zehnder, \textit{Properties of pseudoholomorphic curves in symplectizations. III. Fredholm theory.} In: Topics in Nonlinear Analysis, Progr. Nonlinear Differential Equations Appl. 35, Birkhäuser, Basel, 381--475 (1999)

\bibitem{convex} H. Hofer, K. Wysocki and E. Zehnder, \textit{The dynamics on three-dimensional strictly convex energy surfaces.} Ann. of Math. (2) 148, no. 1, 197--289 (1998)

\bibitem{fols} H. Hofer, K. Wysocki and E. Zehnder, \textit{Finite energy foliations of tight three-spheres and Hamiltonian dynamics.} Ann. of Math. (2) 157, no. 1, 125–-255 (2003)

\bibitem{fast} U. Hryniewicz, \textit{Fast finite-energy planes in symplectizations and applications.} Trans. Amer. Math. Soc. 364, 1859--1931 (2012)

\bibitem{HS13} U. Hryniewicz and P. A. S. Salom\~ao, {\em Global properties of tight Reeb flows with applications to Finsler geodesic flows on $S^2$\/}. Math. Proc. Cambridge Phil. Soc. {\bf 154} (2013), 1--27.

\bibitem{elliptic} U. Hryniewicz and P. A. S. Salom\~ao, \textit{Elliptic bindings for dynamically convex Reeb flows on the real projective three-space.} Calc. Var. Partial Differential Equations 55, art. 43, 57 pp. (2016)

\bibitem{HSW} U. Hryniewicz, P. A. S. Salom\~ao, and K. Wysocki, \textit{Genus zero global surfaces of section for Reeb flows and a result of Birkhoff.} J. Eur. Math. Soc., (9) (2023), 3365–3451.




\bibitem{HutJEMS} M. Hutchings, \textit{An index inequality for embedded pseudoholomorphic curves in symplectizations.} J. Eur. Math. Soc. (JEMS) 4, no. 4, 313--361 (2002)

\bibitem{ir} M. Hutchings, {\em The embedded contact homology index revisited\/}, in New Perspectives and Challenges in Symplectic Field Theory, CRM Proc. Lecture Notes, vol. {\bf 49} (2009), 263--296.

\bibitem{tw} M. Hutchings, {\em Taubes's proof of the Weinstein conjecture in dimension three\/}, Bull. AMS {\bf 47} (2010), 73--125.

\bibitem{qech} M. Hutchings, {\em Quantitative embedded contact homology\/}, J. Diff. Geom. {\bf 88} (2011), 231--266.

\bibitem{ECHlecture} M. Hutchings, \textit{Lecture notes on embedded contact homology,} Contact and symplectic topology, 389--484, Bolyai Soc. Math. Stud., 26, Budapest (2014)

\bibitem{HTJSG1} M. Hutchings and C. H. Taubes, \textit{Gluing pseudoholomorphic curves along branched covered cylinders. I.} J. Symplectic Geom. {\bf 5\/} (2007), 43--137.

\bibitem{HTJSG2} M. Hutchings and C. H. Taubes, \textit{Gluing pseudoholomorphic curves along branched covered cylinders. II.} J. Symplectic Geom. {\bf 7\/} (2009), 29--133.

\bibitem{HT} M. Hutchings and C.H. Taubes, \textit{The Weinstein conjecture for stable Hamiltonian structures.} Geom. Topol. {\bf 13\/} (2009), 901--941.

\bibitem{cc2} M. Hutchings and C.H. Taubes, {\em Proof of the Arnold chord conjecture in three dimensions II\/}, Geom. Topol. {\bf 17} (2013), 2601--2688.

\bibitem{irie} K. Irie, {\em Dense existence of periodic orbits and ECH spectral invariants\/}, J. Mod. Dyn. {\bf 9} (2015), 357--363.

\bibitem{Katok} A. Katok, \textit{Ergodic perturbations of degenerate integrable Hamiltonian systems}, Izv. Akad. Nauk SSSR Ser. Mat. 37, 539--576 (1973)

\bibitem{KMbook} P. Kronheimer and T. Mrowka, \textit{Monopoles and three-manifolds.} New Mathematical Monographs, 10. Cambridge University Press, Cambridge (2007)

\bibitem{long} Y. Long, {\em Multiplicity and stability of closed geodesics on Finsler $2$-spheres,} J. Eur. Math. Soc. {\bf 8} (2006), 341 - 353.

\bibitem{lw} Y. Long and W. Wang, {\em Stability of closed geodesics on Finsler $2$-spheres,} J. Func. Anal. {\bf 255} (2008), 620- 641.

\bibitem{ls} L. Lusternik and L. Schnirelmann, {\em Sur le probl\'eme de trois g'eod'esiques ferm'ees sur les surfaces de genre 0\/}, C. R. Acad. Sci. Paris {\bf 189} (1929), 269--271.

\bibitem{Rabinowitz} P. Rabinowitz, {\em Periodic solutions of Hamiltonian systems,} Comm. Pure Appl. Math. {\bf 31} (1978), no. 2, 157--184.

\bibitem{SiefringCPAM} R. Siefring, \textit{Relative asymptotic behavior of pseudoholomorphic half-cylinders.} Comm.Pure Appl. Math. 61, 1631--1684 (2008)

\bibitem{SiefringMathAnn} R. Siefring, \textit{Finite-energy pseudoholomorphic planes with multiple asymptotic limits.} Math. Ann. 368, no. 1-2, 367--390 (2017)

\bibitem{taubes_currents} C. H. Taubes, {\em The structure of pseudoholomorphic subvarieties for a degenerate almost complex structure and symplectic form on $S^1\times B^3$}, Geom. Topol. {\bf 2} (1998), 221--332.

\bibitem{Taubes_Weinstein_Conjecture} C. H. Taubes, \textit{The Seiberg-Witten equations and the Weinstein conjecture}, Geom. Topol. {\bf 11\/} (2007), 2117--2202.

\bibitem{Taubes-I} C. H. Taubes, \textit{Embedded contact homology and Seiberg–Witten Floer cohomology I.} Geom. Topol. {\bf 14\/} (2010), 2497--2581.

\bibitem{Taubes-V} C. H. Taubes, \textit{Embedded contact homology and Seiberg–Witten Floer cohomology V.} Geom. Topol. {\bf 14\/} (2010), 2961--3000.

\bibitem{weinstein} A. Weinstein, {\em On the hypothesis of Rabinowitz' periodic orbit theorems\/}, J. Diff. Equations {\bf 33} (1979), 353--358.



\end{thebibliography}
\end{document}